\documentclass[11pt]{aims}
\usepackage{amsmath, amssymb, mathrsfs}
  \usepackage{paralist}
  \usepackage{graphics, color} 
  \usepackage{epsfig} 
\usepackage{graphicx}  \usepackage{epstopdf}
 \usepackage[colorlinks=true]{hyperref}
 \hypersetup{urlcolor=blue, citecolor=red}
 \usepackage[toc,page]{appendix}
\usepackage{chngcntr}
\usepackage{bbm}
\usepackage{xcolor}
\usepackage{listings}
\usepackage[left=2cm,bottom=2cm,top = 2cm, right =2cm]{geometry}

\usepackage{titlesec}
\titleformat{\section}{\Large\bfseries}{\thesection.}{4pt}{}
\titleformat{\subsection}{\large\bfseries}{\thesection.\arabic{subsection}.}{4pt}{}
\titleformat{\subsubsection}{\bfseries}{\thesection.\arabic{subsection}.\arabic{subsubsection}.}{4pt}{}
\titleformat*{\paragraph}{\bfseries}
\titleformat*{\subparagraph}{\bfseries}
     \def\DOI{}

\newtheorem{theorem}{Theorem}[section]
\newtheorem{corollary}[theorem]{Corollary}

\newtheorem{lemma}[theorem]{Lemma}
\newtheorem{proposition}[theorem]{Proposition}
\theoremstyle{definition}

\newtheorem{remark}[theorem]{Remark}

\newcommand{\Rb}{\mathbb{R}}

\newcommand{\Cc}{\mathcal{C}}

\newcommand{\Ec}{\mathcal{E}}

\newcommand{\Gc}{\mathcal{G}}

\newcommand{\Ic}{\mathcal{I}}

\newcommand{\Oc}{\mathcal{O}}

\newcommand{\Rc}{\mathcal{R}}

\newcommand{\Kc}{\mathcal{K}}

\newcommand{\Zc}{\mathcal{Z}}

\newcommand{\bee}{\begin{eqnarray*}}
\newcommand{\eee}{\end{eqnarray*}}

\newcommand{\bea}{\begin{eqnarray*}}
\newcommand{\eea}{\end{eqnarray*}}
\newcommand{\inn}{\textup{in}}
\newcommand{\out}{\textup{ex}}

\newcommand{\Ls}{\mathscr{L}}
\newcommand{\Ms}{\mathscr{M}}
\newcommand{\Hs}{\mathscr{H}}
\newcommand{\As}{\mathscr{A}}

\def\fref#1{{\rm (\ref{#1})}}

\newcommand{\rj}{\left< r \right>}
\newcommand{\zj}{\left< z \right>}

\newcommand{\pa}{\partial}
\newcommand{\be}{\begin{equation}}
\newcommand{\ee}{\end{equation}}
\newcommand{\ba}{\begin{array}}
\newcommand{\ea}{\end{array}}
\newcommand{\la}{\langle}
\newcommand{\ra}{\rangle}
\newcommand{\e}{\varepsilon}

\newcommand{\tnu}{\tilde \nu}

\numberwithin{equation}{section}

\title[Spectral analysis of a linearized operator for the 2D Keller-Segel system] 
      {Spectral analysis for singularity formation of the two dimensional Keller-Segel system}

 \keywords{Keller-Segel system, Blowup solution, Blowup profile, Stability, Construction, spectral analysis}
\author[C. Collot]{Charles Collot}
\address{Courant Institute of Mathematical Sciences, New York University, 251 Mercer Street, New York, NY 10003, United States of America.}
\email{cc5786@nyu.edu}
\author[T. Ghoul]{Tej-Eddine Ghoul}
\address{Department of Mathematics, New York University in Abu Dhabi, Saadiyat Island, P.O. Box 129188, Abu Dhabi, United Arab Emirates.}
\email{teg6@nyu.edu}
\author[N. Masmoudi]{Nader Masmoudi}
\address{Department of Mathematics, New York University in Abu Dhabi, Saadiyat Island, P.O. Box 129188, Abu Dhabi, United Arab Emirates.}
\email{masmoudi@cims.nyu.edu}
\author[V. T. Nguyen]{Van Tien Nguyen}
\address{Department of Mathematics, New York University in Abu Dhabi, Saadiyat Island, P.O. Box 129188, Abu Dhabi, United Arab Emirates.}
\email{Tien.Nguyen@nyu.edu}

\thanks{\today}
\begin{document}
\maketitle


\medskip

\begin{abstract} 

We analyse an operator arising in the description of singular solutions to the two-dimensional Keller-Segel problem. It corresponds to the linearised operator in parabolic self-similar variables, close to a concentrated stationary state. This is a two-scale problem, with a vanishing thin transition zone near the origin. Via rigorous matched asymptotic expansions, we describe the eigenvalues and eigenfunctions precisely. We also show a stability result with respect to suitable perturbations, as well as a coercivity estimate for the non-radial part. These results are used as key arguments in a new rigorous proof of the existence and refined description of singular solutions for the Keller-Segel problem by the authors \cite{CGNNarx19b}. The present paper extends the result by Dejak, Lushnikov, Yu, Ovchinnikov and Sigal \cite{DphysD12}. Two major difficulties arise in the analysis: this is a singular limit problem, and a degeneracy causes corrections not being polynomial but logarithmic with respect to the main parameter.

\end{abstract}

\section{Introduction} 
We describe in this paper a detailed spectral analysis for the linear operator
\begin{equation} \label{def:Lz}
\Ls^z f = \Delta f - \nabla \cdot(f\nabla \Phi_{U_\nu} + U_\nu\nabla \Phi_f) - \beta \nabla \cdot(z f)
\end{equation}
in the radial and non-radial settings, where
$$\Phi_f = -\frac{1}{2\pi} \log|z|\ast f, \quad  U_\nu(z) = \frac{8\nu^2}{(\nu^2 + |z|^2)^2}, \quad \nabla \Phi_{U_\nu}(z) = -\frac{4 z}{\nu^2 + |z|^2},$$
$\beta > 0$ is a fixed constant and $0 < \nu \ll 1$ is the main parameter of the problem.   
\subsection{Origin of the spectral problem}
The linear operator $\Ls^z$ appears in the study of singularities of the following two dimensional parabolic-elliptic Keller-Segel system:
\begin{equation}\label{eq:KS}
\arraycolsep=1.4pt\def\arraystretch{1.6}
\left\{ \begin{array}{ll}
\partial_t u = \Delta u -\nabla \cdot \big( u \nabla \Phi_u\big),\\
\Phi_u  = -\frac{1}{2\pi} \log |x| \ast u,
\end{array}
\right. \quad (x,t)\in \Rb^2\times [0,T),
\end{equation}
see \cite{KSjtb70},  \cite{KSjtb71a}, \cite{KSjtb71b}, \cite{PATbmb53}, and \cite{Hjdmv03} for a survey of the problem. It is well known (see for example, \cite{JLtams92}, \cite{CPmb81}, \cite{Clec84}, \cite{DNRjde98}, \cite{BDPjde06}, \cite{BCCjfa12} and references therein) that the problem \eqref{eq:KS} exhibits finite time blowup solutions if the initial datum satisfies $u_0\geq 0$, some localisation assumptions and
$$M = \int_{\Rb^2} u_0 (x) dx > 8\pi. $$
The threshold $8\pi$ is related to the family of stationary solutions $(U_\eta)_{\eta>0}$ of \eqref{eq:KS}, where
\begin{equation} \label{def:U}
U_\eta(x) = \frac{1}{\eta^2}U\big(\frac{x}{\eta}\big) \quad \textup{with} \quad  U(x) = \frac{8}{(1 + |x|^2)^2} \quad \textup{and} \quad \int_{\Rb^2}U_\eta(x)dx = 8\pi.
\end{equation}
The parameter $\eta$ is linked to the scaling symmetry of the problem: if $u$ is a solution to \fref{eq:KS}, then for any $\eta>0$, $u_\eta$ defined by
\be \label{id:scaling}
u_\eta(x,t) = \frac{1}{\eta^2}u\left(\frac{x}{\eta}, \frac{t}{\eta^2}\right)
\ee
is a solution as well. As the mass $M$ which is a conserved quantity for \fref{eq:KS} is invariant under the above transformation, the problem is called critical. A key problem in understanding singular solutions is to analyse their asymptotic self-similarity. If a solution to \fref{eq:KS} is both singular at $t=0$ and invariant under the transformation \eqref{id:scaling}, it would be of the form $(-t)^{-1}w(x/\sqrt{-t})$; non-degenerate self-similarity would then refer to blowup solutions satisfying $u\sim (T-t)^{-1}w(x/\sqrt{T-t})$. However, one of the remarkable facts about finite time blowup solutions of \eqref{eq:KS} is that they present a degenerate self-similarity. Precisely, they are of type II blowup (see Theorem 8.19 in \cite{SSbook11} and Theorem 10 in \cite{NScm08} for such a statement) in the following sense. A solution $u(t)$ of \eqref{eq:KS} exhibits type I blowup at $t = T$ if there exists a constant $C > 0$ such that
\begin{equation}\label{def:TypeI}
\limsup_{t\to T}(T-t)\|u(t)\|_{L^\infty(\Rb^2)} \leq C,
\end{equation}
otherwise, the blowup is of type II. Equivalently, in the parabolic self-similar variables
\begin{equation} \label{eq:parabolicrenorm}
u(x,t) = \frac{1}{\mu^2}w(z,\tau), \quad \Phi_u(x,t) = \Phi_w(y,s), \quad z = \frac{x}{\mu}, \quad \frac{d\tau}{dt} = \frac{1}{\mu^2}, \quad \mu(t) = \sqrt{T-t},
\end{equation}
where $w(z,\tau)$ solves the equation 
\begin{equation}\label{eq:wztau}
\pa_\tau w = \nabla \cdot\big( \nabla w - w \nabla \Phi_w \big) - \beta \nabla\cdot(z w) \quad \textup{with} \quad \beta = -\frac{\mu_\tau}{\mu} = \frac{1}{2},
\end{equation}
$u$ is a type II finite time blowup solution of \eqref{eq:KS} if and only if $w$ is a global but unbounded solution of \eqref{eq:wztau}. The mechanism of singularity formation then involves crucially the above family of solutions $U_\eta$, see for example, \cite{HV1,Vsiam02, Vsiam04b, Vsiam04a,dejak2014blowup,dyachenko2013logarithmic,RSma14} and references therein. The key idea is that in equation \fref{eq:KS} the time variation $\partial_t u$ is asymptotically of lower order, the solution approaches the family of stationary states $u\sim U_{\eta}$ and a scaling instability drives the parameter $\eta$ to $0$  as $t\rightarrow T$. This motivates the study of a solution in the variables \fref{eq:parabolicrenorm} having the form 
$$w(z,\tau) =  U_\nu(z) + \e(z,\tau),$$
where $\nu=\eta/\sqrt{T-t}$ is time dependent with $\nu(\tau)\rightarrow 0$ as $\tau \rightarrow \infty$, and $\e$ is a lower order perturbation solving:
\begin{equation}\label{eq:eps}
\pa_\tau \e = \Ls^z \e - \nabla \cdot(\e \nabla \Phi_\e) +\left( \frac{\nu_\tau}{\nu}  - \beta\right) \nabla \cdot(zU_\nu).
\end{equation} 
Above, $\Ls^z$ is precisely the operator introduced in \fref{def:Lz} that we aim at studying in the present paper. The importance of its study is motivated by the following.The first rigorous construction of a blow-up solution by Herrero and Vel\'azquez for \fref{eq:KS}, \cite{HV1}, does not provide its stability, which is formally obtained in \cite{Vsiam02}. The work \cite{Vsiam02} shows linear stability in the inner zone $|z|\sim \nu$ if the scaling term is neglected, and gives an expansion for $\pa_\tau \e = \Ls^z \e $ in the parabolic zone $|z|\sim 1$ via formal series and matched asymptotics. A rigorous radial stability result is given by Rapha\"el and Schweyer in \cite{RSma14} in which the solution is studied in blow up variables $y = \frac{z}{\nu} \sim 1$ where a refined description is obtained. The description involves parameters, and their evolution (the modulation laws) is computed based on so called tail-dynamics, relying on suitable cancellations in the parabolic zone $|z|\sim 1$. The analysis of the tail-dynamics is however heavy, as it does not involve a refined understanding of the solution in self-similar variables. Our precise spectral study for the operator \fref{eq:eps}, however, gives a framework to control the solution accurately, \emph{on both scales simultaneously}, and the temporal evolution of the parameters is easily related to the projection of the dynamics on its eigenmodes. The present paper is a key result in this new approach to the construction of singular solutions to \fref{eq:KS} that is implemented in \cite{CGNNarx19b}, and allows to obtain a refined description (see Remark \ref{rk:blowup}).\\

\noindent It is remarkable that in the radial setting, the nonlocal operator $\Ls^z$ reduces to a local one in terms of the partial mass  
\begin{equation}
m_f(\zeta) = \frac{1}{2\pi} \int_{B(0,\zeta)} f(z) z dz, \quad \zeta = |z|,
\end{equation}
where $B(0,\zeta)$ the ball centered at $0$ of radius $\zeta$. Indeed, if $f$ is spherically symmetric, then we have the relation
$$
\Ls^z f(\zeta) = \frac{1}{\zeta}\pa_\zeta \Big( \As^\zeta m_f(\zeta) \Big),
$$
where $\As^\zeta$ is the linear operator defined by
\begin{equation}\label{def:AszetaAs0}
\As^\zeta = \As^\zeta_0 - \beta \zeta\pa_\zeta \quad \textup{with} \quad \As_0^\zeta = \pa_\zeta^2 - \frac{1}{\zeta} \pa_\zeta +\frac{Q_\nu }{\zeta}\pa_\zeta  + \frac{ \pa_\zeta(Q_\nu )}{\zeta} \quad \textup{and} \quad Q_\nu(\zeta) = \frac{4 \zeta^2}{\zeta^2 + \nu^2}.
\end{equation}
Hence, in the radial setting $\Ls^z$ and $\As^\zeta$ share the same spectrum and if $\varphi$ and $\phi$ are the radial eigenfunctions of $\Ls^z$ and $\As^\zeta$ respectively, we have the relation 
$$\varphi(\zeta) = \frac{\pa_\zeta \phi(\zeta)}{\zeta}, \quad \Phi_\varphi(\zeta) = -\frac{\phi(\zeta)}{\zeta}.$$ 
Therefore, we are interested in the eigenproblem
\begin{equation}\label{eq:eiPb}
\As^\zeta\phi(\zeta) = \lambda \phi(\zeta), \quad \zeta \in [0, \infty),
\end{equation}
in the regime
\begin{equation}\label{rangeOfnu}
\beta\sim 1, \quad  0 < \nu \ll 1.
\end{equation}
Note that the constant $\beta$ is not necessarily close to 1, it can be any fixed positive constant. 

\subsection{Main results}

Our first result concerns the spectrum of $\Ls^z$ in the radial setting. Its analysis has been done by Dejak, Lushnikov, Yu, Ovchinnikov and Sigal \cite{DphysD12} via matched asymptotic expansions. Our approach, similar in spirit to \cite{DphysD12}, is inspired by the work of Collot, Merle, and Rapha\"el \cite{CMRarx17}  for the study of type II supercritical singularities of the semilinear heat equation $u_t=\Delta u +|u|^{p-1}u$ (see also \cite{CRSmams19,HRjems19,meraszaniso} for related problems). The strategy is to construct suitable eigenfunctions near the origin and away from the origin, and to match them rigorously to produce a full eigenfunction. Differentiating the matching condition then provides information on the dependence of the eigenfunctions on the parameters. The current work extends this approach to a critical problem, showing its robustness. Solving \eqref{eq:eiPb}, though, is not just a mere adaptation the techniques of \cite{CMRarx17} because of the following points.\\

\noindent This critical case displays two new degeneracies. First, this is a singular limit problem. Indeed, from the explicit formula \fref{def:AszetaAs0} for $Q_\nu$, we note that the operator $\As^\zeta$ converges to a limit operator pointwise outside the origin, namely that for any smooth function $f$ and at any fixed $\zeta>0$, we have
$$
\As^\zeta f(\zeta)\rightarrow \pa_\zeta^2 f(\zeta) + \frac{3}{\zeta} \pa_\zeta f(\zeta) -\beta \zeta \partial_\zeta f(\zeta) \quad \mbox{ as } \nu \rightarrow 0.
$$
The limit operator $ \pa_\zeta^2 + 3/\zeta \pa_\zeta  -\beta \zeta \partial_\zeta$ is well understood, its spectrum is $\{ 0,-2\beta,-4\beta,-6\beta,...\}$ and its eigenfunctions are Hermite polynomials. However, the limit $\nu \rightarrow 0$ for the problem \fref{eq:eiPb} is a singular one. The problem involves two scales: one is $\zeta \sim 1$ and the other is $\zeta \sim \nu$. What happens at the latter actually prevents the convergence to the aforementioned limit operator: the spectrum is shifted by the constant $2\beta$ at the leading order as is shown in Proposition \ref{prop:SpecRad} below. This in particular prevents the use of a bifurcation argument. Second, this problem also presents another degeneracy from which most of the technical difficulty stems from, since next order corrections, instead of being polynomial in the parameter $\nu$, are actually polynomial in $1/|\log \nu|$. We then need to refine to higher order the description of both the inner solution at $\zeta \sim \nu$ and the outer solution at $\zeta \sim 1$.\\

\noindent We provide a precise description of the eigenfunctions, relating them to the iterated kernel $(T_i)_{i\in \mathbb N}$ of $\As_0$, the linearised operator near the stationary state, a rescaled version of $\As_0^\zeta$ via the change $\zeta = \nu r$, i.e.
\begin{equation}\label{def:As}
\As_0 = \pa_r^2 - \frac{1}{r}\pa_r + \frac{\pa_r (Q \cdot)}{r} \quad \textup{with} \quad Q(r) = \frac{4r^2}{1 + r^2},
\end{equation}
defined by:
$$T_{j+1}(r) = - \As_0^{-1} T_j(r) \quad \textup{with} \quad T_0(r) = \frac{r^2}{(1 + r^2)^2}=\frac 18 r\pa_r Q, \;\; \As_0 T_0 = 0.$$
As $r\pa_r Q$ is the direction of scaling instability for the stationary state, this description allows to understand the rescaled blow-up dynamics \fref{eq:eps}. In addition, the properties of $T_j$ can be explicitly computed, such as its asymptotic behavior (see Lemma \ref{lemm:proAs})
$$T_j(r) \sim \hat d_j r^{2j - 2} \ln r, \quad \mbox{as } r\rightarrow \infty, \quad \hat d_j\neq 0 \mbox{ a constant.} $$

\noindent To state our results, we use the notation $A \lesssim B$ to say that there exists a constant $C > 0$ such that $ 0 \leq A \leq CB$. Similarly, $A \sim  B$ means that there exist constants $0 < c < C$ such that $cA \leq B \leq CA$. We write $\rj = \sqrt{1 + r^2}$, and use the notation $D_\zeta^k$ for $k\in \mathbb N$ for $k$-th adapted derivative with respect to $\zeta$ defined by
$$
D_\zeta^{2k}=\left(\zeta \pa_\zeta (\frac{\pa_\zeta}{\zeta}) \right)^{2k}, \ \ D_\zeta^{2k+1}=\pa_\zeta D^{2k},
$$
and the weight functions
\be \label{omega}
\omega_\nu(\zeta) = \frac{\nu^2 }{U_\nu(\zeta)} e^{-\frac{\beta \zeta^2}{2}}=\frac{\nu^2}{U_\nu(\zeta)}\rho_0(\zeta), \quad \rho_0 (\zeta)=e^{-\frac{\beta \zeta^2}{2}},
\ee
with corresponding weighted Lebesgue space $L^2_{\frac{\omega_\nu}{\zeta}}$ and Sobolev spaces $H^k_{\frac{\omega_\nu}{\zeta}}=\{ f: \; \sum \limits_{0}^k \| D^k f\|_{L^2_{\frac{\omega_\nu}{\zeta}}}<\infty\}$. Our first main result is to describe in details spectral properties of $\As^\zeta$ in the regime \eqref{rangeOfnu}. 

\begin{proposition}[Spectral properties of $\As^\zeta$] \label{prop:SpecRad} The linear operator $\As^\zeta: H^2_{\frac{\omega_\nu}{\zeta}} \rightarrow L^2_{\frac{\omega_\nu}{\zeta}}$ is essentially self-adjoint with compact resolvant. Moreover, given any $N\in \mathbb N$, $0<\beta_*< \beta^*$ and $0<\delta \ll 1$, there exists a $\nu^*>0$ such that the following holds for all $0<\nu \leq \nu^*$ and $\beta_*\leq \beta \leq \beta^*$.\\
\noindent $(i)$ \textup{(Eigenvalues)} The first $N + 1$ eigenvalues are given by
\begin{equation}\label{def:specAsbmain}
\lambda_{n, \nu} = 2\beta \Big(1 - n   + \tilde{\alpha}_{n, \nu}\Big), \quad \mbox{for }n = 0, 1, \cdots, N,
\end{equation}
where
\begin{equation}\label{est:nu0ntilmain}
\tilde{\alpha}_{n, \nu}=\frac{1}{2\ln \nu}+\bar \alpha_{n, \nu} \quad  \mbox{with } \quad  | \bar{\alpha}_{n, \nu}| + \big|\nu\pa_\nu \tilde \alpha_{n, \nu}\big| \lesssim \frac{1}{|\ln \nu|^{2}}.
\end{equation}
In particular, we have the refinement of the first two eigenvalues with $\gamma$ the Euler's constant:
$$\left| \tilde \alpha_{n,\nu}- \frac{1}{2\ln \nu} -  \frac{\ln 2 - \gamma - n - \ln \beta}{4|\ln \nu|^2} \right| \lesssim \frac{1}{|\ln \nu|^3}, \quad \mbox{for }n=0,1.$$

\noindent $(ii)$ \textup{(Eigenfunctions)} There exist eigenfunctions $\phi_{n, \nu}$ satisfying the following. There holds the pointwise estimates for $ k = 0,1,2$:
\begin{align}
\left|D_\zeta^k \phi_{n, \nu}(\zeta)\right|+\left|D_\zeta^k \beta\partial_\beta \phi_{n, \nu}(\zeta) \right|& +\left|D_\zeta^k \nu\partial_\nu \phi_{n, \nu}(\zeta) \right| \nonumber \\
  & \qquad \lesssim \left(\frac{\zeta}{\nu+\zeta}\right)^{2-(k\; \textup{mod}\; 2)}\frac{ \langle\zeta\rangle^{2n +\delta} \big(1 + \ln \big\la\frac{\zeta}{\nu} \big\ra\, \mathbf{1}_{\{n \geq 1\}} \big)}{(\zeta + \nu)^{2 + k}}. \label{est:PhinPointEstmain}
\end{align}
There holds in addition the refined identity:
$$
 \phi_{n,\nu}(\zeta) = \sum_{j = 0}^n c_{n,j}\beta^{j}\nu^{2j-2} T_j\big(\frac{\zeta}{\nu}\big) + \tilde{\phi}_{n,\nu}(\zeta),
$$
where the profiles $T_j$ and the constants $c_{n,j}$ are defined in Lemma \ref{lemm:proAs}, with for $k=0,1,2$:
$$
\left|D_\zeta^k \tilde \phi_{n, \nu}(\zeta)\right| +\left|D_\zeta^k \nu\partial_\nu \tilde \phi_{n, \nu}(\zeta) \right| +\left|D_\zeta^k \beta\partial_\beta \tilde \phi_{n, \nu}(\zeta) \right| \lesssim  \min \left(\nu^2\langle \frac{\zeta}{\nu}\rangle ^2, \frac{1}{|\ln \nu|}  \right) \left(\frac{\zeta}{\nu+\zeta}\right)^{2-(k\; \textup{mod}\; 2)}\frac{ \langle\zeta\rangle^{2n +\delta}}{(\zeta + \nu)^{2 + k}}.
$$

There holds the $L^2_{\frac{\omega_\nu}{\zeta}}$ estimates for all $0\leq m,n\leq N$:
\begin{equation}\label{est:PhinL2norm1main}
 \langle \phi_{n, \nu}, \phi_{m, \nu} \rangle^2_{L^2_{\frac{\omega_\nu}{\zeta}}} = c_{n}\delta_{m,n},  \quad c_0\sim \frac{|\ln \nu|}{8}, \quad c_1\sim \frac{|\ln \nu|^2}{4}, \quad c|\ln \nu|^2\leq c_n \leq  \frac 1c |\ln \nu|^2 \mbox{ for }n\geq 2,
\end{equation}
where $c$ is some positive constant.

\noindent $(iii)$ \textup{(Spectral gap estimate)} For any $g \in L^2_{\frac{\omega_\nu}{\zeta}}$ with $\langle g, \phi_{j, \nu}\rangle_{L^2_{\frac{\omega_\nu}{\zeta}}} = 0$ for $0 \leq j \leq N$, one has
\begin{equation}\label{est:SpecGap1main}
\big \la g, \As^\zeta g \big \ra_{L^2_\frac{\omega_\nu}{\zeta}}   \leq  \lambda_{N+1, \nu}\|g \|^2_{L^2_\frac{\omega_\nu}{\zeta}}.
\end{equation}
\end{proposition}

\begin{remark}
We recover the same eigenvalues as \cite{DphysD12}. Though our proof is similar to \cite{DphysD12} since relying on matched asymptotics, we here adopt the approach of \cite{CMRarx17}, yielding detailed information on the eigenfunctions and on the variations with respect to the parameter $\nu$. We also mention that the matching procedure performed in \cite{DphysD12} was formal as the analysis did not involve the matching of derivatives. To match the derivatives, we found a degeneracy that forces us to expand both inner and outer solutions to the next order, which renders the analysis much more involved.
\end{remark}

\begin{remark} \label{rk:blowup} Based on Proposition \ref{prop:SpecRad}, we are able to construct for the problem \eqref{eq:KS} in \cite{CGNNarx19b} finite time blowup solutions according to the following precise asymptotic dynamics as $t \to T$:
$$u(x,t) \sim \frac{1}{\eta^2(t)}U \big(\frac{x}{\eta(t)}\Big),$$
where the blowup rate is given by either 
$$\eta(t) = 2e^{-\frac{\gamma + 2}{2}} \sqrt{T-t} e^{-\sqrt{\frac{|\ln (T-t)|}{2}}}(1+o_{t\uparrow T}(1)),$$
or 
$$ \eta(t) \sim C(u_0)(T-t)^\frac{\ell}{2} |\ln (T-t)|^{-\frac{\ell + 1}{ 2(\ell - 1)}} \quad \textup{for some } \;\; \ell \in \mathbb{N}, \;\; \ell \geq 2.$$  
It is worth saying that the rigorous analysis performed in \cite{CGNNarx19b} is greatly simplified thanks to Proposition \ref{prop:SpecRad} in comparison with the one of \cite{RSma14}. Importantly, we believe that the precise description of the spectrum of $\As^\zeta$ is one of the crucial steps toward the classification of all possibilities of blowup speeds for \eqref{eq:KS} (at least in the radial setting) which is a challenging problem in the analysis of blowup.
\end{remark}

\begin{remark} 

The present result deals with the critical Keller-Segel system. We believe that other critical problems can be studied with this approach, such as the harmonic heat flow and the semilinear heat equation. Related spectral studies were performed in the case of non-degenerate self-similar singularities for wave type equations, see for example \cite{CDGcmp17}, \cite{CDXnonl16} for the study of stability of self-similar wave maps. It is an interesting direction to implement the present work to the hyperbolic setting.
\end{remark}

Our second result aims at understanding under what kind of perturbations is Proposition \ref{prop:SpecRad} stable. This is of a particular importance for the full nonlinear problem \fref{eq:KS} analysed in \cite{CGNNarx19b}, and shows the robustness of our approach. As a direct consequence of our construction, the spectral properties of $\As^\zeta$ stated in the previous proposition still hold true for the following perturbed operator of the form
\begin{equation} \label{id:modifiedop}
\bar \As^\zeta = \As^\zeta + \frac{1}{\zeta}\pa_\zeta \left(P \cdot \right),
\end{equation}
where the  perturbation $P$ satisfies
\begin{equation} \label{id:conditionP}
 |P(\zeta)| + |\zeta\pa_\zeta P(\zeta)| \lesssim \frac{\nu^2}{|\ln \nu|} \frac{\zeta^2 }{(\nu^2 + \zeta^2)^2}.
 \end{equation}

\begin{proposition} \label{pr:spectralbarAzeta}
Assume the bound \fref{id:conditionP} and the same hypotheses as in Proposition \ref{prop:SpecRad}. Then, the operator $\bar \As^\zeta: H^2_{\frac{\bar \omega_\nu}{\zeta}} \rightarrow L^2_{\frac{\bar \omega_\nu}{\zeta}}$ is essentially self-adjoint with compact resolvent, where
$$\bar \omega_\nu(\zeta)=\omega_\nu(\zeta) \exp \left( \int_0^\zeta \frac{P(\tilde \zeta)}{\tilde \zeta}d\tilde \zeta\right).$$
The first $N+1$ eigenvalues $\{\bar \lambda_{n, \nu}\}_{0\leq n\leq N}$ of $\bar \As^\zeta$ satisfy
\be \label{bd:stabilityalpha}
|\bar \lambda_{n, \nu}-\lambda_{n, \nu} |\leq \frac{C'}{|\log \nu|^2},
\ee
and there exist associated eigenfunctions $\{\bar \phi_{n,\nu}\}_{0 \leq n \leq N}$ satisfying
\be \label{bd:stabilityeigenmodes}
\frac{\| \bar \phi_{n,\nu}-\phi_{n,\nu} \|_{L^2(\frac{\omega_\nu}{\zeta})}}{\| \phi_{n,\nu} \|_{L^2(\frac{\omega_\nu}{\zeta})}}\leq \frac{C'}{\sqrt{|\log \nu|}}.
\ee

\end{proposition}

\begin{remark} Note that Proposition \ref{pr:spectralbarAzeta} is not a direct consequence of Proposition \ref{prop:SpecRad} in the sense that a standard perturbation argument does not work here. Indeed, the potential part $\pa_\zeta P/\zeta$ of the perturbation in \fref{id:modifiedop} is of size $\nu^{-2}$ in $L^{\infty}$ (up to a logarithmic accuracy), while the eigenvalues of the unperturbed operator $\As^\zeta$ are of order $1$. The crucial point is that the algebraic form of the perturbation, $\pa_\zeta (P\cdot)/\zeta$, ensures its orthogonality to the resonance of the operator $\As_0$ near the origin, see Lemma \ref{lem:perturbationinner} and its proof.

\end{remark}

\begin{remark} In \cite{CGNNarx19b}, the use of Proposition \ref{pr:spectralbarAzeta} is essential to handle nonlinear terms, where the precise control of the solution near the origin involves the rescaled stationary state at a slightly different scale $\tilde \nu$, and the corresponding perturbed linear operator is \fref{id:modifiedop} with the perturbation potential
$$
P(\zeta) = \frac{Q_{\tnu}(\zeta) - Q_\nu(\zeta)}{2}, \quad \left|\frac{\tnu}{\nu}-1\right|\lesssim \frac{1}{|\log \nu|},
$$
and the corresponding weight function $\bar \omega_\nu(\zeta) = \frac{\nu \tilde{\nu}}{\sqrt{U_\nu U_{\tnu}}} \rho(\zeta).$
\end{remark}

Our third and last result concerns the decay of the linearised dynamics associated to $\Ls^z$ for the nonradial part of the perturbation. The work \cite{Vsiam02} provides Lyapunov functionals for the inner zone $|z|\sim \nu$, and \cite{RSma14} uses a suitable extension to a higher regularity level of these Lyapunov functionals, inspired from \cite{RR}. Both results do not include the scaling term in the functionals, making the control in the parabolic region $|z|\sim 1$ hard. We prove here a coercivity estimate that takes the scaling term into account, for a modified version of the linearised operator, in which the source term for the Poisson field is localised near the origin. Note that an analogue of the radial spectral analysis of Proposition \ref{prop:SpecRad} is not straightforward. Indeed, while the operator $\nabla \Delta^{-1}$ is an integral operator from the origin in the radial case, the integral involve the behaviour of the function at infinity on higher order spherical harmonics, see \eqref{id:nablaPhi0} and \eqref{id:nablaPhiki}. In particular, it is not possible to make sense of $\nabla \Delta^{-1}$ for nonradial functions with strong polynomial growth at infinity.\\

\noindent On the one hand, at the $|z|\sim \nu$ scale, there is a natural scalar product for the linearised operator without scaling term, coming from the free energy. The following corresponds to \cite{RSma14}, Lemma 2.1 and Proposition 2.3. The linearized operator at scale $\nu$ is written as
\be
 \Ls_0 u =\Delta u-\nabla  \cdot (u \nabla \Phi_U)-\nabla \cdot (U\nabla  \Phi_u)= \nabla \cdot(U \nabla \Ms u) \quad \textup{with} \quad \Ms u = \frac{u}{U} - \Phi_u, \label{def:Lsform0}
\ee
The quadratic form $\int u\Ms vdy$ is symmetric. There hold the estimates if $\int udy=0$:
\begin{equation}\label{eq:contM}
\int_{\Rb^2} U |\Ms u|^2dy \lesssim \int_{\Rb^2} \frac{u^2}{U}dy,
\end{equation}
the nonnegativity $\int u\Ms u \geq0$ and, for some  $\delta_1,C>0$,
\be \label{bd:coercivite L2}
\int_{\Rb^2} u\Ms u dy \geq \delta_1\int_{\Rb^2} \frac{u^2}{U}dy - C\Big[ \la u, \Lambda U\ra_{L^2}^2  + \la u, \pa_1 U\ra_{L^2}^2 + \la u, \pa_2 U\ra_{L^2}^2 \Big],
\ee
where $\Lambda$ is the scaling group infinitesimal generator:
$$
\Lambda u = 2u+y.\nabla u.
$$

For functions orthogonal to $\Lambda U,\pa_{y_1}U,\pa_{y_2}U$ in the $L^2$ sense, the norms defined by $\int \frac{u^2}{U}dy$ and $\int u\Ms udy$ are then equivalent. On the other hand, at scale $|z|\sim 1$, from \fref{def:U} and as $\partial_\zeta \Phi_{U_\nu}=-4\zeta/(\nu^2+\zeta^2)$ we get that $\Ls^z$ converges pointwise to $\Delta+4/\zeta \pa_\zeta-\beta\nabla.(z\cdot)$ as $\nu \rightarrow 0$. This operator is self adjoint in $L^2(\zeta^{4}\rho_0)$. We thus introduce the "mixed" scalar product
\begin{equation}\label{def:quadform}
\langle u,v\rangle_\ast:= \nu^2\int_{\Rb^2} u\sqrt{\rho_0}\Ms^z (v \sqrt{\rho_0})dz, \quad \Ms^zf = \frac{f}{U_\nu} - \Phi_f.
\end{equation}
It matches to leading order the first scalar product at scale $\nu$ and the second at scale $1$, and localises the Poisson field. It is equivalent to the $L^2_{\omega_\nu}$ scalar product under the aforementioned orthogonality conditions. We localize the Poisson field in the linearized operator accordingly,
\be \label{def:tildeLz}
\tilde \Ls^z u :=\Delta u-\nabla  \cdot (u \nabla \Phi_{U_\nu})-\nabla \cdot (U_\nu\nabla \tilde \Phi_u)-\beta\nabla \cdot (zu), \quad  \tilde \Phi_u:= \frac{1}{\sqrt{\rho_0}}(-\Delta)^{-1}(\sqrt{\rho_0} u).
\ee
We show that in the non-radial sector, the localised operator $\tilde \Ls^z$ is coercive for the mixed scalar product $\langle \cdot ,\cdot \rangle_*$ under the natural orthogonality assumption to $\nabla U_\nu$. Its proof adapts the arguments of \cite{RSma14} for the above coercivity of $\Ls_0$ to the range $0<\nu \ll 1$.

\begin{proposition}\label{pr:coercivitenonradial}
For any $0<\beta_*<\beta^*$, there exists $c,C>0$ and $\nu^*>0$ such that for all $\beta_*<\beta<\beta^*$ and $0<\nu \leq \nu^*$, if $u\in \dot H^1_{\omega_{\nu}}$ satisfies $\int_{|z|=\zeta}udz=0$ for almost every $\zeta$, then:
\begin{equation}\label{est:coerLznonradial}
\langle -\tilde{\Ls}^z  u,u\rangle_\ast \geq c \| \nabla u \|_{L^2_{\omega_\nu}}^2 - C\nu^6 \left( \left(\int_{\Rb^2} u\partial_{z_1}U_\nu \sqrt{\rho_0}dz\right)^2+\left(\int_{\Rb^2} u\partial_{z_2}U_\nu\sqrt{\rho_0}dz\right)^2\right).
\end{equation}
\end{proposition}

\begin{remark}
The above Proposition holds for $\tilde \Ls^z$ instead of $\Ls^z$: a part of the Poisson field outside the origin has been neglected. However, in the singularity formation studied in \cite{CGNNarx19b} the worst contribution to this field from the perturbation comes from the origin, and the stationary states decays rapidly at infinity. The difference $\Ls^z-\tilde \Ls^z$ can then be controlled from other norms, see \cite{CGNNarx19b}.
\end{remark}

The paper is organised as follows. Section \ref{sec:lin} is devoted to the proof of Propositions \ref{prop:SpecRad} and \ref{pr:spectralbarAzeta}. The proof of Proposition \ref{pr:coercivitenonradial} is done in Section \ref{sec:CoerNonR}.\\

\paragraph{\textbf{Acknowledgement:}} C. Collot is supported by the ERC- 2014-CoG 646650 SingWave. N. Masmoudi is supported by NSF grant DMS-1716466. This work was  supported  by Tamkeen under the NYU Abu Dhabi Research Institute grant of the center SITE. Part of this work was done while C. Collot, T.-E. Ghoul and N. Masmoudi were visiting IH\'ES and they thank the institution.

\section{Proof of the spectral Propositions \ref{prop:SpecRad} and \ref{pr:spectralbarAzeta}} \label{sec:lin}
This section is devoted to the proof of Proposition \ref{prop:SpecRad}. After the change of variable $\zeta = \nu r$, the problem \eqref{eq:eiPb} is equivalent to the following
\begin{equation}\label{eq:phi}
(\As_0 - br\partial_r) \phi  = \alpha \phi, \quad r \in [0,\infty),
\end{equation}
where $\As_0$ is introduced in \eqref{def:As} and 
$$b = \beta \nu^2, \quad \alpha = \lambda \nu^2.$$
We will solve the problem \fref{eq:phi} in the regime $0 < b \ll 1$ by means of matched asymptotic expansions in the following sense. Let $\zeta_0$ and $R_0$ be fixed as
$$ 0 < \zeta_0 \ll 1, \quad  R_0 = \frac{\zeta_0}{\sqrt{b}} \gg 1.$$
Relying on perturbation theory, we first solve \fref{eq:phi} in the inner region $r \leq R_0$, and the solution is named by $\phi^{\inn}$, then in the outer region $r \geq R_0$ and the solution is named by $\phi^\out$. The two solutions must coincide at $r = R_0$ up to the first derivative from which we determine the value of $\alpha$ by standard arguments based on the implicit function theorem. Proposition \ref{prop:SpecRad} is a direct consequence of the following.


\begin{proposition}[Spectral properties of $\As = \As_0 - br\pa_r$] \label{prop:SpecRadp} The linear operator $\As$ is self-adjoint in $L^2(\Rb^+,\omega_b dr)$ with $\omega_b^{-1} =rU e^{\frac{br^2}{2}}$ with compact resolvent. Given any $N\in \mathbb N$, and $0<\delta \ll 1$, there exists a $b^*>0$ such that the following holds for all $0<b \leq b^*$:\\
\noindent $(i)$ \textup{(Eigenvalues)} We have that the first $N+1$ eigenvalues are given by:
\begin{equation}\label{def:specAsb}
\alpha_{n} = 2b\Big(1 - n + \tilde{\alpha}_{n}\Big), \qquad n=0,...,N,
\end{equation}
where:
\begin{equation}\label{est:nu0ntil}
\tilde{\alpha}_n=\frac{1}{\ln b}+\bar \alpha_n \quad  \mbox{with } \quad  | \bar{\alpha}_n| + \big|b\pa_b \tilde \alpha_n\big| \lesssim \frac{1}{|\ln b|^{2}}.
\end{equation}
In particular, we have a refinement of the first two eigenvalues, with $\gamma$ the Euler constant:
\begin{equation}\label{est:nurefinedn=01}
\left| \tilde \alpha_n - \frac{1}{\ln b} -  \frac{\ln 2 - \gamma - n}{|\ln b|^2} \right| \lesssim \frac{1}{|\ln b|^3}, \quad \mbox{for }n=0,1.
\end{equation}

$(ii)$ \textup{(Eigenfunctions)} An eigenfunction $\phi_{n}$ is defined by \fref{def:Mn} and the following properties hold:
\begin{itemize}
\item[-] \textup{(Sign-changing)} On the interval $[0,\infty)$, $\phi_0$ has constant sign, and $\phi_n$ vanishes exactly $n$ times for $n\geq 1$.
\item[-] \textup{(Orthogonality)} For some positive constants $\bar e_n$ there holds:
\begin{equation}\label{est:PhinL2norm}
\forall (m,n) \in \mathbb{N}^2, \quad \langle \phi_n, \phi_m \rangle^2_{L^2_{\omega_b}} = c_{n}\delta_{m,n},  \qquad c_{n} \sim \left\{\begin{array}{ll} 2^{-4} |\ln b|^{n + 1} & \; \textup{for}\; n = 0,1,\\
\bar e_n|\ln b|^2& \; \textup{for}\; n \geq 2.
\end{array} \right.
\end{equation}
\item[-] \textup{(Pointwise estimates)} For $ k = 0,1,2$,
\begin{align}\label{est:PhinPointEst}
\left| D_r^k\phi_{n}(r)\right| + \left|D_r^k b\partial_b \phi_{n}(r) \right| & \lesssim  \left(\frac{r}{\langle r \rangle}\right)^{2-(k\; \textup{mod}\; 2)} \langle r \rangle^{-2-k}\langle \sqrt b r \rangle^{2n+\delta}(1+\mathbf{1}_{\{n \geq 1\}} \ln \langle r\rangle).
\end{align}

\item[-] \textup{(Refined pointwise estimates)} There holds the refined identity:
\be \label{id:deftildephin}
 \phi_{n}(r) = \sum_{j = 0}^n c_{n,j}b^{j} T_j\big(r\big) + \tilde{\phi}_{n}(r),
\ee
where $T_j$ and $c_{n,j}$ are defined in Lemma \ref{lemm:proAs}, with for $k=0,1,2$:
\be \label{bd:refinedtildephin}
\left|D^k_r \tilde \phi_{n}(r)\right| +\left|D^k_r b\partial_b \tilde \phi_{n}(r) \right| \lesssim \min \left(b\langle r\rangle ^2, \frac{1}{|\ln b|}  \right) \left(\frac{r}{\langle r \rangle}\right)^{2-(k\; \textup{mod}\; 2)} \langle r \rangle^{-2-k}\langle \sqrt b r \rangle^{2n+\delta}
\ee

\end{itemize}
\noindent $(iii)$ \textup{(Spectral gap estimate)} For any $g \in L^2(\Rb^+,\omega_b dr)$ with $\langle g, \phi_j\rangle_{L^2_{\omega_b}} = 0$ for $0 \leq j \leq k$, one has
\begin{equation}\label{est:eigL0b}
\langle g, \As g\rangle_{L^2_{ \omega_b}}^2 \leq \alpha_{k+1}\|g\|^2_{L^2_{\omega_b}}.
\end{equation}
\end{proposition}

\begin{proof} Since the computation of $(\alpha_n, \phi_n)$ through the matched asymptotic procedure is long and technical, it is left to next subsections. In particular, the existence of the $N+1$ eigenvalues $\alpha_0,...,\alpha_N$ satisfying \fref{def:specAsb}, \fref{est:nu0ntil} and \fref{est:nurefinedn=01}, and the refined bound \fref{bd:refinedtildephin}, are proved in Lemma \ref{lemm:radialmode}. It then remains to use them to prove all the other results in Proposition \ref{prop:SpecRadp}.\\

\noindent 
\textbf{Step 1} \emph{Self-adjointness and discreteness of $\sigma(\As)$}: We deal with the discreteness of $\sigma(\As )$ and its uniqueness relying on classical arguments for second order linear operators. The first observation is that we can rewrite the linear operator $\As $ as
\begin{align*}
\As  f &= \partial_r^2f - \left(\frac{1}{r} + \Phi_U' + br\right)\partial_r f + Uf= \frac{1}{\omega_b}\partial_r \left(\omega_b \partial_r f \right) + Uf,
\end{align*}
which gives the self-adjointness of $\As $ in $L^2(\Rb^+, \omega_b dr)$. For the discreteness of its spectrum, we let
$$\rho_b(r) = \frac{1}{r^2\sqrt{U}}e^{-\frac{br^2}{4}},$$
and observe that
\begin{align*}
\tilde{\As} f &= \rho_b \As  \big(\rho_b^{-1}f\big) =\frac{\rho_b}{\omega_b}\partial_r \left(\omega_b \partial_r(\rho_b^{-1}f)\right) + Uf\\
& = \partial_r^2 f + \frac{\rho_b}{\omega_b}\left[\partial_r\left(\frac{\omega_b}{\rho_b}\right) + \omega_b\partial_r(\rho_b^{-1})\right]\partial_r f + \frac{\rho_b}{\omega_b}\partial_r \left(\omega_b\partial_r(\rho_b^{-1})\right)f + Uf\\
& = \partial_r^2 f + \frac{3}{r}\partial_r f + \left[-\frac{b^2r^2}{4} + \frac{2br^2}{1+r^2} \right] f + Uf.
\end{align*}
The linear operator $\tilde{\As}$ is of Schr\"odinger type and self-adjoint in $L^2(\Rb^+, r^3 dr)$. Since $\tilde{\As}$ has the real potential tending to infinity as $r \to \infty$, its spectrum is purely discrete by standard arguments. This concludes the discreteness of spectrum of  $\As $.\\

\noindent 
\textbf{Step 2} \emph{Uniqueness of the eigenvalues:}  We first prove the sign changing property of (ii). This is a direct consequence of lemmas \ref{lemm:inn0} and \ref{lemm:out0}. In particular, we show that on the interval $(0, \frac{\zeta_0}{\sqrt{b}}]$,  $\phi_0$ does not vanish and $\phi_n$ has exactly one zero located at $r_0 \sim \frac{1}{\sqrt{n b |\log b|}}$ for $n \geq 1$ (see page  \pageref{eq:expphi0n} for a detailed proof). On the interval $( \frac{\zeta_0}{\sqrt{b}}, +\infty)$, the eigenfunction $\phi_n$ is a perturbation of a Kummer's function (see Lemma \ref{lemm:out0} for a proper definition) where it does not vanish for $n = 0,1$ and possesses $n - 1$ zeros for $n \geq 2$.\\

We now rely on Sturm-Liouville theory to show that the $N+1$ first eigenvalues of $\As $ are those given by \fref{def:specAsb}. We argue by contradiction  and assume that there exists $\alpha^* \in (\alpha_{n+1}, \alpha_{n})$ for some $n \in \mathbb{N}$ that is an eigenvalue of $\As $. Denote by $\phi^*$ the eigenfunction corresponding to $\alpha^*$, and by $\Zc\big[f, (0,\infty)\big]$ the number of zeros of $f$ on $(0,\infty)$. Sturm-Liouville theory asserts that 
\begin{equation*}
\Zc\big[\phi_{n+1}, (0,\infty)\big] > \Zc\big[\phi^*,(0,\infty)\big] > \Zc\big[\phi_{n},(0,\infty)\big] ,
\end{equation*} 
which is a contradiction as the term on the left is $n + 1 $ and that on the right is $n$ from Step 2. The case of an eigenvalue $\alpha^*>\alpha_0$ is ruled out similarly using that $\phi_0$ has no zero on $(0,\infty)$. Note that the multiplicity of the eigenvalues of $\As$ has to be of one, since from a direct check, the eigenfunction equation is an ODE that admits at least one solution growing like $r^{C}e^{br^2/2}$ at infinity for some constant $C$, and thus which is not in $L^2(\omega_b)$.\\

\noindent 
\textbf{Step 3} \emph{Proof of remaining bounds:} The spectral gap (iii) is an immediate consequence of the fact that $\As$ is self-adjoint with compact resolvent, that $\alpha_0,...,\alpha_N$ have multiplicity one, and that any other eigenvalue is smaller than $\alpha_N$.\\

To prove the pointwise estimate \fref{est:PhinPointEst}, we use the decomposition \fref{id:deftildephin}. The function $\tilde \phi_n$ satisfies the bound \fref{est:PhinPointEst} because of the even sharper bound \fref{bd:refinedtildephin}. The function $\sum_{j = 0}^n c_{n,j}b^{j} T_j\big(r\big)$ satisfies also the bound \fref{est:PhinPointEst} because of the estimates \fref{est:T0iat0} and \fref{est:T0iatinf}. Hence $\phi_n$ satisfies \fref{est:PhinPointEst}.\\

To prove the $L^2(\omega_b)$ estimate \fref{est:PhinL2norm}, we use the decomposition \fref{id:deftildephin} again. From the bound \fref{bd:refinedtildephin} we infer that, using $U^{-1}(r)\lesssim \langle r\rangle^{3}$ and dividing the integral the two zones $r\leq \sqrt{b}^{-1}$ and $r\leq \sqrt{b}^{-1}$:
\bee
\int_0^\infty |\tilde \phi_n|^2\omega_b\lesssim \int_0^{\sqrt b^{-1}}\frac{\langle r\rangle^{-4}}{|\ln b|^2}\langle r \rangle^3dr+\int_{\sqrt b^{-1}}^{\infty} \frac{r^{4n-4+2\delta}b^{2n+\delta}}{|\ln b|^2}r^3 e^{-\frac{br^2}{2}}dr\lesssim \frac{1}{|\ln b|}.
\eee
One then computes from \fref{def:U}, \fref{def:psi01}, \fref{est:T0iatinf} and \fref{id:di}  that for $n=0,1$ as $b\rightarrow 0$:
$$
\int_0^\infty T_0^2(r)\omega_b (r)\sim \frac 18 \int_1^\infty r^{-1}e^{-\frac{br^2}{2}}dr\sim \frac{|\ln b|}{16},
$$
$$
\int_0^\infty b^2T_1^2(r)\omega_b (r)\sim \frac{1}{32} \int_1^\infty b^2 r^3 |\ln r|^2e^{-\frac{br^2}{2}}dr\sim \frac{|\ln b|^2}{64},
$$
by using that $\int_0^{\infty}\zeta^3 e^{-\zeta^2/2}d\zeta=2$. The three above identities, with the explicit values \fref{id:recurrencecnj} for $c_{i,j}$, prove \fref{est:PhinL2norm} for $n=0,1$. The general bound for $n\geq 2$ follows similarly.

\end{proof}

\subsection{Analysis in the inner zone $r \leq R_0$} In this part, we solve equation \fref{eq:phi} in the interval $[0, R_0]$ where we consider $-br\partial_r - \alpha$ as a small perturbation of $\As_0$. Let us recall some basic properties of $\As_0$ in the following. We introduce the norms

$$
\| f \|_{X^a_\imath}:=\sup_{r\in [0,R_0]} \frac{\rj^{2-a}}{r^2 \big(1 + \ln \rj\big)^\imath}|f(r)| \quad \mbox{ for } \imath=0,1 \mbox{ and } \quad \| f \|_{X^a_{-1}}:=\sup_{r\in [0,R_0]} \frac{\rj^{2-a}}{r^2 \big(1+\frac{2\ln \rj}{\ln b}\big)}|f(r)|
$$

for any $a\in \mathbb R$, and the function spaces for $\imath=-1,0,1$,
\begin{equation}\label{def:I0a}
\Ic_{\imath}^{a} := \left\{ f:\;\; \|f\|_{\Ic_{\imath}^{a}} \triangleq \| f \|_{X^a_\imath}+\| r\pa_r f \|_{X^a_\imath}+\|r^2\pa_{r}^2 f \|_{X^a_\imath} < \infty\right\}.
\end{equation}

\begin{lemma}[Properties of $\As_0$]  \label{lemm:proAs}$\quad$ 
\begin{itemize}
\item[(i)] \textup{(Inversion)} For any $f \in \Cc(\Rb^+, \Rb)$, a solution to $\As_0 u =f$ is given by:

\begin{align}\label{def:invAs}
\As_0^{-1} f(r) &:= \frac{1}{2} \psi_0(r) \int_r^{1} \frac{\zeta^4 + 4\zeta^2 \ln \zeta - 1}{\zeta} f(\zeta) d\zeta + \frac{1}{2}\tilde\psi_0(r)\int_0^r \zeta f(\zeta) d\zeta,
\end{align}

where $\psi_0$ and $\tilde{\psi}_0$ are the two linearly independent solutions to $\As_0 \psi = 0$ given by
\begin{equation}\label{def:psi01}
\psi_0(r) = \frac{r^2}{\rj^4} \quad \textup{and} \quad \tilde{\psi}_0(r) = \frac{r^4 + 4r^2\ln r - 1}{\rj^4}.
\end{equation}
\item[(ii)] \textup{(Continuity)} Let $\imath \in \{-1,0,1\}$ and $a>-2$, then there holds the estimate:
\begin{equation}\label{est:conLc0}
\|\As_0^{-1}f\|_{\Ic_\imath^{a+2}} \lesssim \|f\|_{X_\imath^{a}}.
\end{equation}
\item[(iii)] \textup{(Iterative kernel of $\As_0$)} There exists a family of smooth radial functions $\big\{T_{i}\big\}_{i \in \mathbb{N}}$ defined as 
\begin{equation}\label{def:T0i}
\As_0 T_{i+1} := -T_{i}, \quad T_{0} := \psi_0,
\end{equation}
which admit the asymptotic estimates
\begin{align}
 &\Big|\big(r \partial_r\big)^p T_{i}\Big| = \Oc(r^{2}) \quad \textup{as} \quad r \to 0, \quad \forall p \in \mathbb{N}, \quad \label{est:T0iat0} \\
&T_{i} =  r^{2(i - 1)} \Big(\hat d_{i} \ln r + d_i\Big) + \Oc\big(r^{2(i-2)}\ln^{i+1}r\big) \quad \textup{as} \quad r \to \infty,\label{est:T0iatinf}\\
&r\partial_rT_{i} = r^{2(i - 1)} \left[ 2(i - 1) \big(\hat d_{i} \ln r  + d_i\big) +  \hat d_i\right]  + \Oc\big(r^{2(i-2)}\ln^{i+1}r\big), \label{est:T0iatinfdev}\\
 & \Big|\big(r\partial_r\big)^p T_{i}\Big| = \Oc(r^{2(i-1)}\ln r) \quad \textup{as} \quad r \to \infty, \quad \forall p \in \mathbb{N},\label{est:T0iatinfHigdev}
\end{align}
where $d_i \in \Rb$ and 
\begin{align} \label{id:di} 
\hat d_1 = -\frac{1}{2}, \;\; d_1 = \frac{1}{4},  \quad \hat d_{i + 1} = -\frac{\hat d_i}{4i(i+1)}, \quad d_{i+1} = \frac{1}{8}\left(\frac{\hat d_i - 2id_i}{i^2} - \frac{\hat d_i - (2i + 2)d_i}{(i+1)^2}\right).
\end{align}
\end{itemize}
\end{lemma}

\begin{proof} $(i)$ By the scaling invariance of the problem \fref{eq:KS}, we have $\frac{d}{d \lambda} \left[ \Delta U_\lambda - \nabla \cdot( \nabla U_\lambda) \right]_{\lambda = 1} = 0$, or $\Ls^y \Lambda U = 0$. Hence $\psi_0 = \frac{1}{8}\int_0^r \Lambda U (x) x dx$ is the first fundamental solution to $\As_0 \psi = 0$. The explicit formula of $\tilde{\psi}_0$ follows from the integration of the Wronskian relation, and the formula \fref{def:invAs} is a standard way to solve linear second order ODEs.

$(ii)$ We denote $u=\As_0^{-1}f$. We directly compute from \fref{def:invAs} for $r \leq 1$ that for any $a,\imath$:
 \begin{align*}
|u(r)| &\lesssim \left|\psi_0(r)\int_r^1 \frac{\xi^4 + 4\xi^2 \ln \xi - 1}{\xi}f(\xi) d\xi + \tilde{\psi}_0(r)\int_0^r \xi f(\xi)d\xi\right|\\
&\quad \lesssim \left(\sup_{0 \leq \xi \leq 2} \xi^{-2} |f(\xi)|\right) \left(r^2 \int_r^1 \xi d\xi  + \int_0^r \xi^3 d\xi\right) \lesssim r^4 \| f \|_{X^a_\imath}.
\end{align*}
For $1 \leq r \leq R_0$, we use again formula \fref{def:invAs} to compute for $\imath=0,1$ and $a>-2$:
\begin{align*}
|u(r)| &\lesssim |\psi_0(r)| \int_1^r \xi^3 |f(\xi)| d\xi   + |\tilde{\psi}_0(r)| \int_0^r \xi |f(\xi)|d\xi \\
& \lesssim r^{-2} \sup_{1 \leq \xi \leq R_0} \frac{\xi^{- a}|f(\xi)|}{(1+ \ln \langle \xi\rangle)^\imath } \int_1^r \xi^3\xi^a (1+\ln \langle \xi\rangle)^\imath d\xi   +\sup_{0 \leq \xi \leq R_0} \frac{\langle \xi\rangle ^{- a}|f(\xi)|}{(1+ \ln \langle \xi\rangle)^\imath } \int_0^r \xi \langle \xi\rangle ^{ a} (1+\ln \langle \xi\rangle )^\imath d\xi\\
&  \lesssim  \| f \|_{X_\imath^a}r^{a + 2}(1+\ln \rj)^\imath 
\end{align*}
For $\imath=-1$ we first notice that the function $1+\frac{2\ln \rj}{\ln b}$ is decreasing and satisfies for any $r\in[0,R_0]$:
$$
\frac{1}{|\ln b|}\leq \frac{|\ln \zeta_0|}{|\ln b|} \leq 1+\frac{2\ln \rj}{\ln b}\leq 1,
$$
so that for $r\in [1,R_0]$ and $a>-1$, with constants independent on $b$:
\begin{align}
&\int_0^r \langle \xi\rangle^a (1+\frac{2\ln \langle \xi \rangle}{\ln b})d\xi \lesssim 1 +\left| \int_2^r \xi^a (1+\frac{2\ln  \xi}{\ln b})d\xi\right| \nonumber \\
& \qquad  \lesssim 1 +\left| r^{a+1}\left(1+\frac{2\ln r}{\ln b}\right)-\frac{r^{a+1}}{\ln b} \right| \lesssim r^{a+1}\left(1+\frac{2\ln \rj}{\ln b}\right). \label{estimationpoidslogloss} 
\end{align}
Hence for $\imath=-1$ and $a>-2$, computing as above:
\begin{align*}
|u(r)|  & \lesssim r^{-2}  \sup_{1 \leq \xi \leq R_0} \frac{\xi^{- a}|f(\xi)|}{1+\frac{2\ln \langle \xi \rangle}{\ln b} } \int_1^r \xi^3\xi^a (1+\frac{2\ln \langle \xi \rangle}{\ln b}) d\xi   +\sup_{0 \leq \xi \leq R_0} \frac{\langle \xi\rangle ^{- a}|f(\xi)|}{1+\frac{2\ln \langle \xi \rangle}{\ln b} } \int_0^r \xi \langle \xi\rangle ^{ a} (1+\frac{2\ln \langle \xi \rangle}{\ln b}) d\xi \\
&  \lesssim  \| f \|_{X_\imath^a} r^{a + 2}(1+\frac{2\ln \rj}{\ln b}).
\end{align*}
The estimates above imply for any $a>-2$ and $\imath=-1,0,1$, with a constant independent on $b$ and $\zeta_0$:
$$
\| \As_0^{-1} f\|_{X_\imath^{a+2}}\lesssim \| f\|_{X_\imath^a}.
$$
To estimate the derivatives, we notice from \fref{def:invAs} that

\be \label{id:differentiationAinv}
\pa_r u=\frac 12 \pa_r \psi_0 \int_r^1 \frac{\xi^4+4\xi^2\ln \xi -1}{\xi}f(\xi)d\xi+\frac 12 \pa_r \tilde \psi_0 \int_0^r \xi f(\xi)d\xi.
\ee

Hence, with the very same computations that we do not repeat we obtain for $\imath=-1,0,1$ and $a>-2$:
$$
\| r \pa_r u \|_{X_\imath^{a+2}}\lesssim \| f\|_{X_\imath^a}.
$$
Next, using that $\As_0 u =f$ and the definition of $\As_0$ yields
$$
\pa_r^2 u=f +\left(\frac 1r-\frac{4r}{\rj^2}\right)\pa_r u -\frac{8}{\rj^4}u.
$$
so that for $\imath=-1,0,1$ and $a>-2$, using the previous estimates for $u$ and $r\pa_r u$:
\bee
\| r^2 \pa_{rr} u \|_{X_\imath^{a+2}} & \lesssim & \| r^2 f\|_{X_\imath^{a+2}}+\|\left(r-\frac{4r^3}{\rj^2}\right)\pa_r u \|_{X_\imath^{a+2}}+ \| \frac{8r^2}{\rj^4}u\|_{X_\imath^{a+2}}\\
&\lesssim & \| f \|_{X_\imath^a}+\| f \|_{X_\imath^a}+\| f \|_{X_\imath^a}\lesssim \| f \|_{X_\imath^a}.
\eee
This concludes the proof of \fref{est:conLc0}.\\

$(iii)$ For $r \ll 1$, we compute from \fref{def:invAs} 
$$|T_{1}(r)| + |r\partial_r T_{1}(r)| = \Oc\left(r^2\int_r^1 \xi^{-1}\xi^2 d\xi + \int_0^r \xi \xi^2 d\xi\right) = \Oc(r^2) \quad \textup{as} \quad r \to 0.$$
We use $\As_0 T_{1} = -\psi_0$ and the definition \fref{def:As} of $\As_0$ to estimate for $k \in \mathbb{N}$,
$$|(r\partial_r)^{k+2}T_1(r)| = \Oc\left(\sum_{j = 0}^k|r^{j+1}\partial^{j+1}_r T_1| + r^{k+2}|\partial_r^k\psi_0|\right) = \Oc(r^2) \quad \textup{as} \quad r \to 0.$$
Hence, the estimate \fref{est:T0iat0} holds for $i = 1$. By induction, we assume that estimate \fref{est:T0iat0} holds for $i \geq 1$. We compute from \fref{def:invAs} and the relation $T_{i+1} = -\As_0^{-1}T_{i}$, 
\begin{align*}
|T_{i+1}| + |r\partial_r T_{i+1}| = \Oc\left(r^2 \int_r^1 \xi^{-1} \xi^{2}d\xi + \int_0^r \xi \xi^{2}d\xi\right) = \Oc(r^{2}),
\end{align*}
as $r \to 0$. The estimate for higher derivative follows from the relation $\As T_{i+1} = -T_{i}$ and the definition \fref{def:As} of $\As_0$. 

For $1 \ll r \leq R_0$, we prove \fref{est:T0iatinf} by induction. For $i = 1$, we compute from \fref{def:invAs} and the relation $T_{1} = -\As_0^{-1}\psi_0$ 
\begin{align*}
T_{1}(r)& = \frac{1}{2}\psi_0 \int_r^1 \frac{\xi^4 + 4\xi^2\ln \xi - 1}{\xi} \psi_0(\xi) d\xi + \frac{1}{2}\tilde{\psi}_0 \int_0^r \xi \psi_0(\xi) d\xi \\
& =  \left(\frac{1}{2r^2} + \Oc(r^{-4})\right) \left(\frac{1}{2}r^2 + \Oc(\ln^2 r) \right) - \left(\frac{1}{2} + \Oc(r^{-2}\ln r)\right) \left(\ln r + \Oc(r^{-2}) \right)\\
& = -\frac{1}{2}\ln r + \frac{1}{4} + \Oc\left(\frac{\ln ^2 r}{r^2}\right),
\end{align*}
which is \fref{est:T0iatinf} for $i = 1$. Assuming now that expansion \fref{est:T0iatinf} holds for some $i \geq 1$, we use formula \fref{def:invAs}, the relation $T_{i+1} = -\As_0^{-1}T_{i}$ and the elementary identity
$$\int_0^r s^k \ln s \; ds = \frac{r^{k+1}\big[(k+1) \ln r - 1\big]}{(k+1)^2} \quad \textup{for all} \;\; k \in \mathbb{N},$$
to compute
\begin{align*}
T_{i+1} &= \frac{1}{2}\psi_0 \int_r^1 \frac{\xi^4 + 4\xi^2\ln \xi - 1}{\xi} T_{i}(\xi) d\xi + \frac{1}{2}\tilde{\psi}_0 \int_0^r \xi T_{i}(\xi) d\xi \\
& = \left(\frac{1}{2r^2} + \Oc(r^{-4})\right) \left[ \frac{\hat d_i}{(2i + 2)}\ln r - \frac{\hat d_i - (2i + 2)d_i}{(2i + 2)^2} +\Oc\left(r^{-2}\ln^{i+2}r\right) \right] \\
& \quad -\left(\frac{1}{2} + \Oc(r^{-2}\ln r)\right) r^{2i}\left[ \frac{\hat d_i}{2i}\ln r - \frac{\hat d_i - 2id_i}{4i^2} +\Oc\left(r^{-2}\ln^{i+1}r\right) \right]\\
& = r^{2i}\left[ \frac{-\hat d_i}{4i(i+1)}\ln r + \frac{1}{8}\left(\frac{\hat d_i - 2id_i}{i^2} - \frac{\hat d_i - (2i + 2)d_i}{(i+1)^2}\right)  \right]  +  \Oc\left(r^{2i-2}\ln^{i+2}r\right),
\end{align*}
which gives 
$$\hat d_{i + 1} = -\frac{\hat d_i}{4i(i+1)}, \quad d_{i+1} = \frac{1}{8}\left(\frac{\hat d_i - 2id_i}{i^2} - \frac{\hat d_i - (2i + 2)d_i}{(i+1)^2}\right).$$
This concludes the proof of \fref{est:T0iatinf}.
 
The proof of \fref{est:T0iatinfdev} follows similarly by induction. Indeed, assuming that \fref{est:T0iatinfdev} holds for $i \in \mathbb{N}$, we compute from \fref{def:invAs}, the relation $T_{i+1} - =\As_0^{-1}T_{i}$ and the expansion \fref{est:T0iatinfdev} for $1 \ll r \leq R_0$:
\begin{align*}
r \partial_r T_{i+1} &= \frac{r}{2}\partial_r\psi_0 \int_r^1 \frac{\xi^4 + 4\xi^2\ln \xi - 1}{\xi} T_{i}(\xi) d\xi + \frac{r}{2}\partial_r\tilde{\psi}_0 \int_0^r \xi T_{i}(\xi) d\xi \\
& = \left(-r^{2i} + \Oc(r^{2i-2})\right)\left[ \frac{\hat d_i}{(2i + 2)}\ln r - \frac{\hat d_i - (2i + 2)d_i}{(2i + 2)^2} +\Oc\left(r^{-2}\ln^{i+2}r\right) \right] +\Oc\left(\frac{\ln^2 r}{r^{2-2i}}\right)\\
& = r^{2i}\left[ \frac{-\hat d_i}{2(i+1)}\ln r +  \frac{\hat d_i - 2(i+1)d_i}{4(i+1)^2}\right]  +  \Oc\left(r^{2i-2}\ln^{i+2}r\right).
\end{align*}
Using the recursive definition of $\hat d_i$ and $d_i$, i.e,
\be \label{id:recurrencedi}
\hat d_{i + 1} =  - \frac{\hat d_i}{4i(i+1)}, \quad d_i = -4i(i+1)d_{i+1} - 2(2i +1)\hat d_{i+1},
\ee
we have the simplification $ \frac{-\hat d_i}{2(i+1)} = 2i \hat d_{i+1}$ and 
\begin{align*}
\frac{\hat d_i - 2(i+1)d_i}{4(i+1)^2} & = \frac{\hat d_i}{4i^2} - \frac{d_i}{2i} - 2d_{i+1}= -\frac{i+1}{i}\hat d_{i+1} + \frac{2i + 1}{i}\hat d_{i + 1} + 2i d_{i+1} = \hat d_{i+1} + 2i d_{i+1}.
\end{align*}
This  concludes the proof of \fref{est:T0iatinfdev}. The estimate \fref{est:T0iatinfHigdev} follows by induction from the definition of $\As_0$, the relation $\As_0 T_{i+1} = -T_{i}$ and the Leibniz rule. This completes the proof of Lemma \ref{lemm:proAs}. 
\end{proof}

\medskip

In the following we show that the profiles $T_{j}$ given in Lemma \ref{lemm:proAs} are actually the building blocks of the eigenfunction of the linear operator $\As  = \As_0 - br\partial_r$ on $[0,R_0]$. In particular, we have the following.

\begin{lemma}[Inner eigenfunctions for the radial mode]  \label{lemm:inn0} Let $n \in \mathbb{N}$, $0 < \zeta_0 \ll 1$ and $0 < b \ll 1$ be small enough. Then for any $|\bar \alpha|\lesssim |\ln b|^{-2}$ there exists a smooth function $\phi_n^\inn \in \Cc^\infty([0, R_0], \Rb)$ satisfying
\begin{equation}\label{eq:Mninn}
\As  \phi_n^\inn = 2b\big(1 - n +\tilde \alpha \big)\phi_n^\inn \quad \textup{with} \quad \tilde \alpha=  \frac{1}{\ln b}+\bar{\alpha},
\end{equation}
where $\phi_n^\inn$ is of the form
\begin{equation}\label{eq:forMinn}
\phi_n^\inn(r) = \sum_{j =0}^nc_{n,j}b^jT_{j} +b\left(-\frac{2}{\ln b}T_1+\As_0^{-1}\Theta_0\right)+ 2\bar{\alpha}\sum_{j = 0}^n b^{j + 1}\Big(-c_{n,j}T_{j + 1} + S_{j}\Big)  + b\Rc_{n} ,
\end{equation}
and the constants $\big(c_{n,j} \big)_{0 \leq j \leq n}$ are given by 
\begin{equation} \label{id:recurrencecnj}
c_{n,j} = 2^j \frac{n!}{(n - j)!}, \quad c_{n, j+1} = 2(n - j)c_{n,j}, \quad c_{n,0} = 1.
\end{equation}
The corrective functions $R_{n}$, $S_{j}$ satisfy the following estimates for any $n\geq 0$:
\begin{align}
&\| S_{j}\|_{\Ic_1^{2j}+ \| b \partial_{b} S_{j} \|_{\Ic_1^{2j}}} + \|\partial_{\bar{\alpha}} S_{j} \|_{\Ic_1^{2j}} \lesssim \zeta_0^2, \label{est:S0j}\\
 & \|\Rc_{n}\|_{\Ic_{-1}^{0}}+\|b\partial_{b}\Rc_{n}\|_{\Ic_{-1}^{0}}+\|\partial_{\bar{\alpha}} \Rc_{n} \|_{\Ic_{-1}^0} \lesssim 1, \label{est:R0n}
\end{align}
with the following refinements for $n=0$:
\be \label{id:refinementSjn=0}
S_0=\frac 12 \sum_{i=1}^{\infty} \frac{1}{(2)_i2^i}b^i r^{2i}\log (r+1)+\tilde{S}_0, \quad \| \tilde S_0\|_{\Ic_0^{2}}+\| b\partial_b \tilde S_0\|_{\Ic_0^{2}}+\| \pa_{\bar \alpha}\tilde S_0\|_{\Ic_0^{2}}\lesssim b,
\ee
$$
\Rc_0=-\frac{1}{2} \sum_{i=1}^{\infty}\frac{1}{(2)_i 2^i}b^ir^{2i} \left\{ \frac{1}{\log b}\left[2\ln (r+1)-\Psi (i+2)-\gamma \right]+1 \right\}+\tilde \Rc_0,
$$
\be \label{id:refinementRn=0}
 \| \tilde \Rc_0\|_{\Ic_{-1}^0}+ \| b\partial_b \tilde \Rc_0\|_{\Ic_{-1}^0}\lesssim |\log b|^{-1}, \qquad  \| \partial_{\bar \alpha} \tilde \Rc_0\|_{\Ic_{-1}^0}\lesssim 1,
\ee
and for $n=1$:
$$
\Rc_1 =-\frac{1}{2} \sum_{i=1}^{\infty} \frac{(1)_{i-1}}{(2)_i i! 2^i}b^ir^{2i}\left\{ \frac{1}{\log b}\left[2\ln (r+1) -\frac 1i -\Psi (i+2)-\gamma \right]+1-\frac{1}{\log b} \right\}+\tilde \Rc_1,
$$
\begin{align} \label{id:refinementRn=1}
\| \tilde \Rc_1\|_{\Ic_{-1}^0}+\|b\partial_b \tilde \Rc_1\|_{\Ic_{-1}^0}\lesssim |\log b|^{-1}, \quad  \| \partial_{\bar \alpha} \tilde \Rc_1\|_{\Ic_{-1}^0}\lesssim 1.
\end{align}
$$
S_{1}=-\frac{1}{2} \sum_{i=2}^{\infty} \frac{(1)_{i-1}}{(2)_i i! 2^i}b^{i-1}r^{2i}\ln (r+1)+\tilde S_1,
$$
\be \label{id:refinementSjn=1}
\| S_0\|_{\Ic_0^{2}}+\| b\partial_b S_0\|_{\Ic_0^{2}}+\| \partial_{\bar \alpha}S_0\|_{\Ic_1^{2}}\lesssim b, \ \ \| \tilde S_1\|_{\Ic_0^{2}}+\| b\partial_b\tilde S_1\|_{\Ic_{0}^{2}}+\| \partial_{\bar \alpha}\tilde S_1\|_{\Ic_{1}^{2}}\lesssim  1.
\ee

Finally, on the interval $(0,R_0]$, $\phi^\inn_0$ does not vanish and $\phi^\inn_n$ has exactly one zero for $n\geq 1$.
\end{lemma}
\begin{proof} The proof mainly relies on classical arguments based on the Banach fixed point theorem to construct the corrective profiles $\Rc_{n}$ and $S_{j}$ for $0 \leq j \leq n$.\\

\noindent \textbf{Step 1} \emph{Preliminary results:}  For $j \in \mathbb{N}$, we let 
\begin{equation}\label{def:Theta0j}
\Theta_{j} = r\partial_r T_{j} - 2(j-1)T_{j},
\end{equation}
which admits the following slowly growing tail from  \fref{est:T0iatinf} and \fref{est:T0iatinfdev}, 
\begin{equation}\label{est:Theta0j}
|\Theta_0(r)| = \Oc(r^{-4}), \quad |\Theta_{j}(r)| = \Oc(r^{2(j-1)}) \;\; \textup{for} \;\; j \geq 1,
 \quad \textup{as} \quad r \to \infty.
\end{equation}
and for $j\geq 1$:
\be \label{bd:improveddecayThetajTj}
\left| \Theta_j(r)+\frac{2}{\ln b}T_j(r)\right|\lesssim r^{2} \langle r\rangle^{2(j-2)}\left(1+\frac{2\ln (r+1)}{\ln b}\right).
\ee
We compute the following integral by integrating by parts:
\bee
\int_0^{\infty} r \Theta_0(r)dr&=& \lim_{R\rightarrow \infty}\int_0^R r \Theta_0(r)dr= \lim_{R\rightarrow \infty}\left( 2\int_0^{R}rT_0 (r)dr+\int_0^Rr^2 \pa_{r} T_0(r)dr\right) \\
&=& \lim_{r\rightarrow \infty} R^2T_0(R) =1.
\eee

From this and \fref{def:invAs}, as $|r\pa_r T_0+2T_0|\lesssim (1+r)^{-4}$ the following corrective term satisfies as $r\rightarrow \infty$:
\bee
\As_0^{-1}\Theta_0(r) &=& \Oc( \psi_0 (r)\ln (r))+\Oc( r^{-2})+\frac 12 \tilde \psi_0(r)\int_0^{\infty} \zeta f(\zeta)d\zeta = \frac 12+\Oc( r^{-2}\ln r),
\eee
and hence:
\be \label{bd:improveddecayTheta0T0}
\left|-\frac{2}{\ln b}T_1+\As_0^{-1}\Theta_0\right|\lesssim r^{2} \langle r\rangle^{-2}\left(1+\frac{2\ln (r+1)}{\ln b}\right).
\ee
These estimates show that
\begin{align}
& \| b^{j-1}\left( \Theta_j(r)+\frac{2}{\ln b}T_j(r)\right) \|_{X^{0}_{-1}}\lesssim \zeta_0^{2(j-1)} \|  \Theta_j(r)+\frac{2}{\ln b}T_j(r) \|_{X^{2(j-1)}_{-1}}\lesssim 1, \nonumber \\
& \quad \| -\frac{2}{\ln b}T_1+\As_0^{-1}\Theta_0\|_{\mathcal I^0_{-1}}\lesssim 1.\label{bd:forcingweightedspaceinner}
\end{align}

\noindent \textbf{Step 2} \emph{Equations satisfied by $S_{j}$ and $\Rc_{n}$:} Plugging the decomposition \fref{eq:forMinn} into \fref{eq:Mninn} and using $\As_0 T_{j} = -T_{j-1}$ with the convention $T_{-1} = 0$ yields
\begin{align*}
&\Big[\As_0 - b r\partial_r - 2b\big(1 - n +\frac{1}{\ln b}+\bar \alpha \big)\Big] \sum_{j = 0}^n c_{n,j} b^j T_{j} \\
 & \quad = -\sum_{j = 0}^{n - 1} c_{n,j}b^{j  +1} T_{j} \Big[2(n - j) + 2(j - 1) - 2(n-1)\Big]- \sum_{j = 0}^n c_{n,j}b^{j + 1} \Theta_{j} - 2\left(\frac{1}{\ln b}+\bar \alpha\right)\sum_{j = 0}^n c_{n,j}b^{j+1}T_{j}\\
 & \quad = - \sum_{j = 1}^n c_{n,j}b^{j + 1} \left(\Theta_{j} +\frac{2}{\ln b}T_j\right)-b\Theta_0-\frac{2b}{\ln b}T_0 -2\bar \alpha \sum_{j = 0}^n c_{n,j}b^{j + 1} T_{j},
\end{align*}
and 
\bee
&&\Big[\As_0 - b r\partial_r - 2b\big(1 - n +\frac{1}{\ln b}+\bar \alpha \big)\Big]b\left(-\frac{2}{\ln b}T_1+\As_0^{-1}\Theta_0\right)\\
&=&b\Theta_0+\frac{2b}{\ln b}T_0-b\Big[ r\partial_r +2\big(1 - n +\frac{1}{\ln b}+\bar \alpha \big)\Big]b\left(-\frac{2}{\ln b}T_1+\As^{-1}\Theta_0\right)
\eee
and
\begin{align*}
&\Big[\As_0 - b r\partial_r - 2b\big(1 - n + \tilde{\alpha}\big)\Big] \left( 2\bar{\alpha}\sum_{j = 0}^nb^{j + 1}\Big[-c_{n,j} T_{j+1} + S_{j}\Big] \right)\\
& \quad = 2\bar{\alpha}\sum_{j = 0}^n b^{j + 1} \left\{\As_0 S_{j} - \Big[b r\partial_r + 2b\big(1 - n + \tilde{\alpha}\big)\Big]\Big(-c_{n,j} T_{j+1} + S_{j}\Big)   \right\} + 2\bar{\alpha} \sum_{j = 0}^n c_{n,j}b^{j + 1}T_{j}. 
\end{align*}
We then rewrite equation \fref{eq:Mninn} as 
\begin{align}
0=& \Big[\As_0 - b r\partial_r - 2b\big(1 - n + \tilde{\alpha}\big)\Big] \phi_n^\inn \nonumber \\
=& \bar{\alpha}\sum_{j = 0}^n b^{j + 1} \left\{\As_0 S_{j} - b\Big[r\partial_r + 2\big(1 - n + \tilde{\alpha}\big)\Big]\Big(-c_{n,j} T_{j+1} + S_{j}\Big)   \right\}\nonumber \\
\nonumber & + b\Bigl\{\As_0 \Rc_{n} - b\Big[r\partial_r + 2\big(1 - n + \tilde{\alpha}\big)\Big]\Rc_{n} - \sum_{j = 1}^n c_{n,j}b^{j} \left(\Theta_{j}+\frac{2}{\ln b}T_j\right)\\
&-\Big[ r\partial_r +2\big(1 - n +\tilde \alpha \big)\Big]b\left(-\frac{2}{\ln b}T_1+\As_0^{-1}\Theta_0\right)\Bigr\} \label{eq:idMinn}
\end{align}

\noindent \textbf{Step 3} \emph{Computation of $\big(S_{j}\big)_{0 \leq j \leq n}$:} From equation \fref{eq:idMinn}, we choose $S_{j}$ to be the solution of the equation 
\begin{equation}\label{eq:S0j}
\As_0 S_{j} = b\Big[ r\partial_r + 2\big(1 - n + \tilde{\alpha}\big)\Big]\Big(-c_{n,j} T_{j+1} + S_{j}\Big).
\end{equation}
Note from part $(iii)$ of Lemma \ref{lemm:proAs} that $T_{j+1} \in \Ic_1^{ 2j}$ for $j \geq 0$. We aim at proving that for $b$ and $\zeta_0$ small enough, there exists a unique solution $S_{j} \in \Ic_1^{2j}$ to equation \fref{eq:S0j} via the Banach fixed point theorem. Let $\Gamma$ be the affine mapping acting on $f \in \Ic_1^{ 2j}$ defined as 
\begin{align*}
\Gamma(f) &= \As_0^{-1} \Big[b \big( r\partial_r + 2\big(1 - n + \tilde{\alpha}\big)\big)\big(-c_{n,j} T_{j+1} + f\big)\Big] = \Gamma(0) + D\Gamma(f),
\end{align*}
where $\As_0^{-1}$ is defined as in \fref{def:invAs} and 
\begin{align*}
\Gamma(0) &= bc_{n,j}\As_0^{-1}\left(\Big[ r\partial_r + 2\big(1 - n + \tilde{\alpha}\big)\Big]T_{j+1}\right), \quad D\Gamma(f) = b\As_0^{-1}\left(\Big[ r\partial_r + 2\big(1 - n + \tilde{\alpha}\big)\Big] f\right).
\end{align*}
We estimate from \fref{est:conLc0}, 
\begin{align*}
\|\Gamma(0)\|_{\Ic_1^{ 2j}}\lesssim R_0^{2} \|\Gamma(0)\|_{\Ic_1^{ 2j+2}}\lesssim R_0^2 b \|T_{j+1}\|_{\Ic_1^{2j}} \lesssim R_0^2b\lesssim \zeta_0^2,
\end{align*}
and for all $f \in \Ic_{1}^{a}$ with $ a = 2j$ or $a = 2j + 2$, 
\begin{align}
\nonumber \|D\Gamma(f)\|_{\Ic_1^{a}} &\lesssim R_0^{-2}\|D\Gamma(f)\|_{\Ic_1^{a+2}}  = R_0^2b\| \As_0^{-1}\left(\Big[ r\partial_r + 2\big(1 - n + \tilde{\alpha}\big)\Big] f\right)\|_{\Ic_1^{a+2}} \\
\label{est:DGamma} & \lesssim \zeta_0^2\| \Big[ r\partial_r + 2\big(1 - n + \tilde{\alpha}\big)\Big] f \|_{X_1^{a}}  \lesssim \zeta_0^2 \|f\|_{\Ic_1^{ a}} 
\end{align}
Since $0 < \zeta_0 \ll 1$ and $\Gamma$ is an affine mapping, the above estimates imply that $\Gamma$ is a contraction on $B_{\Ic_1^{ 2j}}(0, C\zeta_0^2)$ for some constant $C>0$ independent of the problem. Therefore, there exists a unique fixed point $S_{j} = \Gamma(S_{j})$ such that $\| S_{j}\|_{\Ic_1^{2j}}\lesssim \zeta^2_0$ so that the first estimate in \fref{est:S0j} holds.  For the estimates of $\partial_{\tilde{\alpha}}S_{j}$ and $\partial_b S_{j}$, we differentiate the relation $S_{j} = \Gamma(S_{j})$ to obtain
\begin{align*}
\partial_b S_{j} &= D\Gamma(\partial_b S_{j}) + \big(\partial_b \Gamma\big) (S_{j}), \quad \partial_{\bar{\alpha}} S_{j} = D\Gamma(\partial_{\bar{\alpha}} S_{j}) + \big(\partial_{\bar{\alpha}} \Gamma\big) (S_{j}),
\end{align*}
where we have the identities since $b \pa_b \tilde \alpha=-1/(\ln b)^2$ and $\pa_{\bar \alpha}\tilde \alpha =1$:
\begin{align}
\partial_b \Gamma(f) &= \As_0^{-1} \Big[\big( r\partial_r + 2\big(1 - n + \tilde{\alpha}-\frac{1}{|\ln b|^2}\big)\big)\big(-c_{n,j} T_{j+1} + f\big)\Big],\label{id:Gdevb}\\
\partial_{\bar{\alpha}} \Gamma(f) &= b\As_0^{-1}\Big[-c_{n,j}T_{j + 1} + f\Big].\label{id:Gdevnu}
\end{align}
From \fref{est:DGamma}, we see that $\|D\Gamma\|_{\Ic_1^{ a} \to \Ic_{1}^{ a}} \lesssim \zeta_0^2$ with $a = 2j + 2$ or $a = 2j$. Hence $Id-D\Gamma$ is invertible and:
\begin{align*}
\big\|\partial_b S_{j}\big\|_{\Ic_1^{2j+2}} = \big\|(\textup{Id} - D\Gamma)^{-1} (\partial_b \Gamma)(S_{j}) \big\|_{\Ic_1^{ 2j+2}} \lesssim  \big\|(\partial_b \Gamma)(S_{j}) \big\|_{\Ic_1^{ 2j+2}}, \\
\big\|\partial_{\bar{\alpha}} S_{j}\big\|_{\Ic_1^{ 2j}} = \big\|(\textup{Id} - D\Gamma)^{-1} (\partial_{\bar{\alpha}} \Gamma)(S_{j})\big\|_{\Ic_1^{2j}} \lesssim  \big\|(\partial_{\bar{\alpha}} \Gamma)(S_{j}) \big\|_{\Ic_1^{2j}}.
\end{align*}
We estimate from \fref{est:conLc0} and \fref{id:Gdevb},
\begin{align*}
&\big\|(\partial_b \Gamma)(S_{j}) \big\|_{\Ic_1^{ 2j+2}} \quad \lesssim \|T_{j+1}\|_{\Ic_1^{ 2j}} + \|S_{j+1}\|_{\Ic_1^{ 2j}}  \lesssim 1.
\end{align*}
Similarly, we estimate from \fref{est:conLc0} and \fref{id:Gdevnu}, 
\begin{align*}
&\big\|(\partial_{\bar{\alpha}} \Gamma)(S_{j}) \big\|_{\Ic_1^{ 2j}} \lesssim bR_0^2 \Big(\|T_{j+1}\|_{\Ic_1^{ 2j}} + \|S_{j+1}\|_{\Ic_1^{2j}}\Big) \lesssim \zeta_0^2,
\end{align*}
which concludes the proof of \fref{est:S0j}.\\

\noindent \underline{\textit{Refinement for $n=1$}}. We do not give technical details are these are the very same ones as above for the general case. For $n=1$ the $S_0$ equation is:
$$
\As_0 S_{0} = b\Big[ r\partial_r + 2 \tilde{\alpha} \Big]\Big(- T_{1} + S_{0}\Big).
$$
As $ \| b(r\partial_r + 2 \tilde{\alpha})  T_{1}\|_{\Ic^0_0}\lesssim b$ from \fref{est:T0iatinf} and \fref{eq:Mninn}  we get $ \| S_0\|_{\Ic_0^{2}}+\| b\pa_b S_0 \|_{\Ic_{0}^2}+\| b\pa_{\tilde \alpha} S_0 \|_{\Ic_{0}^2}\lesssim \zeta_0^2b$ by the same strategy as above. The $S_1$ equation is:
$$
\As_0 S_{1} = b\Big[ r\partial_r + 2  \tilde{\alpha} \Big]\Big(-2 T_{2} + S_{1}\Big).
$$
Let $\hat S_1=-\frac{1}{2} \sum_{i=2}^{\infty} \frac{(1)_{i-1}}{(2)_i i! 2^i}b^{i}r^{2i}\ln r$, which produces $\left(\partial_{rr}+\frac 3 r-br\partial_r\right)\hat S_1=-\frac{br^2\ln r}{4}$. Looking for a solution $S_1=\hat S_1+\tilde S_1$ produces
$$
\As_0  \tilde S_1=b\Big[ r\partial_r + 2  \tilde{\alpha} \Big]\tilde S_1+b\left(\frac{r^2\ln r}{4}-2r\partial_r T_2\right)-4\tilde \alpha T_2 +b2\tilde \alpha \hat S_{1}-\left(\left(\frac{4r}{\langle r\rangle^2}-\frac 4r\right)\partial_r +\frac{8}{\langle r\rangle^4}\right) \hat S_{1}.
$$
The source term above is of size $1$ in $\Ic_0^0$ from \fref{est:T0iatinf} and \fref{eq:Mninn} so that from the strategy used above one obtains $\| \tilde S_1 \|_{\Ic^2_0}+\| b\partial_b\tilde S_1 \|_{\Ic^2_0}+\| \partial_{\bar \alpha }\tilde S_1 \|_{\Ic_1^2}\lesssim \zeta_0^2$.\\

\noindent \underline{\textit{Refinement for $n=0$}}. For $n=0$ the $S_0$ equation is:
$$
\As_0 S_{0} = b\Big[ r\partial_r + 2\big(1 + \tilde{\alpha}\big)\Big]\Big(- T_{1} + S_{0}\Big).
$$
We look for a solution $S_0=\hat S_0(r+1)+\tilde S_0$ with $\hat S_0=\frac 12 \sum_{i=1}^{\infty} \frac{1}{(2)_i2^i}b^i r^{2i}\log (r)$. As $(\partial_{rr}+\frac 3 r\partial_r-b(r\partial_r+2))\hat S_0=b\log r$, $\tilde S_0$ solves

\bee
\As_0  \tilde S_0 & = &b\Big[ r\partial_r + 2\big(1 + \tilde{\alpha}\big)\Big]\tilde S_{0}\\
&&-b(2T_1+\log (r+1))-b(r\partial_r +2\tilde \alpha)T_1 -\left(\left(\frac{4r}{\langle r\rangle^4}-\frac 1r-\frac{3}{r+1}\right)\partial_r +\frac{8}{\langle r\rangle^4}+b\partial_r-2b\tilde \alpha\right) \hat S_{0}(r+1).
\eee

The source term above is of size $b$ in $\Ic_0^0$ from \fref{est:T0iatinf} and \fref{eq:Mninn} so that from the strategy used above  $\| \tilde S_0 \|_{\Ic^2_0}+\| b\partial_b \tilde S_0 \|_{\Ic^2_0}+\| \pa_{\bar \alpha}\tilde S_0 \|_{\Ic^2_1}\lesssim b\zeta_0^2$.\\

\noindent \textbf{Step 4} \emph{Computation of $\Rc_{n}$}: From \fref{eq:idMinn}, we choose $\Rc_{n}$ to be the solution of the equation 
\bea 
\nonumber \As_0 \Rc_{n} &=&  b\Big[r\partial_r + 2\big(1 - n + \tilde{\alpha}\big)\Big]\Rc_{n}  + \sum_{j = 1}^n c_{n,j}b^{j} \left(\Theta_{j}+\frac{2}{\ln b}T_j\right),\\
\label{eq:R0n} &&+\Big[ r\partial_r +2\big(1 - n +\tilde \alpha \big)\Big]b\left(-\frac{2}{\ln b}T_1+\As_0^{-1}\Theta_0\right).
\eea
where $\Theta_{j}$ is introduced in \fref{def:Theta0j}. The computation is similar to that for $S_{j}$. We let $\Gamma$ be the affine mapping $\Gamma(f) = \Gamma(0) + D\Gamma(f)$, where 
\begin{align*}
\Gamma(0) &= -b\As_0^{-1}\left[\sum_{j = 1}^n c_{n,j}b^{j-1} \left(\Theta_{j}+\frac{2}{\ln b}T_j\right)-\Big[ r\partial_r +2\big(1 - n +\tilde \alpha \big)\Big]\left(-\frac{2}{\ln b}T_1+\As_0^{-1}\Theta_0\right)\right],
\end{align*}
$$
D\Gamma(f) = b\As_0^{-1}\Big[\big(r\partial_r + 2\big(1 - n + \tilde{\alpha}\big)\big)f\Big].
$$
From \fref{bd:forcingweightedspaceinner} and \fref{est:conLc0} we obtain:
\begin{align*}
\|\Gamma(0)\|_{\Ic_{-1}^{2}} & \lesssim b \sum_{j = 1}^n \left\|\As_0^{-1} b^{j-1} \left(\Theta_{j}+\frac{2}{\ln b}T_j\right)\right\|_{\Ic_{-1}^{2}}+b\left\| \As^{-1}_0(-\frac{2}{\ln b}T_1+\As_0^{-1}\Theta )\right\|_{\Ic_{-1}^2}\\
&\lesssim b \sum_{j = 1}^n \left\| b^{j-1} \left(\Theta_{j}+\frac{2}{\ln b}T_j\right)\right\|_{\Ic_{-1}^{0}}+b\left\| -\frac{2}{\ln b}T_1+\As_0^{-1}\Theta \right\|_{\Ic_{-1}^0}\lesssim b.
\end{align*}
Using \fref{est:conLc0}, we estimate for all $f\in \Ic_{-1}^{2}$,
\begin{align}
\|D\Gamma(f)\|_{\Ic_{-1}^{2}}\lesssim b \| \big(r\partial_r + 2\big(1 - n + \tilde{\alpha}\big)\big)f \|_{X_{-1}^0} \lesssim bR_0^2 \|f\|_{\Ic_{-1}^{2}} \lesssim \zeta_0^2 \|f\|_{\Ic_{-1}^{2 }}.\label{est:DGamma1}
\end{align}
We then deduce that $\Gamma(f)$ is contraction on $B_{\Ic_{-1}^{2}}(0,bC)$ for some constant $C>0$, hence, there exists a unique fixed point $\Rc_{n} = \Gamma(\Rc_{n})$ satisfying $\| \Rc_{n} \|_{\Ic_{-1}^{2}}\lesssim b$. As $\| \Rc_{n} \|_{\Ic_{-1}^{0}}\lesssim b^{-1}\| \Rc_{n} \|_{\Ic_{-1}^{2}}\lesssim 1$ the first estimate in \fref{est:R0n} holds. For the estimates of $\partial_b \Rc_{n}$ and $\partial_{\bar{\alpha}}\Rc_{n}$, we differentiate the relation $\Rc_{n} = \Gamma(\Rc_{n})$: 
\begin{align*}
\partial_b \Rc_{n} &= D\Gamma(\partial_b \Rc_{n}) + \big(\partial_b \Gamma\big) (\Rc_{n}),  \quad 
\partial_{\bar{\alpha}} \Rc_{n} = D\Gamma(\partial_{\bar{\alpha}} \Rc_{n}) + \big(\partial_{\bar{\alpha}} \Gamma\big) (\Rc_{n}),
\end{align*}
where we have the identities since $b \pa_b \tilde \alpha=-1/(\ln b)^2$ and $\pa_{\bar \alpha}\tilde \alpha=1$,
\bee
\nonumber \partial_b \Gamma(f) &=& \As_0^{-1} \sum_{j = 1}^n c_{n,j}b^{j-1} \left(j\left(\Theta_{j}+\frac{2}{\ln b}T_j\right)-\frac{2}{|\ln b|^2}T_j\right)\\
&&-\As_0^{-1}\Big[ r\partial_r +2\big(1 - n +\tilde \alpha-\frac{1}{|\ln b|^2} \big)\Big]\left(-\frac{2}{\ln b}T_1+\As_0^{-1}\Theta_0+\frac{2}{|\ln b|^2}T_j\right)\\
&&+\As_0^{-1}\Big[\big(r\partial_r + 2\big(1 - n + \tilde{\alpha}-\frac{1}{|\ln b|^2}\big)\big)f\Big],
\eea
$$
\partial_{\bar{\alpha}} \Gamma(f) = b\As_0^{-1}f+b\As_0^{-1}\left(-\frac{2}{\ln b}T_1+\As_0^{-1}\Theta_0\right)
$$
We have derived from \fref{est:DGamma1} that $\|D\Gamma\|_{\Ic_{-1}^{2} \to \Ic_{-1}^{2}} \lesssim \zeta_0^2$, hence, $\textup{Id} - D\Gamma$ is invertible on $\Ic_{-1}^{2}$. In particular, we have the estimates
\begin{align*}
\|\partial_b \Rc_{n}\|_{\Ic_{-1}^{2}} = \big\|(\textup{Id} - D\Gamma)^{-1}\big(\partial_b \Gamma\big) (\Rc_{n}) \big\|_{\Ic_{-1}^{2}} \lesssim \big\|\big(\partial_b \Gamma\big) (\Rc_{n}) \big\|_{\Ic_{-1}^{2}},\\
\|\partial_{\bar{\alpha}} \Rc_{n}\|_{\Ic_{-1}^{2}} = \big\|(\textup{Id} - D\Gamma)^{-1}\big(\partial_{\bar{\alpha}} \Gamma\big) (\Rc_{n}) \big\|_{\Ic_{-1}^{2}} \lesssim \big\|\big(\partial_{\bar{\alpha}} \Gamma\big) (\Rc_{n}) \big\|_{\Ic_{-1}^{2}}.
\end{align*}
Using \fref{bd:forcingweightedspaceinner}, $\| \frac{b^{j-1}}{|\log b|^2}T_j\|_{\Ic^0_{-1}}\lesssim 1$, the estimate on $\Rc_{n}$, we have by \fref{est:conLc0}:
\begin{align*}
&\big\|\big(\partial_b \Gamma\big) (\Rc_{n}) \big\|_{\Ic_{-1}^{2}} \leq
\sum_{j = 1}^n \left( b^{j-1} \| \As_0^{-1}  \left(\Theta_{j}+\frac{2}{\ln b}T_j\right)\|_{\Ic^2_{-1}}+\| \As_0^{-1} b^{j-1}\frac{1}{|\ln b|^2}T_j\|_{\Ic^{2}_{-1}}\right)\\
&+\| \As_0^{-1}\left(r\partial_r +2\big(1 - n +\tilde \alpha-\frac{1}{|\ln b|^2} \big)\Big]\left(-\frac{2}{\ln b}T_1+\As_0^{-1}\Theta_0+\frac{2}{|\ln b|^2}T_j\right)\right)\|_{\Ic^{2}_{-1}}\\
&+\| \As_0^{-1}\Big[\big(r\partial_r + 2\big(1 - n + \tilde{\alpha}-\frac{1}{|\ln b|^2}\big)\big)\Rc_n \|_{\Ic^{2}_{-1}}\\
& \quad \lesssim 1+\| \Rc_n \|_{\Ic^{0}_{-1}} \lesssim 1+b^{-1}\| \Rc_n \|_{\Ic^{2}_{-1}} \lesssim 1.
\end{align*}
Similarly, we have by \fref{est:conLc0},
\begin{align*}
\big\|\big(\partial_{\bar{\alpha}} \Gamma\big) (\Rc_{n}) \big\|_{\Ic_{-1}^{2}} &\leq \| b\As_0^{-1}\Rc_n\|_{\Ic^{2}_{-1}}+b\|\As^{-1}_0\left(-\frac{2}{\ln b}T_1+\As_0^{-1}\Theta_0\right)\|_{\Ic_{-1}^2}\\
& \lesssim \| b \Rc_n\|_{\Ic^{0}_{-1}}+b\| -\frac{2}{\ln b}T_1+\As_0^{-1}\Theta_0 \|_{\Ic_{-1}^0} \lesssim  \| \Rc_n\|_{\Ic^{2}_{-1}}+b \lesssim b.
\end{align*}
Hence $\|\partial_{\bar{\alpha}} \Rc_{n}\|_{\Ic_{-1}^{0}}\lesssim b^{-1}\|\partial_{\bar{\alpha}} \Rc_{n}\|_{\Ic_{-1}^{2}} \lesssim 1$.\\

\noindent \underline{ \textit{Computation of $\Rc_{1}$}:} For $n=1$ a refinement is necessary. The equation for $\Rc_1$ is:
\bea 
\nonumber \As_0 \Rc_{1} &=&  b\Big[r\partial_r + 2  \tilde{\alpha} \Big]\Rc_{1}  + 2 b \left(r\pa_r T_1+\frac{2}{\ln b}T_1\right)+\Big[ r\partial_r +2 \tilde \alpha \Big]b\left(-\frac{2}{\ln b}T_1+\As_0^{-1}\Theta_0\right).
\eea
We look for a solution under the form $\Rc_1(r)= \Rc_{1,1}(r)+\Rc_{1,2}(r+1)+\tilde \Rc$ where
$$
\Rc_{1,1}=- (\frac 12-\frac{1}{2\ln b}) \sum_{i=1}^\infty \frac{(1)_{i-1}}{(2)_i i!2^i}b^i r^{2i}, \ \ (\pa_{rr}+\frac 3r\partial_r-br\partial_r)\Rc_{1,1}= -(1-\frac{1}{\ln b})b.
$$
$$
\Rc_{1,2}=-\frac{1}{2\log b} \sum_{i=1}^{\infty} \frac{(1)_{i-1}}{(2)_i i! 2^i}b^ir^{2i}\left[2\ln (r) -\frac 1i -\Psi (i+2)-\gamma \right],
$$
$$
 \big(\pa_{r}^2+\frac 3r\partial_r-br\partial_r\big)\Rc_{1,2}=-\frac{b}{\log b}(2\log r -1),
$$
so that
\bee
&& \As_0 \tilde{\Rc}_1-b(r\partial_r+2\tilde \alpha)\tilde{\Rc}_1\\
&=& \left(1-\frac{1}{\log b}\right)b+\frac{b}{\log b}(2\log (r+1) -1)+2b(r\partial_r T_1+\frac{2}{\log b}T_1)+\Big[ r\partial_r +2 \tilde \alpha \Big]b\left(-\frac{2}{\ln b}T_1+\As_0^{-1}\Theta_0\right)\\
&& -\left(\left(\frac{4r}{\langle r\rangle^4}-\frac 4r\right)\partial_r +\frac{8}{\langle r\rangle^4}\right) \Rc_{1,1}-\left(\left(\frac{4r}{\langle r\rangle^4}-\frac 1r-\frac{3}{r+1}\right)\partial_r +\frac{8}{\langle r\rangle^4}+b\partial_r\right) \Rc_{1,2}(r+1)\\
&&+2b\tilde \alpha (\Rc_{1,1}+\Rc_{1,2}(r+1))
\eee
Each line in the right hand side above contains cancellations as $r\rightarrow \infty$: the first is $\Oc(br^{-1}\log r)$ from \fref{est:T0iatinf}, so is the second from the definition of $\Rc_{1,1}$ and $\Rc_{1,2}$. For the last line, $\|\Rc_{1,1}+\Rc_{1,2}\|_{\Ic_{-1}^0}\lesssim 1$ and $|\tilde \alpha|\lesssim |\log b|^{-1}$. This shows that the right hand side is of size $|\log b|^{-1}$ in $\Ic_{-1}^0$. So that $\|\tilde \Rc_1\|_{\Ic_{-1}^0}\lesssim |\log b|^{-1}$, $\|b\pa_b\tilde \Rc_1\|_{\Ic_{-1}^0}\lesssim |\log b|^{-1}$ and $\| \partial_{\bar \alpha} \tilde \Rc_1\|_{\Ic_{-1}^0}\lesssim 1$.\\

\noindent \underline{ \textit{Computation of $\Rc_{0}$}:} For $n=0$ a refinement is also necessary. The equation for $\Rc_0$ is:
$$
\nonumber \As_0 \Rc_{0} =  b\Big[r\partial_r + 2\big(1+ \tilde{\alpha}\big)\Big]\Rc_{0}  +\Big[ r\partial_r +2\big(1 +\tilde \alpha \big)\Big]b\left(-\frac{2}{\ln b}T_1+\As_0^{-1}\Theta_0\right).
$$
We look for a solution under the form $\Rc_0(r)= \Rc_{0,1}(r)+\Rc_{0,2}(r+1)+\tilde \Rc_0$ where
$$
\Rc_{0,1}=\frac 12  \sum_{i=1}^{\infty} \frac{1}{(2)_i2^i}b^i r^{2i}, \ \ (\partial_{rr}+\frac 3r \partial_{r}-b(r\partial_r+2))\Rc_{0,1}=b
$$
$$
\Rc_{0,2}=\frac{1}{2\log b} \sum_{i=1}^{\infty}\frac{1}{(2)_i 2^i}b^ir^{2i}[2\log (r) -\Psi (i+2)-\gamma], \ \ (\partial_{rr}+\frac 3r \partial_{r}-b(r\partial_r+2))\Rc_{0,2}=\frac{2b}{\log b}\log r
$$
so that
\bee
&& \As_0 \tilde{\Rc}_0-b(r\partial_r+2\tilde \alpha)\tilde{\Rc}_0\\
&=&-b-\frac{2b}{\log b}\log (r+1)+\Big[ r\partial_r +2+2 \tilde \alpha \Big]b\left(-\frac{2}{\ln b}T_1+\As_0^{-1}\Theta_0\right)\\
&& -\left(\left(\frac{4r}{\langle r\rangle^4}-\frac 4r\right)\partial_r +\frac{8}{\langle r\rangle^4}\right) \Rc_{1,1}-\left(\left(\frac{4r}{\langle r\rangle^4}-\frac 1r-\frac{3}{r+1}\right)\partial_r +\frac{8}{\langle r\rangle^4}+b\partial_r\right) \Rc_{1,2}(r+1)\\
&&+2b\tilde \alpha (\Rc_{0,1}+\Rc_{0,2}(r+1))
\eee
In the right hand side, the first line is $\Oc(br^{-1}\log r)$ from \fref{est:T0iatinf}, and so is the second from the definition of $\Rc_{0,1}$ and $\Rc_{0,2}$. For the last line, $\|\Rc_{0,1}+\Rc_{0,2}\|_{\Ic_{-1}^0}\lesssim 1$ and $|\tilde \alpha|\lesssim |\log b|^{-1}$. Therefore the right hand side is of size $|\log b|^{-1}$ in $\Ic_{-1}^0$, and we get $\|\tilde \Rc_0\|_{\Ic_{-1}^0}+\| b\partial_b \tilde \Rc_0\|_{\Ic_{-1}^0}\lesssim |\log b|^{-1}$ and $\| \partial_{\bar \alpha} \tilde \Rc_0\|_{\Ic_{-1}^0}\lesssim 1$.\\

\noindent \textbf{Step 5} \emph{Number of zeros}: For the case $n=0$, the identity \fref{eq:forMinn} gives with \fref{id:recurrencecnj}, \fref{def:T0i} and \fref{def:psi01}:
\begin{equation}\label{eq:expphi0n}
\phi_0^\inn(r) = \frac{r^2}{\langle r \rangle^4} +b\left(-\frac{2}{\ln b}T_1+\As_0^{-1}\Theta_0\right)+ 2\bar{\alpha}b \Big(-T_{1} + S_{0}\Big)  + b\Rc_{0} .
\end{equation}
From the pointwise bounds \fref{bd:improveddecayTheta0T0}, \fref{est:T0iatinf}, \fref{id:refinementSjn=0} and \fref{id:refinementRn=0}, and $|\bar \alpha|\lesssim |\ln b|^{-2}$ we infer:
$$
\left|b\left(-\frac{2}{\ln b}T_1+\As_0^{-1}\Theta_0\right)+ 2\bar{\alpha}b \Big(-T_{1} + S_{0}\Big)  + b\Rc_{0}\right|<\frac{r^2}{\langle r \rangle^4}
$$
on $(0,R_0]$ for $\zeta_0$ small enough and $b$ small enough, so $\phi_0^\inn(r) $ has no zero. For $n\geq 1$ one has that, from the identities \fref{eq:forMinn} and \fref{id:recurrencecnj}, the bounds \fref{bd:improveddecayTheta0T0}, \fref{est:T0iatinf}, \fref{est:S0j}, \fref{est:R0n}, \fref{id:refinementRn=1}, \fref{id:refinementSjn=1} and $|\bar \alpha|\lesssim |\ln b|^{-2}$:
\be \label{estimationphiinnperturbation}
\phi^\inn_n=T_0(r)+2nbT_1(r)+\tilde \phi^\inn_n(r), \quad \mbox{ with } \quad |\tilde \phi^\inn_n (r)|+ r |\partial_r \tilde \phi^\inn_n (r)| \leq  b\zeta_0^2 r^2\langle r\rangle^{-2} \langle \ln \langle r\rangle \rangle,
\ee
where the bound is valid on $[0,R_0]$. We recall from \fref{est:T0iat0}:
\be \label{estimationphiinnperturbation1}
T_0(r)+2nbT_1(r)= \frac{1}{r^2}-bn\log (r) +\Oc( b+r^{-3}) \quad \mbox{as } r\rightarrow \infty,
\ee
\be \label{estimationphiinnperturbation2}
\partial_r T_0(r)+2nb\partial_r T_1(r)= \frac{-2}{r^3}-\frac{bn}{r}+\Oc( b|\ln r|r^{-2}+r^{-4}) \quad \mbox{as } r\rightarrow \infty.
\ee
From the above identities, we obtain that $\phi^\inn $ vanishes exactly once on $[0,R_0]$ at the point $r_0$,
\be \label{id:defr0}
r_0=\frac{1}{\sqrt b \sqrt{n|\log b|}}(1+\Oc( \zeta_0^2),
\ee
and that there exists a constant $c>0$ such that
\be \label{as:phiinnr0}
\frac{c(r_0-r)}{r_0r^2} \leq \phi^\inn_n (r)\leq  \frac{(r_0-r)}{c r_0r^2} \mbox{ on } [1,r_0], \quad \frac{(r_0-r)}{c r_0r^2} \leq \phi^\inn_n (r)\leq \frac{c(r_0-r)}{r_0r^2} \mbox{ on } [r_0,R_0].
\ee

\end{proof}

\begin{lemma} \label{lem:perturbationinner}

Let $V$ be a smooth function satisfying $|\pa_r^k V|\lesssim |\ln b|^{-1}r^{2-k}\langle r \rangle^{-4}$ for $k=0,1$. Then for any fixed $n$, for $\zeta_0$ small enough, there exists $b^*>0$ such that for all $0<b<b^*$ and $\tilde \alpha=\Oc( |\ln b|^{-1})$, there exists a solution $\phi^{\inn, V}_n$ to
$$
\As_0 \phi^{\inn, V}_n-b\big[r\pa_r+2(1-n+\tilde \alpha)\big]\phi^{\inn, V}_n +r^{-1}\pa_r (V\phi^{\inn, V}_n)=0
$$
on $[0,R_0]$ which satisfies
\be \label{bd:perturbationinner}
\| \phi^{\inn, V}_n-\phi^{\inn}_n \|_{\Ic^{-2}_0}\lesssim \frac{1}{|\ln b|}.
\ee

\end{lemma}

\begin{proof}

We only treat the case $n\geq 1$. Indeed, from Lemma \ref{lemm:inn0}, $\phi^\inn$ vanishes once on $[0,R_0]$ for $n\geq 1$ at the point $r_0$ defined by \fref{id:defr0}, whereas for $n=0$ it does not. Reintegrating the Wronskian relation is then harder in the case $n=1$, and the case $n=0$ can be treated with the very same ideas but simpler computations. We shall use results on $\phi^\inn$ proved in the "\emph{Number of zeros}" part of the proof of Lemma \ref{lemm:inn0}.\\

\noindent \textbf{Step 1} \emph{Uniform asymptotic for the second fundamental solution}: We claim that there exists $\Gamma$ another linearly independent solution to
$$
\As_0 \Gamma-b(r\pa_r+2(1-n+\tilde \alpha))\Gamma =0
$$
on $[0,R_0]$ such that:
\be \label{bd:controlegammamodif}
|\Gamma(r)|\leq C \mbox{ and } |\pa_r\Gamma (r)|\leq Cr|\ln r|\langle r\rangle^{-2}\langle \ln r \rangle^{-1} \mbox{ on } [0,R_0]
\ee
with a constant $C$ that is independent of $b$ and $\tilde \alpha$. Indeed, from standard arguments, the Wronskian $W=\Gamma'\phi_n^\inn-\Gamma \phi_n^{\inn '}$ is (fixing the integration constant without loss of generality):
\be \label{id:perturbationwronskian}
W=\frac{r}{(1+r^2)^2}e^{b\frac{r^2}{2}}
\ee
so the second fundamental solution is given by, reintegrating the Wronskian relation (we again fix here an integration constant without loss of generality):
$$
\Gamma (r)=\phi^\inn (r) \int_1^r \frac{W(\xi)}{|\phi^\inn (\xi)|^2}d\xi=\phi^\inn (r) \int_1^r \frac{\xi e^{b\frac{\xi^2}{2}}}{(1+\xi^2)^2|\phi^\inn (\xi)|^2}d\xi.
$$
The asymptotic near the origin follows from \fref{eq:forMinn}, \fref{estimationphiinnperturbation} and \fref{est:T0iat0}, and direct computations, so we only focus on the asymptotic of $\Gamma$ for $r$ large. For $1\leq r \leq r_0$ from \fref{as:phiinnr0}:
$$
|\Gamma (r)|\lesssim \frac{(r_0-r)r_0}{r^2}\int_1^r \frac{\xi}{(r_0-\xi)^2}d\xi \lesssim 1.
$$
Next, for $r\geq r_0$, we avoid the singularity in the integral by noticing that there exists a constant $C$ such that
$$
\Gamma (r)=C\phi^\inn (r)+\phi^\inn (r)\int_{R_0}^r \frac{W(\xi)}{|\phi^{\inn}(\xi)|^2}d\xi .
$$
To estimate $C$, one computes from the first formula for $\Gamma$ and the asymptotic \fref{as:phiinnr0} near $r_0$ of $\phi^\inn$:
$$
\Gamma'(r_0)=\lim_{r\uparrow r_0}\left(\partial_r\phi^\inn (r) \int_1^r \frac{W(\xi)}{|\phi^\inn (\xi)|^2}d\xi+ \frac{W(r)}{\phi^\inn (r)}\right)=\Oc( r_0^{-1}).
$$
Similarly, we have
$$
\Gamma'(r_0)=C(\phi^\inn (r_0))'+\lim_{r\downarrow r_0}\left(\partial_r\phi^\inn (r) \int_{R_0}^r \frac{W(\xi)}{|\phi^\inn (\xi)|^2}d\xi+ \frac{W(r)}{\phi^\inn (r)}\right)=C(\phi^\inn (r_0))'+\Oc( r_0^{-1}).
$$
As $\partial_r \phi^\inn (r_0)=-2r_0^{-3}(1+\Oc( 1))$ we obtain $C=\Oc(r_0^{2})=\Oc( b^{-1}|\log b|^{-1})$. For all $r_0\leq r \leq R_0$ we find from \fref{as:phiinnr0}:
$$
\left|\phi^\inn (r)\int_{R_0}^r \frac{W(\xi)}{|\phi^{\inn}(\xi)|^2}d\xi\right|\lesssim \frac{(r-r_0)r_0^2}{r}\int_{r}^{R_0}\frac{d\xi}{(\xi-r_0)^2\xi}\lesssim 1,
$$
and:
$$
|C\phi^\inn (r)|\lesssim r_0^2\frac{(r-r_0)}{r_0^2r^{-1}}\lesssim 1.
$$
Hence $|\Gamma (r)|\lesssim 1$ for $r_0< r\leq R_0$ as well. This proves \fref{bd:controlegammamodif} for $\Gamma$. The proof for $\partial_r \Gamma$ is verbatim the same so that we skip it.\\

\noindent \textbf{Step 2} \emph{Bound for the resolvant under orthogonality condition}: Let a solution to $\As_{b,\tilde \alpha}u=r^{-1}\pa_r (Vf)$ be given by
$$
u(r) =  \phi^\inn(r) \int_r^{R_0} \frac{\Gamma (\xi)}{W(\xi)} \xi^{-1} \partial_\xi (Vf)(\xi) d\xi + \Gamma (r) \int_0^r \frac{\phi^\inn(\xi)}{W(\xi)}  \xi^{-1} \partial_\xi (Vf)(\xi) d\xi,
$$
then we claim the resolvent bound:
\be \label{bd:solventperturb}
\| u \|_{\Ic_0^{-2}}\lesssim \frac{1}{|\ln b|} \left(\| f \|_{X^{-2}_{1}}+\|r \pa_r f \|_{X^{-2}_{1}}\right).
\ee
We now prove this claim. From the hypothesis on $V$, \fref{estimationphiinnperturbation} and \fref{bd:controlegammamodif}, the first term can be bounded by
\begin{align*}
&\left| \phi^\inn(r) \int_r^{R_0} \frac{\Gamma (\xi)}{W(\xi)} \xi^{-1} \partial_\xi (Vf)(\xi) d\xi \right|  \lesssim  \frac{1}{|\ln b|}r^2\left(\langle r \rangle^{-4}+b\langle r \rangle^{-2} \langle \ln \langle r\rangle \rangle \right)\int_r^{R_0} \left(|\xi|^{-1}|f(\xi)|+|\partial_\xi f(\xi)|\right)d\xi \\
\quad & \lesssim  \frac{\| f \|_{X^{-2}_{1}}+\|r \pa_r f \|_{X^{-2}_{1}}}{|\ln b|}r^2\left(\langle r \rangle^{-4}+b\langle r \rangle^{-2} \langle \ln \langle r\rangle \rangle \right)\int_r^{R_0} \xi \langle \xi \rangle^{-4}\langle \ln \langle \xi \rangle \rangle d\xi \\
\quad & \lesssim  \frac{\| f \|_{X^{-2}_{1}}+\|r \pa_r f \|_{X^{-2}_{1}}}{|\ln b|}r^2\left(\langle r \rangle^{-6}\langle \ln \langle \xi \rangle \rangle+b\langle r \rangle^{-4} \langle \ln \langle r\rangle \rangle^2 \right) \lesssim   \frac{\| f \|_{X^{-2}_{1}}+\|r \pa_r f \|_{X^{-2}_{1}}}{|\ln b|} r^2 \langle r \rangle^{-4}.
\end{align*}
For the second term, we use the decomposition \fref{estimationphiinnperturbation}, the identities \fref{def:psi01} and \fref{id:perturbationwronskian}, the bound \fref{bd:controlegammamodif} and the bounds on $V$ to get
\bee
&& \left| \Gamma^\inn (r) \int_0^r \frac{\phi^\inn(\xi)}{W(\xi)}  \xi^{-1} \partial_\xi (Vf)(\xi) d\xi\right| = \left| \Gamma^\inn (r) \int_0^r  \frac{\langle \xi \rangle^4e^{-b\frac{r^2}{2}}}{\xi^2}\left(\frac{\xi^2}{\langle \xi \rangle^4}+bT_1+\tilde \phi^\inn\right) \partial_\xi (Vf)(\xi) d\xi \right|\\
&= & \left| \Gamma^\inn (r)\left(V(r)f(r)e^{-\frac{br^2}{2}}-\int_0^r b\xi Vf e^{-\frac{b\xi^2}{2}}d\xi +\int_0^r  \frac{\langle \xi \rangle^4e^{-b\frac{r^2}{2}}}{\xi^2}\left(bT_1+\tilde \phi^\inn\right) \partial_\xi (Vf)(\xi) d\xi \right)\right| \\
&\lesssim &\frac{1}{|\ln b|} \left( r^{-2}\langle r \rangle^{-4} |f(r)| +b \int_0^r \xi^3 \langle \xi \rangle^{-4} |f|d\xi +b \int_0^r \xi \langle \xi\rangle^{-2}\langle \ln \langle \xi \rangle \rangle (|f|+\xi|\partial_\xi f|) d\xi \right)\\
&\lesssim &\frac{\| f \|_{X^{-2}_{1}}+\|r \pa_r f \|_{X^{-2}_{1}}}{|\ln b|} \left( r^4 \langle r \rangle^{-8}+b\int_0^r \xi^5 \langle \xi \rangle^{-8}d\xi+b\int_0^r \xi^3 \langle \xi \rangle^{-6}\langle \ln \langle \xi \rangle \rangle d\xi  \right)\\
&\lesssim &\frac{\| f \|_{X^{-2}_{1}}+\|r \pa_r f \|_{X^{-2}_{1}}}{|\ln b|} r^2 \langle r \rangle^{-4}\\
\eee
because $r\lesssim b^{-1}$. Combining the above two bounds yields the following estimate on $[0,R_0]$,
$$
|u(r)|\lesssim \frac{\| f \|_{X^{-2}_{1}}+\|r \pa_r f \|_{X^{-2}_{1}}}{|\ln b|} r^2 \langle r \rangle^{-4}.
$$
Differentiating the identity satisfied by $u$ yields
$$
\partial_r u = \partial_r \phi^\inn(r) \int_r^{R_0} \frac{\Gamma (\xi)}{W(\xi)} \xi^{-1} \partial_\xi (Vf)(\xi) d\xi + \partial_r \Gamma (r) \int_0^r \frac{\phi^\inn(\xi)}{W(\xi)}  \xi^{-1} \partial_\xi (Vf)(\xi) d\xi.
$$
Hence, computing the same way the integral terms as we just did, and using \fref{estimationphiinnperturbation}, \fref{estimationphiinnperturbation2} and \fref{bd:controlegammamodif} we get
$$
|\partial_r u(r)|\lesssim \frac{\| f \|_{X^{-2}_{1}}+\|r \pa_r f \|_{X^{-2}_{1}}}{|\ln b|} r \langle r \rangle^{-4}.
$$
Using the definition of $\As_0$, we write 
$$
\partial_{r}^2 u = \left(\frac 1 r -\frac Qr\right)\partial_r u -\frac{\partial_r Q}{r}u+b(r\partial_r +2(1-n+\tilde \alpha))u+r^{-1}\partial_r (Vf),
$$
from which and the hypotheses on $V$ and the bounds on $u$ and $\partial_r u$, we obtain
$$
|\partial_{r}^2 u|\lesssim \frac{\| f \|_{X^{-2}_{1}}+\|r \pa_r f \|_{X^{-2}_{1}}}{|\ln b|} \langle r \rangle^{-4}.
$$
The bounds on $u$, $\pa_r u$ and $\pa_{r}^2 u$ imply \fref{bd:solventperturb}.\\

\noindent \textbf{Step 3} \emph{Fixed point}: We look for a solution to $\big[\As_0-b(r\pa_r+2(1-n+\tilde \alpha))-r^{-1}\pa_r(V\cdot)\big]\phi^{\inn, V}=0$ under the form
$$
\phi^{\inn,V}=\phi^{\inn}+\tilde \phi^{\inn,V}.
$$
Then, $\tilde \phi^{\inn, V}$ solves
$$
\big[\As_0-br\pa_r+2b(1-n+\tilde \alpha)\big]\tilde \phi^{\inn, V}=r^{-1}\pa_r (V\phi^{\inn})+r^{-1}\pa_r (V\tilde \phi^{\inn,V}).
$$
We solve this using a fixed point argument in $\Ic^{-2}_0$. As $\| \phi^\inn\|_{\Ic^{-2}_1}\lesssim 1$ from Lemma \ref{lemm:inn0}, as $\| \cdot \|_{\Ic^{-2}_1}\lesssim \| \cdot \|_{\Ic^{-2}_0}$ from the very definition of these spaces, the bound \fref{bd:solventperturb} implies
\begin{align*}
&\| \big[\As_0-br\pa_r+2b(1-n+\tilde \alpha)\big]^{-1} (r^{-1}\pa_r (V\tilde \phi^{\inn}))\|_{\mathcal I^{-2}_0}\\
& \qquad  \lesssim \frac{1}{|\ln b|} (\| \phi^{\inn}\|_{X^{-2}_1}+\| r\partial_r \phi^{\inn}\|_{X^{-2}_1})\lesssim \frac{1}{|\ln b|} \| \phi^{\inn}\|_{\Ic^{-2}_1}\lesssim \frac{1}{|\ln b|},
\end{align*}
$$
\left\| \As_0-br\pa_r+2b(1-n+\tilde \alpha)\big]^{-1} (r^{-1}\pa_r (V\tilde \phi^{\inn,V}))\right\|_{\mathcal I^{-2}_0}\lesssim \frac{\| \tilde \phi^{\inn,V}\|_{\Ic^{-2}_1}}{|\ln b|}\lesssim \frac{\| \tilde \phi^{\inn,V}\|_{\Ic^{-2}_0}}{|\ln b|}.
$$
Hence, the mapping which to $\tilde \phi^{\inn, V}$ assigns
$$
\big[\As_0-br\pa_r+2b(1-n+\tilde \alpha)\big]^{-1}\left(r^{-1}\pa_r (V\phi^{\inn})+r^{-1}\pa_r (V\tilde \phi^{\inn,V})\right)
$$
is a contraction in $B_{\Ic^{-2}_0}(0,C|\ln b|^{-1})$ for $C$ large enough and then for $b$ small enough. Its unique fixed point is the desired solution, and satisfies the conclusion of the lemma.
\end{proof}

\subsection{Analysis in the outer zone $r \geq R_0$} 
In this part we solve problem \fref{eq:phi} in the interval $[R_0, \infty)$ where the potential term can be treated as a small perturbation. To this end, we rewrite equation \fref{eq:phi} as 
\begin{equation}\label{eq:Mout}
\partial_r^2 \phi + \frac{3}{r}\partial_r \phi - br\partial_r \phi - \alpha \phi - \frac{4}{r(1 + r^2)}\partial_r \phi + \frac{8}{(1 + r^2)^2}\phi = 0.
\end{equation} 
Introducing the change of variable 
\begin{equation}\label{def:qMout}
\phi^\out(r) = q(z) \quad \textup{with} \;\; z = \frac{b r^2}{2},
\end{equation}
yields the equation satisfied by $q$,
\begin{equation}\label{eq:qout}
\big(\Kc_\theta + P_0\big) q(z) = 0, \quad z \geq z_0=\frac{\zeta_0^2}{2}, \quad \theta = \frac{\alpha}{2b},
\end{equation}
where $\Kc_\theta$ is a Kummer type operator defined by 
\begin{equation}
\Kc_\theta  = z\partial_{z}^2 + (2 - z)\partial_z  - \theta, \label{def:Acnu}
\end{equation}
and $P_0$ is the potential
\begin{equation}
P_0  = -\frac{2b}{(b + 2z)}\partial_z + \frac{4b}{(b+ 2z)^2}.\label{def:P0}
\end{equation}
We will treat the differential operator $P_0$ as a perturbation of $\Kc_\theta$ in the outer zone. We first claim the following.

\begin{lemma}[Properties of $\Kc_\theta$] \label{lemm:proAnu} $\quad$
\begin{itemize}
\item[(i)] \textup{(Inversion)} Assume that $-\theta \not \in \mathbb{N}$, then an explicit inversion of $\Kc_\theta$ is given by 

\begin{equation}\label{def:invAnu}
\Kc_\theta^{-1}f :=  h_\theta(z) \int_{z_0}^{z} \tilde h_\theta(\xi) f(\xi) \xi e^{-\xi} d\xi +  \tilde h_\theta(z) \int_z^{\infty} h_\theta(\xi) f(\xi) \xi e^{-\xi} d\xi,
\end{equation}

where $h_\theta$ and $\tilde h_\theta$ are the two linearly independent solutions to Kummer's equation $\Kc_\theta  h = 0$:
\begin{align}
& h_\theta(z) = \frac{1}{z\Gamma(\theta)} + \frac{1}{\Gamma(\theta - 1)}\sum_{i = 0}^\infty\frac{(\theta)_i}{(2)_i i!} z^i \Big[\ln z + \Psi(\theta + i) - \Psi(1 + i) - \Psi(2+i) \Big],\label{exp:varphi1at0}\\
& \qquad \qquad \tilde h_\theta(z) = \sum_{i = 0}^\infty\frac{(\theta)_i}{(2)_i i!} z^i, \label{exp:varphi0at0}
\end{align}
where $(a)_i = \frac{\Gamma(a + i)}{\Gamma(a)}$, $\Gamma$ is the Gamma function, and $\Psi = \Gamma'/\Gamma$ is the digamma function. Moreover, we have the asymptotic behavior as $z \to \infty$,
\begin{equation} \label{eq:asypvarphi01}
 h_\theta(z) = z^{-\theta}\big(1 + \Oc(z^{-1})\big), \quad \tilde h_\theta(z) = \frac{\Gamma(2)}{\Gamma(\theta)}e^z z^{\theta - 2} \big(1 + \Oc(z^{-1})\big). 
\end{equation}
and for $z_0\leq z\leq 2$, for $C$ dependent of $z_0$ if $n=0$ and independent if $n\geq 1$:
\be \label{eq:asypvarphi022}
|h_\theta (z)|+|\tilde h_\theta (z)|\lesssim C
\ee
\item[(ii)] \textup{(Continuity)} Let $a \in \Rb$, and $\Ec_0^{p,a}$ be the Banach space of functions $f: [z_0, \infty) \to \Rb$ equipped with the norm

\begin{align*}
\|f\|_{\Ec^{a}} &:=  \sup_{z \geq z_0} \langle z\rangle^{-a} \big(|f(z)| + |z\partial_z f(z)|  + |z^2\partial_z^2 f(z)|\big).
\end{align*}

Then for any continuous function $f:[z_0, \infty) \to \Rb$, we have the estimate for $a>-\theta$:
\begin{equation}\label{est:conAnu}
\|\Kc_\theta ^{-1}f\|_{\Ec^{a}} \lesssim C(z_0) \sup_{z_0 \leq  z < \infty} \zj^{-a} |f(z)|.
\end{equation}

\end{itemize}
 
\end{lemma} 
\begin{proof} $(i)$ See formulas 13.1.2, 13.1.6 and 13.1.22 in \cite{ASbook92} for the definition of $h_\theta$, $\tilde h_\theta$ and the Wronskian $W(h_\theta, \tilde h_\theta)$ respectively. For the bound for $z_0\leq z \leq 2$, notice that from the Gamma function's recurrence relation and the bound on $\bar \alpha$:
\begin{align} 
\Gamma(\theta)=\frac{\Gamma (\theta+n)}{\theta (\theta+1)...(\theta+n-1)}&=\frac{\Gamma (1+\tilde \alpha)}{(1-n+\tilde \alpha) (2-n+\tilde \alpha)...(-1+\tilde \alpha)\tilde \alpha} \nonumber \\
& \quad \sim \frac{(-1)^n}{(n-1)!\tilde \alpha}+\Oc( 1)=\Oc( |\ln b|), \label{id:asymptotiquegammafunction}
\end{align}
for $n\geq 1$, and $\Gamma (\theta)=\Gamma (1+\tilde \alpha)=\Oc( 1)$ for $n=0$.

$(ii)$ The proof follows from straightforward computations. Let 
$$D =  \sup_{z_0 \leq z <\infty}  \zj^{-a}|f(z)|,$$
From \fref{eq:asypvarphi01}, we compute for $z \geq 2$, 
\begin{align*}
\left|\tilde h_\theta (z)\int_z^{\infty} f  h_\theta  \xi e^{-\xi}d\xi\right| \lesssim Dz^{\theta-2}e^z\int_z^{\infty} \xi^a \xi^{-\theta} \xi e^{-\xi}d\xi \lesssim D z^{a-1},
\end{align*}
and from \fref{exp:varphi1at0}, we compute for $z \in [z_0, 2]$,
\begin{align*}
\left|\tilde h_\theta (z)\int_z^{\infty} f  h_\theta  \xi e^{-\xi}d\xi\right| &\lesssim  \left|\tilde h_\theta(z)\int_z^2 f  h_\theta  \xi e^{-\xi}d\xi + \tilde h_\theta (z)\int_2^{\infty} f  h_\theta  \xi e^{-\xi}d\xi\right|\\
& \lesssim  D \int_z^2 \xi  d\xi + D\int_2^{\infty}\xi^{a-\theta+1}e^{-\xi}d\xi \lesssim D.
\end{align*}
Similarly, we have for $z \geq 2$, as $a>-\theta$
\begin{align*}
\left|h_\theta(z)\int_{z_0}^{z} f \tilde  h_\theta  \xi e^{-\xi}d\xi\right| &\lesssim Dz^{-\theta} \int_{z_0}^2 \xi d\xi  + Dz^{-\theta} \int_2^z \xi^a \xi^{\theta - 2}\xi^1 d\xi \lesssim D z^{a}. 
\end{align*}
and for $z_0 \leq  z \leq 2$,
$$\left|\int_{z_0}^{z} f \tilde h_\theta  \xi e^{-\xi}d\xi\right| \lesssim D\int_{z_0}^z \xi d\xi \lesssim D.$$
This proves the continuity bound \fref{est:conAnu} for $\Kc_\theta^{-1}f$. We now take derivatives. For $z \geq 2$, we estimate from \fref{def:invAnu}, \fref{eq:asypvarphi01}:
\begin{align*}
\Big| z\partial_z \Kc_\theta ^{-1}f(z)\Big| &\lesssim \left| z\partial_z \tilde  h_\theta \int_z^{\infty} f  h_\theta  \xi e^{-\xi} d\xi\right| + \left| z\partial_z  h_\theta  \int_{z_0}^z f \tilde  h_\theta  \xi e^{-\xi} d\xi\right|\\
& \lesssim D \left(e^z z^{\theta - 1} z^{a -\theta + 1}e^{-z} + z^{-\theta} z^{a + \theta }\right) \lesssim Dz^{a}. 
\end{align*}
For $z \in [z_0, 2]$, we estimate from \fref{exp:varphi0at0}:
\begin{align*}
\Big| z\partial_z \Kc_\theta ^{-1}f(z)\Big| &\lesssim \left| z\partial_z \tilde  h_\theta \int_z^{\infty} f  h_\theta  \xi e^{-\xi} d\xi\right| + \left| z\partial_z  h_\theta  \int_{z_0}^z f \tilde  h_\theta  \xi e^{-\xi} d\xi\right|\\
& \lesssim D  \big( \int_z^2 \xi d \xi+\int_2^\infty \xi^{a-\theta+1}e^{-\xi}d\xi\big) +D\int_{z_0}^z \xi d \xi \lesssim D. 
\end{align*}
Using $\Kc_\theta  \Kc_\theta^{-1} f = f$ and the definition of $\Kc_\theta $, we have the estimate for $z \geq 2$, 
\begin{align*}
|z\partial_z^2 \Kc_\theta ^{-1} f(z)| \lesssim |z\partial_z\Kc^{-1}_\theta f(z)| + |\Kc_\theta ^{-1}f(z)| \lesssim Dz^{a},
\end{align*}
and for $z \in [z_0, 2]$, 
\begin{align*}
|z^2 \partial_z^2 \Kc_\theta ^{-1}f(z)| \lesssim |z\partial_z \Kc_\theta ^{-1}f(z)| + |z\Kc_\theta ^{-1}f(z)| \lesssim D .
\end{align*}
Collecting the above estimates yields the estimate \fref{est:conAnu}.
This concludes the proof of Lemma \ref{lemm:proAnu}.
\end{proof}

We are now in the position of computing the solution $q$ to equation \fref{eq:qout} by a perturbation argument.

\begin{lemma}[Outer eigenfunctions for the radial mode] \label{lemm:out0}  Fix $n\in \mathbb N$, and $\theta=1-n+1/\ln b+\bar \alpha$. For $0 < \zeta_0 \ll 1$ and any small $0<\delta\ll1$, there exist $b^* > 0$ such that for all $0 < b \leq b^*$, for all $\bar \alpha=\Oc( |\ln b|^{-2})$ there exists a smooth solution 
\be \label{eq:qn}
q(b,\bar \alpha,z)=\Gamma (\theta)h_\theta(z)+\Gc(b,\bar \alpha,z)
\ee
to \fref{eq:qout} on $[z_0, \infty)$, where $h_\theta$ is introduced in Lemma \ref{lemm:proAnu} and $\Gc$ satisfies the following estimates for some universal $C>0$:
\begin{equation}\label{est:Gc}
\|\Gc\|_{\Ec^{-\theta+\delta}} \lesssim b|\ln b|^C, \quad \| b \partial_b \Gc\|_{\Ec^{ -\theta+\delta}} \lesssim  b|\ln b|^C, \quad \|\partial_{\theta} \Gc\|_{\Ec_0^{-\theta+\delta}} \lesssim b|\ln b|^C.
\end{equation}
where the constants in the estimates depend on $z_0$. Finally, on the interval $[z_0,\infty)$, $q$ does not vanish for $n=0,1$, while for $n\geq 2$ it possesses $n-1$ zeros.

\end{lemma}

\begin{lemma} \label{lem:perturbationouter}
Assume $P_0$ is replaced by $P_0(q)+\frac 12 \partial_z (\tilde V q)/z$ where $\tilde V$ satisfies $|\tilde V|+|z\partial_z \tilde V|\lesssim b|\ln b|^{-1}z^{-1}$ on $[z_0,\infty)$. Then existence result of Lemma \ref{lemm:out0} of a solution $q^V=\Gamma (\theta)h_\theta(z)+\Gc^V(b,\bar \alpha,z)$ and the first bound in \fref{est:Gc} still hold true.
\end{lemma}

\begin{proof}[Proof of Lemma \ref{lemm:out0}] From the bound on the Gamma function \fref{id:asymptotiquegammafunction}, we will simply consider a solution of the form $q(z)=h_\theta(z)+\Gc(b,\bar \alpha,z)$ (with the abuse of notation of keeping the notation $\Gc$), and prove the estimate \fref{est:Gc} for $\Gc(b,\bar \alpha,z)$, which will prove the Lemma upon multiplication by $\Gamma (\theta)$.\\

\noindent \textbf{Step 1} \emph{Existence and bounds}: Note that $P_0$ has the form:
\be \label{estimationP0}
P_0(q)= V_1q+V_2\partial_z q, \quad \mbox{with } |V_1|+|zV_1|\lesssim bz^{-1}.
\ee
Let us write from \fref{eq:qn} the equation satisfied by $\Gc$,
\begin{align*}
\Kc_\theta  \Gc +  P_0 \Gc + P_0 h_\theta  = 0.
\end{align*} 
Let $\Gamma$ the affine mapping defined as 
$$\Gamma(f) = -\Kc_\theta ^{-1}\big[P_0 f +  P_0 h_\theta \big] \equiv D\Gamma(f) + D\Gamma( h_\theta ),$$
where 
$$D\Gamma(f) = -\Kc_\theta ^{-1}\big[P_0f\big],$$
and $\Kc_\theta ^{-1}$ is given by \fref{def:invAnu}. We estimate from the definition \fref{def:P0} of $P_0$ and \fref{est:conAnu}, 
\begin{align*}
\|D\Gamma( h_\theta )\|_{\Ec^{ - \theta+\delta}} \lesssim \| \Kc_\theta ^{-1}P_0  h_\theta \|_{\Ec^{-\theta+\delta}} \lesssim \sup_{z \in [z_0, \infty)}\zj^{\theta-\delta}|P_0  h_\theta | \lesssim b.
\end{align*}
From \fref{est:conAnu}, we estimate for all $f \in \Ec^{0, -\theta+\delta}$,
\begin{equation}\label{est:DGammaEc0nu}
\|D\Gamma(f)\|_{\Ec^{-\theta+\delta}} \lesssim \sup_{z \in [z_0, \infty)}\zj^{\theta-\delta}|P_0f(z)| \lesssim  b \|f\|_{\Ec^{-\theta+\delta}}.
\end{equation}
It follows that $\Gamma$ is a contraction mapping on $B_{\Ec^{-\theta+\delta}}(0, M b)$ for some $M=M(\zeta_0)>0$ large enough. Hence, there exists a unique fixed point $\Gc$ with
$$\Gc = \Gamma(\Gc) \quad \textup{with} \quad \|\Gc\|_{\Ec^{-\theta+\delta}} \lesssim b.$$ 
Differentiating the above fixed point relation yields:
\begin{align*}
\partial_\theta \Gc = D\Gamma(\partial_\theta \Gc) + (\partial_\theta \Gamma) (\Gc), \quad \partial_b \Gc = D\Gamma(\partial_b \Gc) + (\partial_b \Gamma) (\Gc).
\end{align*}
Since $P_0$ depends on $b$ and not on $\theta$, whereas $h_\theta$, $h'_\theta$ and $\mathcal K_\theta$ depend on $\theta$ and not on $b$, we have the identities:
\begin{align*}
(\partial_\theta \Gamma)(\Gc) =-\partial_\theta ( \Kc_\theta ^{-1})(P_0(\Gc+h_\theta))-\Kc^{-1}(P_0\partial_\theta h_\theta) , \quad (\partial_b \Gamma)(\Gc) = -\Kc_\theta ^{-1}(\partial_b P_0 (\Gc+h_\theta)). 
\end{align*}
We compute from \fref{exp:varphi1at0} that:
\bee
\partial_\theta h_\theta(z) &=& -\frac{\Psi (\theta)}{z\Gamma(\theta)} - \frac{\Psi (\theta-1)}{\Gamma(\theta - 1)}\sum_{i = 0}^\infty\frac{(\theta)_i}{(2)_i i!} z^i \Big[\ln z + \Psi(\theta + i) - \Psi(1 + i) - \Psi(2+i) \Big]\\
&&+\frac{1}{\Gamma(\theta - 1)}\sum_{i = 0}^\infty\frac{(\theta)_i}{(2)_i i!} z^i \Big[\ln z + \Psi'(\theta + i) - \Psi(1 + i) - \Psi(2+i) \Big]\\
&&+\frac{1}{\Gamma(\theta - 1)}\sum_{i = 0}^\infty\frac{(\theta)_i(\Psi(\theta+i)-\Psi(\theta))}{(2)_i i!} z^i \Big[\ln z + \Psi(\theta + i) - \Psi(1 + i) - \Psi(2+i) \Big].
\eee
Hence, we infer from $|\Psi (\theta)|+|\Psi (\theta-1)|\lesssim |\tilde \alpha|^{-1}\lesssim |\ln b|$, $|\Psi'(\theta+i)|\lesssim |\tilde \alpha|^{-2}\lesssim |\ln b|^2$ and $|(\theta)_i(\Psi(\theta+i)-\Psi(\theta))|\lesssim 1$ the rough upper bound on $[z_0,\infty)$:
$$
|\partial_\theta h_\theta (z)|\lesssim |\ln b|^2 z^{-1} \ln \langle z \rangle \langle z\rangle^{-\theta},
$$
which extends to derivatives. Similarly, we have from \fref{exp:varphi0at0}
$$
\partial_\theta \tilde h_\theta (z)= \sum_{i = 0}^\infty\frac{(\theta)_i(\Psi (\theta+i)-\Psi (\theta))}{(2)_i i!} z^i,
$$
satisfies the rough upper bound $|\partial_\theta \tilde h_\theta (z)|\lesssim \ln \langle z \rangle \langle z \rangle^{\theta-2}e^{z}$ on $[z_0,\infty)$. We get from \fref{def:invAnu}:
\bee
(\partial_\theta \Kc_\theta^{-1})f  & = &  (\partial_\theta h_\theta)(z) \int_{z_0}^{z} \tilde h_\theta(\xi) f(\xi) \xi^2e^{-\xi} d\xi +h_\theta(z) \int_{z_0}^{z} \partial_\theta \tilde h_\theta(\xi) f(\xi) \xi^2e^{-\xi} d\xi \\
&& +  \partial_\theta (\tilde h_\theta)(z) \int_z^{\infty} h_\theta(\xi) f(\xi) \xi^2e^{-\xi} d\xi+  \tilde h_\theta(z) \int_z^{\infty} \partial_\theta(h_\theta)(\xi) f(\xi) \xi^2e^{-\xi} d\xi.
\eee
Hence, as from the above, the bounds for $h_\theta$ and $\tilde h_\theta$ still hold up to a logarithmic loss in $z$ and $b$ and $\delta>0$, using the same argument as in the proof of Lemma \ref{lemm:proAnu} we get:
$$
\| (\partial_\theta \Kc_\theta ^{-1})f\|_{\Ec^{-\theta+\delta}} \lesssim |\ln b|^2 \sup_{z_0 \leq  z < \infty} \zj^{\theta-\frac \delta 2 } |f(z)|
$$
and from \fref{est:conAnu}:

$$
\| \Kc^{-1}(P_0\partial_\theta h_\theta) \|_{\Ec^{-\theta+\delta}}\lesssim \| P_0\partial_\theta h_\theta \|_{\Ec^{-\theta+\delta}} \lesssim b|\ln b|^2.
$$

Thus, as $\delta$ is small, from the definition of $P_0$:
$$
\| \partial_\theta ( \Kc_\theta ^{-1})(P_0(\Gc+h_\theta))\|_{\Ec^{-\theta+\delta}}\lesssim |\ln b|^2\| P_0(\Gc+h_\theta) \|_{\Ec^{-\theta+\frac \delta 2}}\lesssim b|\ln b|^2 \| \Gc+h_\theta \|_{\Ec^{-\theta+\delta}}\lesssim b|\ln b|^2.
$$
We proved above the continuity bound $\| D\Gamma\|_{C(\Ec^{-\theta+\delta})}\lesssim b$ and the identity,
$$
\partial_\theta \Gc = D\Gamma(\partial_\theta \Gc) -\partial_\theta ( \Kc_\theta ^{-1})(P_0(\Gc+h_\theta))-\Kc^{-1}(P_0\partial_\theta h_\theta).
$$
Hence one can invert the operator $Id-D\Gamma$ for $b$ small enough, with $\|Id+D\Gamma\|_{C(\Ec^{-\theta+\delta})}\lesssim 1$ and the above identity gives:
$$
\| \partial_\theta \Gc \|_{\Ec^{-\theta+\delta}}=\|(Id-D\Gamma)^{-1}\left(\partial_\theta ( \Kc_\theta ^{-1})(P_0(\Gc+h_\theta))+\Kc^{-1}(P_0\partial_\theta h_\theta)\right)\|_{\Ec^{-\theta+\delta}}\lesssim b|\ln b|^2.
$$
From the definition of $P_0$ and \fref{est:conAnu} we find:
$$
\| \Kc_\theta ^{-1}(\partial_b P_0 (\Gc+h_\theta))\|_{\Ec^{-\theta+\delta}}\lesssim \| \partial_b P_0 (\Gc+h_\theta)\|_{\Ec^{-\theta+\delta}} \lesssim  \| \Gc+h_\theta \|_{\Ec^{-\theta+\delta}}\lesssim 1.
$$
Hence we obtain similarly from the relation $\partial_b \Gc = D\Gamma(\partial_b \Gc) -\Kc_\theta ^{-1}(\partial_b P_0 (\Gc+h_\theta))$ the bound:
$$
\| \partial_b \Gc \|_{\Ec^{-\theta+\delta}}\lesssim \| (Id-D\Gamma)^{-1}\Kc_\theta ^{-1}(\partial_b P_0 (\Gc+h_\theta))\|_{\Ec^{-\theta+\delta}}\lesssim 1.
$$

\noindent \textbf{Step 2} \emph{Number of zeros}: This is a consequence of the well-known properties of Kummer's function $h_\theta$ (see\cite{ASbook92}). Since $\theta = 1 - n + \frac{1}{\ln b} + \Oc\left( |\ln b|^{-2}\right)$, $h_\theta$ has no positive zeros for $n = 0$ and possesses $ \lfloor \theta \rfloor = n - 1$ zeros on the interval $(0, +\infty)$. The estimate \eqref{est:Gc} and the asymptotic behavior \eqref{eq:asypvarphi01} ensure that there exists $z_* > 0$ such that  $|q(z)| \ne 0$ and $z \partial_z q(z) = -\theta z^{-\theta} (1 + \Oc(b)) \ne 0$  for $z \geq z_*$. Thus, $q(z)$ does not change sign for $z \geq z_*$.   It remains to show that on the interval $z \in (z_0, z_*)$, $q$  has the same number of zeros than $h_\theta$. We consider two cases. \\
- If $|h_\theta(z)| \geq c_0$ for all $z \in (z_0, z_*)$ for some $c_0 > 0$, then the estimate \eqref{est:Gc} implies that  $|q(z)| > 0$ as well on $(z_0, z_*)$ for $b$ sufficiently small. \\
- If $h_\theta(z)$ has $n-1$ zeros on $(z_0, z_*)$, say $h(z_1) = h (z_2)  = \cdots = h( z_{n - 1}) = 0$ with $z_1  < z_2 < \cdots < z_{n - 1}$. By definition, we have $|h_\theta(z)| \geq \delta_0$ on  $(z_0, z_*) \setminus \cup_{j = 1}^{n-1}B_{z_j}(\epsilon_0)$ for a fixed small constant $0 <\epsilon_0 \ll 1$ and $\delta_0 = \delta_0(\epsilon_0) > 0$. Using  \eqref{est:Gc} yields $|q(z)| > 0$ for $z \in (z_0, z_*) \setminus \cup_{j = 1}^{n-1}B_{\epsilon_0}(z_j)$ for $b$ small enough. Consider $z_j$ a zero of $h_\theta$, namely that $h_\theta(z_j) = 0$. Since $h_\theta$ is a non-zero solution of a second order differential equation, necessarily $|h_\theta'(z_j)| > 0$. We may assume that $h_\theta'(z_j) > 0$, which infers that there are $z_j^- \in (z_j - \epsilon_0, z_j)$ and $z_j^+\in (z_j, z_j + \epsilon_0)$ such that $h(z_j^-) < 0$ and $h(z_j^+)  > 0$. We then use  \eqref{est:Gc} and the intermediate value Theorem to conclude that there is $\tilde{z}_j \in (z_j^-, z_j^+)$ for which $q(\tilde{z}_j) = 0$. We also note that $|h'_\theta(z_j)| \geq c_1 > 0$ for $z \in B_{\epsilon_0}(z_j)$, from which and \eqref{est:Gc}, we deduce that $|q'(z)| \ne 0$ for $z \in B_{\epsilon_0}(z_j)$. Hence, $\tilde{z}_j$ is the only zero of $q(z)$ in $B_{\epsilon_0}(z_j)$. This concludes the proof of Lemma \ref{lemm:out0}.
\end{proof}

\begin{proof}[Proof of Lemma \ref{lem:perturbationouter}]

The decomposition \fref{estimationP0} and the associated bounds still hold for $P_0+\frac 12 \partial_z (\tilde V \cdot)/z$. This was the only information used on $P_0$ in the proof of Lemma \ref{lemm:out0}, so the very same proof applies.
\end{proof}

\subsection{Conclusion via matched asymptotic expansions, proof of Proposition \ref{pr:spectralbarAzeta}}

From Lemmas \ref{lemm:inn0} and \ref{lemm:out0}, we are now able to derive the full solution to the eigenproblem \fref{eq:phi}. In particular we claim the following.

\begin{lemma}[Matched eigenfunction for the radial mode] \label{lemm:radialmode} Fix $n\in \mathbb N$. Then there exists $C>0$, such that for $\zeta_0$ small enough, there exists $0< b^* \ll 1$ such that for all $0 < b \leq b^*$, there exists $|\bar \alpha_n|\leq C|\ln b|^{-2}$ such that the following holds for the function

\begin{equation}\label{def:Mn}
\phi_n(r) := \left\{ \begin{array}{ll}
\;\;\;\;\phi_n^\inn (r) &\textup{for} \;\; r \leq R_0,\\
\beta_0 \phi_n^\out (r) &\textup{for}\;\; r \geq R_0,
\end{array}  \right. \quad \beta_0 = \frac{\phi_n^\inn (R_0)}{\phi_n^\out (R_0)}, \quad R_0 = \frac{\zeta_0}{\sqrt b},
\end{equation}

where $\phi_n^\inn=\phi_n^\inn[b,\bar \alpha]$ and $\phi_n^\out(r)=\phi_n^\out[b,\bar \alpha](r) = q[b,\bar \alpha]\left(\frac{br^2}{2}\right)=q\left(z\right)$ are described in Lemmas \ref{lemm:inn0} and \ref{lemm:out0} respectively.

\begin{itemize}
\item[(i)] The function $\phi_n$ is a smooth solution to the equation 
\begin{equation}\label{eq:Mn_mu}
\big(\As_0  - br\partial_r\big) \phi_n = 2b\big(1 - n + \frac{1}{\ln b}+\bar \alpha_n \big)\phi_n.
\end{equation}
\item[(ii)] The estimates \fref{est:nu0ntil} and \fref{est:nurefinedn=01} for $\alpha_n$ hold true. The estimate \fref{est:PhinPointEst} for $\phi_n$ holds true.
\end{itemize}

\end{lemma}

\begin{corollary} \label{coro:perturbedspectral}
For the perturbed operator $\As_0 \phi_n - br\partial_r \phi_n+r^{-1}\partial_r (V\cdot)$ where $V$ satisfies $|\pa_r^k V|\lesssim |\ln b|^{-1}r^{2-k}\langle r \rangle^{-4}$ for $k=0,1$, then item (i) of Lemma \ref{lemm:radialmode} holds true if the inner and outer eigenfunctions are those associated to the perturbed problems described by Lemma \ref{lem:perturbationinner} and \ref{lem:perturbationouter} respectively. 
\end{corollary}

\begin{proof}[Proof of Lemma \ref{lemm:radialmode}] 

Recall from \fref{def:Acnu} the relation
\begin{equation}\label{def:nutil}
\theta = 1 - n + \tilde{\alpha} , \quad \tilde \alpha=\frac{1}{\ln b}+\bar \alpha.
\end{equation}
Since the equation \fref{eq:Mn_mu} is a second order ODE with smooth coefficients outside the origin, it suffices to prove that the two functions and their first order derivatives agree on both sides of $R_0$, and \fref{def:Mn} will then provide a global solution to \fref{eq:Mn_mu} on $(0,\infty)$. From the special choice of $\beta_0$ this is equivalent to:
\begin{equation}\label{def:Thetaeps0}
\frac{\partial_r \phi_n^\inn(R_0)}{\partial_r \phi_n^\out(R_0)} = \beta_0 \Longleftrightarrow \Theta(b,\bar{\alpha}) = \frac{(r\partial_r) \phi_n^\inn(R_0)}{2\phi_n^\inn(R_0)} - \frac{(z \partial_z) q(z_0)}{q(z_0)} = 0.
\end{equation}
Therefore, to prove the Lemma it suffices to prove that for $b$ small enough there exists $\bar{\alpha}=\bar \alpha_n (b)$ such that $\Theta(b, \bar{\alpha}) = 0$, and such that item (ii) holds true. We rely in a standard way on implicit function theorem. The estimate for $\partial_b \bar{\alpha}_n$ then follows by
\begin{equation}\label{eq:alphadevb}
\partial_b \bar{\alpha}_n = - \frac{(\partial_b \Theta)(b, \bar{\alpha}_n)}{(\partial_{\bar{\alpha}_n} \Theta)(b, \bar{\alpha}_n)}.
\end{equation}
To ease the writing, we mention only the dependence in $b$ and $\bar \alpha$ at key locations in what follows.\\

\noindent \textbf{Step 1} \emph{The interior term:} It is convenient to rewrite from \fref{eq:forMinn} the expression of $\phi_n^\inn$ as 
\begin{equation} \label{id:innerdecompositionmatching}
\phi^\inn_n[b,\bar \alpha](r) = F_n[b](r) + \bar{\alpha}b G_n[b,\bar \alpha](r) +E_n[b,\bar \alpha](r),
\end{equation}
where $F_n$ and $\bar \alpha G_n$ are leading order terms and $E_n$ is a remainder:
\begin{align}
F_n[b](r) = \sum_{j =0}^nc_{n,j}b^jT_{j}(r),  \quad G_n[b,\bar \alpha] =\sum_{j = 0}^n b^{j }\Big(-c_{n,j}T_{j + 1}(r) + S_{j}[b,\bar \alpha](r)\Big),  \label{def:FFtil}
\end{align}
$$
E_n[b,\bar \alpha](r) =b\left(-\frac{2}{\ln b}T_1(r)+\As^{-1}\Theta_0(r)\right)+b\Rc_n[b,\bar \alpha](r).
$$
We have the following estimates from \fref{est:T0iatinf}, \fref{bd:forcingweightedspaceinner}, \fref{est:R0n}, and assuming $|\bar \alpha|\lesssim |\ln b|^{-2}$:
\be \label{bd:Enmatching}
\sum_{0\leq k \leq 2, \ 0\leq \ell+\ell'\leq 1}^2| ((r \partial_r)^k(b\partial_b)^\ell \partial_{\bar \alpha}^{\ell'} E_n) (R_0)| \leq C(\zeta_0) \frac{b}{|\ln b|},\\
\ee
\begin{align}
F_n(R_0)=b\left(-\frac{\ln b}{2}H_n(\zeta_0)+K_n(\zeta_0)\right)+\Oc( b^{\frac 32}), \label{def:Hn} \\
(r\pa_rF_n(R_0))(R_0)=b\left(-\frac{\ln b}{2}\zeta \partial_\zeta H_n(\zeta_0)+\zeta \partial_\zeta K_n(\zeta_0)\right)+\Oc( b^{\frac 32}), \nonumber
\end{align} 
where $H_n$ and $G_n$ are defined by:
\be \label{def:HnGn}
H_n(\zeta_0)=\sum_{i=1}^{n}  c_{n,i} \hat d_i\zeta_0^{2(i-1)}, \ \ K_n(\zeta_0)=\frac{1}{\zeta_0^2}+\sum_{i=1}^{n} c_{n,i} \zeta_0^{2(i-1)}\left( \hat d_{i}  \ln \zeta_0 + d_{i} \right).
\ee
Notice for $0<\zeta_0\ll 1$ small that $|H_n(\zeta_0)|\neq 0$. Gathering all these estimates and \fref{est:S0j} we arrive at
\be \label{id:exprphiinR0}
\phi^\inn_n (R_0)=b\left(-\frac{\ln b}{2}H_n(\zeta_0)+K_n(\zeta_0)+\bar \alpha G_n(R_0) +\Oc(\frac{1}{|\ln b|})\right),
\ee
$$
r\partial r \phi^\inn_n (R_0)=b\left(-\frac{\ln b}{2}\zeta \partial_\zeta H_n(\zeta_0)+\zeta \partial_\zeta K_n(\zeta_0)+\bar \alpha r\partial_r G_n(R_0) +\Oc(\frac{1}{|\ln b|})\right),
$$
\begin{align*}
\nonumber \partial_b\left( \frac{1}{b\ln b } \phi^\inn_n (R_0)\right)&=-\frac{1}{b|\ln b|^2}(\zeta \partial_\zeta K_n(\zeta_0)+\bar \alpha r\partial_r G_n(R_0) +\Oc(\frac{1}{|\ln b|})) \\
& \quad +\frac{1}{b\ln b}   \left( \bar  \alpha \partial_b G_n(R_0) +\partial_b E(R_0)\right) = \Oc\left(\frac{1}{b|\ln b|^2}\right),
\end{align*}
\begin{align*}
\nonumber \partial_b \left( \frac{1}{b\ln b } r\partial r \phi^\inn_n (R_0)\right)& =-\frac{1}{b|\ln b|^2}(\zeta \partial_\zeta K_n(\zeta_0)+\bar \alpha r\partial_r G_n(R_0) +\Oc(\frac{1}{|\ln b|})) \\
& \quad   +\frac{1}{b\ln b}   \left( \bar  \alpha \partial_b r\partial_rG_n(R_0) +\partial_b r\partial_r E(R_0)\right)=\Oc\left(\frac{1}{b|\ln b|^2}\right),
\end{align*}
\begin{align*}
\nonumber \partial_{\bar \alpha} (\phi^\inn_n (R_0))&=bG_n(R_0)+\bar \alpha b \partial_{\bar \alpha} G_n(R_0)+b \partial_{\bar \alpha}E_n(R_0)\\
&=bG_n(R_0)+b\Oc( |\ln b|^{-2})\Oc(\ln b) +b\Oc(|\ln b|^{-1})=b\left(G_n(R_0)+\Oc(|\ln b|^{-1})\right),
\end{align*}
\begin{align*}
\nonumber \partial_{\bar \alpha} (r\partial_r \phi^\inn_n (R_0))&=br\partial_r G_n(R_0)+\bar \alpha b \partial_{\bar \alpha} r\partial_r G_n(R_0)+b \partial_{\bar \alpha}r\partial_r E_n(R_0)\\
&=br\partial_r G_n(R_0)+b\Oc(|\ln b|^{-2})\Oc( \ln b) +b\Oc(|\ln b|^{-1})=b\left(r\partial_r G_n(R_0)+\Oc(|\ln b|^{-1})\right).
\end{align*}
We compute that, from \fref{est:S0j}:
$$
 |r\partial_r G_n(R_0)|+|G_n(R_0)|\leq C(n) |\ln b|, \quad \mbox{with } C(n) \mbox{ independent of } \zeta_0.
$$
The collection of the above identities gives us the following leading order expression for the quantity involving the inner solution in \fref{def:Thetaeps0}:
\begin{align}
\nonumber &\frac{(r\partial_r) \phi_n^\inn(R_0)}{\phi_n^\inn(R_0)} =\frac{-\frac{\ln b}{2}\zeta \partial_\zeta H_n(\zeta_0)+\zeta \partial_\zeta K_n(\zeta_0)+\bar \alpha r\partial_r G_n(R_0) +\Oc(\frac{1}{|\ln b|})}{-\frac{\ln b}{2}H_n(\zeta_0)+K_n(\zeta_0)+\bar \alpha G_n(R_0) +\Oc(\frac{1}{|\ln b|})} \\
\nonumber &=\frac{\zeta \partial_\zeta H_n(\zeta_0)-\frac{2}{\ln b} \zeta \partial_\zeta K_n(\zeta_0)-\frac{2}{\ln b} \bar \alpha r\partial_r G_n(R_0) +\Oc( \frac{1}{|\ln b|^2})}{H_n(\zeta_0)-\frac{2}{\ln b}K_n(\zeta_0)-\frac{2}{\ln b}\bar \alpha G_n(R_0) +\Oc(\frac{1}{|\ln b|^2})} \\
\nonumber &=\frac{\zeta \partial_\zeta H_n(\zeta_0)}{H_n(\zeta_0)}+\frac{2}{\ln b}\frac{K_n(\zeta_0)\zeta \partial_\zeta H_n(\zeta_0)-H_n(\zeta_0)\zeta\partial_\zeta K_n(\zeta_0)}{H_n^2(\zeta_0)}\\
& \qquad +\frac{2}{\ln b}\bar \alpha \frac{G_n \zeta \partial_\zeta H_n(\zeta_0)-H_n(\zeta_0)\zeta \partial_\zeta G_n}{H_n(\zeta_0)^2}+\Oc(|\ln b|^{-2}) \nonumber\\
&=\frac{\zeta \partial_\zeta H_n(\zeta_0)}{H_n(\zeta_0)}+\frac{2}{\ln b}\frac{K_n(\zeta_0)\zeta \partial_\zeta H_n(\zeta_0)-H_n(\zeta_0)\zeta\partial_\zeta K_n(\zeta_0)}{H_n^2(\zeta_0)}+\bar \alpha \frac{\Oc(1)}{H_n(\zeta_0)^2}+\Oc(|\ln b|^{-2}),
\label{est:phin0D1}
\end{align}
and 
\begin{align}
\nonumber & \partial_b \left(\frac{(r\partial_r) \phi_n^\inn(R_0)}{\phi_n^\inn(R_0)}\right) = \partial_b \left(\frac{(b\ln b)^{-1}(r\partial_r) \phi_n^\inn(R_0)}{(b\ln b)^{-1}\phi_n^\inn(R_0)}\right)\\
\nonumber &= \frac{ \partial_b ((b\ln b)^{-1}(r\partial_r) \phi_n^\inn(R_0))(b\ln b)^{-1}\phi_n^\inn(R_0)-\partial_b (b\ln b)^{-1}\phi_n^\inn(R_0)b\ln b)^{-1}(r\partial_r) \phi_n^\inn(R_0) }{((b\ln b)^{-1}\phi_n^\inn(R_0))^2} \\
\label{est:phin0D1pab} &= \frac{\Oc(b^{-1}|\ln b|^{-2}) }{((b\ln b)^{-1}\phi_n^\inn(R_0))^2} =\Oc\left(\frac{1}{b|\ln b|^2}\right),
\end{align}
and 
\begin{align}
\nonumber & \partial_{\bar \alpha} \left(\frac{(r\partial_r) \phi_n^\inn(R_0)}{\phi_n^\inn(R_0)}\right) = \frac{ \partial_{\bar \alpha }r\partial_r \phi_n^\inn(R_0)\phi_n^\inn(R_0)-\partial_{\bar \alpha}\phi_n^\inn(R_0)\partial_r \phi_n^\inn(R_0)}{|\phi_n^\inn(R_0)|^2}\\
\nonumber &= \frac{ \left(\zeta \partial_\zeta G_n(R_0)+\Oc(|\ln b|^{-1})\right) \left(-\frac{\ln b}{2}H_n(\zeta_0)+\Oc(1)\right)-\left(G_n(R_0)+\Oc( |\ln b|^{-1})\right)\left(-\frac{\ln b}{2}\zeta\partial_\zeta H_n(\zeta_0) +\Oc(1)\right)}{\left(-\frac{\ln b}{2}H_n(\zeta_0)+\Oc(1)\right)^2} \\
\label{est:phin0D1patheta} &=\frac{2}{\ln b} \frac{G_n \zeta \partial_\zeta H_n(\zeta_0)-\zeta \partial_\zeta G_n H_n(\zeta_0)}{H_n^2(\zeta_0)}=\Oc(1)
\end{align}
where the constant in the two $\Oc(1)$ above are independent of $\zeta_0$. \\

\noindent \underline{The case $n=1$}: Injecting $\bar \alpha=e_1/|\log b|^2+\hat \alpha$, $|\hat \alpha|\lesssim |\ln b|^{-3}$ in the refined asymptotics \fref{id:refinementRn=1} and \fref{id:refinementSjn=1} gives
$$
\phi^\inn_1(r)=F_1(r)+\hat \alpha b G_1(r)+E_1(r),
$$
where
\begin{align*}
F_1(r) &=T_0(r)+2bT_1(r)+b\left(-\frac{2}{\ln b}T_1(r)+\As_0^{-1}\Theta_0\right)\\
& \quad +\frac{2e_1}{|\ln b|^2}(-bT_1(r)-2b^2T_2(r)-\frac{b^2}{2}\sum_{i=2}^{\infty}\frac{(1)_{i-1}}{(2)_ii!2^i}b^{i-1} r^{2i}\ln (r+1))\\
& \quad -\frac{b}{2}\sum_{i=1}^{\infty} \frac{(1)_{i-1}}{(2)_ii!2^i}b^ir^{2i}\left\{\frac{1}{\ln b}\left[2\ln (r+1)-\frac 1i -\Psi(i+2)-\gamma\right]+1-\frac{1}{\ln b}\right\},
\end{align*}
$$
G_1(r)=2(-T_1(r)+S_0(r)-2bT_2+bS_1(r)), \quad E_1(r)=b\tilde \Rc_1(r)+\frac{2e_1}{|\ln b|^2}(bS_0(r)+b^2\tilde S_1(r)).
$$
One has from \fref{est:T0iatinf}, as $\hat d_1=-1/2$, $d_1=1/4$ and $\hat d_2=1/16$, $e_1=\ln 2-\gamma-1$ and $R_0=\zeta_0/\sqrt b$:
\begin{align*}
F_1(R_0)&=\frac{b}{\zeta_0^2}+2b\left(-\frac{\ln \zeta_0-\frac{\ln b}{2}}{2}+\frac 14\right)+b\left(-\frac{2}{\ln b}\left(-\frac{\ln \zeta_0-\frac{\ln b}{2}}{2}\right)+\frac 12\right)\\
\nonumber &+\frac{2e_1}{|\ln b|^2}\left( -\frac{b\ln b}{4}+\frac{b\zeta_0^2 \ln b}{16}+ \frac{b\ln b}{4} \sum_{i=2}^{\infty}\frac{(1)_{i-1}}{(2)_ii!2^i}\zeta_0^{2i}\right)\\
\nonumber &-\frac{b}{2}\sum_{i=1}^{\infty} \frac{(1)_{i-1}}{(2)_ii!2^i}\zeta_0^{2i}\left\{\frac{1}{\ln b}\left[2\ln \zeta_0-\ln b-\frac 1i -\Psi(i+2)-\gamma\right]+1-\frac{1}{\ln b}\right\}+\Oc\left(\frac{b}{|\ln b|^2}\right) 
\end{align*}
\begin{align}
\nonumber F_1(R_0) &= b\Bigr[ \frac{\ln b}{2}+\frac{1}{\zeta_0^2}-\ln \zeta_0+\frac 12+\frac{\ln \zeta_0}{\ln b}+\frac{e_1}{2\ln b }\left(-1+ \sum_{i=1}^{\infty}\frac{(1)_{i-1}}{(2)_ii!2^i}\zeta_0^{2i}\right)\\
\nonumber &-\frac{1}{2\ln b}\sum_{i=1}^{\infty} \frac{(1)_{i-1}}{(2)_ii!2^i}\zeta_0^{2i} \left[2\ln \zeta_0-\frac 1i -\Psi(i+2)-\gamma -1\right]\Bigr]\\
\nonumber &= b\left\{ \frac{\ln b}{2}+\frac{1}{\zeta_0^2}-\ln \zeta_0+\frac 12+\frac{\ln \zeta_0}{\ln b}-\frac{e_1}{2\ln b } -\frac{1}{2\ln b}\sum_{i=1}^{\infty} \frac{(1)_{i-1}}{(2)_ii!2^i}\zeta_0^{2i} \left[2\ln \zeta_0-\ln 2-\frac 1i -\Psi(i+2)\right]\right\}\\
&+\Oc\left(\frac{b}{|\ln b|^2}\right),
\end{align}
and similarly, we have
$$
 (r\partial_r F_1)(R_0)= \frac{-2b}{\zeta_0^2}-b +\frac{b}{\ln b} - \frac{b}{2\ln b}\sum_{i=1}^{\infty} \frac{(1)_{i-1}}{(2)_i i!2^i}\zeta_0^{2i}2i\left[2\ln \zeta_0-\Psi(i+2)-\ln 2  \right]+\Oc\left(\frac{b}{|\ln b|^2}\right).
$$
From \fref{id:refinementRn=1} and \fref{id:refinementSjn=1}, we obtain
$$
\sum_{0\leq k \leq 2} ((r \partial_r)^k E_1) (R_0)| \leq C(\zeta_0) \frac{b}{|\ln b|^2},
$$
Hence, as $G_1(R_0)=\Oc(|\ln b|)$ and $r \partial_r G_1(R_0)=\Oc(|\ln b|)$, we obtain from the above identities
\begin{align*}
&\phi_1^\inn(R_0)= b\Bigr[ -\frac{\ln b}{2}H_1(\zeta_0)+K_1(\zeta_0)+\frac{1}{2\ln b}J_1(\zeta_0) +\hat \alpha b G_1(R_0)+\Oc(\frac{1}{|\ln b|^2})\Bigr],\\
 &r\partial_r \phi_1^\inn(R_0) = b\left[\zeta\partial_\zeta K_1+\frac{1}{2\ln b}\zeta \partial_\zeta J_1  + \hat \alpha r\partial_r G_1(R_0)+\Oc\left(\frac{1}{|\ln b|^2}\right)\right].
\end{align*}
where we used \fref{def:HnGn}, so that $H_1(\zeta)=1$ and $K_1(\zeta)=\zeta^{-2}-\ln \zeta+1/2$ and 
\be \label{def:J1matching}
J_1(\zeta_0)= 2 \ln \zeta_0-e_1-\sum_{i=1}^{\infty} \frac{(1)_{i-1}}{(2)_ii!2^i}\zeta_0^{2i} \left[2\ln \zeta_0-\ln 2-\frac 1i -\Psi(i+2)\right].
\ee
We finally obtain
\begin{align}
\nonumber \frac{r\partial_r\phi^{\inn}_1(R_0)}{\phi^{\inn}_1(R_0)} & =  \frac{\zeta\partial_\zeta K_1+\frac{1}{2\ln b}\zeta \partial_\zeta J_1  + \hat \alpha r\partial_r G_1(R_0)+\Oc\left(\frac{1}{|\ln b|^2}\right)}{-\frac{\ln b}{2}H_1(\zeta_0)+K_1(\zeta_0)+\frac{1}{2\ln b}J_1(\zeta_0) +\hat \alpha b G_1(R_0)+\Oc(\frac{1}{|\ln b|^2})}\\
\nonumber& = -\frac{2}{\ln b} \frac{\zeta\partial_\zeta K_1+\frac{1}{2\ln b}\zeta \partial_\zeta J_1  + \hat \alpha r\partial_r G_1(R_0)+O\left(\frac{1}{|\ln b|^2}\right)}{H_1(\zeta_0)-\frac{2}{\ln b}K_1(\zeta_0)-\frac{1}{|\ln b|^2}J_1(\zeta_0) -\frac{2}{\ln b}\hat \alpha b G_1(R_0)+\Oc(\frac{1}{|\ln b|^3})}\\
\nonumber &=-\frac{2}{\ln b}\left\{ \frac{\zeta\partial_\zeta K_1}{H_1}+\frac{1}{\ln b}\frac{\zeta \partial_\zeta J_1H_1+2K_1\zeta \partial_\zeta K_1}{H_1^2} \right.\\
\nonumber & \qquad \quad \left.+\tilde \alpha \frac{r\partial_r G_1 H_1+\frac{2}{\ln b}G_1\zeta \partial_\zeta K_1}{H_1^2}+\Oc(|\ln b|^{-2})\right\}\\
\label{id:refinedinteriorn=1} &=-\frac{2}{\ln b} \frac{\zeta\partial_\zeta K_1}{H_1}-\frac{2}{\ln b^2}\frac{\zeta \partial_\zeta J_1H_1+2K_1\zeta \partial_\zeta K_1}{H_1^2}+\tilde \alpha \frac{\Oc(1)}{H_1^2}+\Oc(|\ln b|^{-3}) 
\end{align}
where the constant in the $\Oc(1)$ is independent of $\zeta_0$.\\

\noindent \underline{The case $n=0$:} We first use the refined asymptotics \fref{id:refinementRn=0} and \fref{id:refinementSjn=0} to obtain:
$$
\phi^\inn_0(r)=F_0(r)+\bar \alpha b G_0(r)+E_0(r),
$$
where:
\begin{align}
\nonumber F_0(r) &=T_0(r)+b\left(-\frac{2}{\ln b}T_1(r)+\As_0^{-1}\Theta_0\right)+\frac{b}{2}\sum_{i=1}^{\infty}\frac{1}{(2)_i2^i}b^{i} r^{2i}\left\{ \frac{1}{\ln b}\left[2\ln (r+1)-\Psi (i+2)-\gamma\right]+1\right\}, 
\end{align}
$$
G_0(r)=2\left(-T_1(r)+\frac 12 \sum_{i=1}^\infty \frac{1}{(2)_i2^i}b^ir^{2i}\ln (r+1)\right), \quad E_0(r)=b\tilde \Rc_0(r)+2\bar \alpha b \tilde S_0.
$$
One has from \fref{est:T0iatinf}, as $\hat d_1=-1/2$, $d_1=1/4$:
\begin{align*}
\nonumber F_0(R_0)&=\frac{b}{\zeta_0^2}+b\left(-\frac{2}{\ln b}\left(-\frac{\ln \zeta_0-\frac{\ln b}{2}}{2}+\frac 14\right)+\frac 12\right) \\
& \quad +\frac{b}{2}\sum_{i=1}^{\infty}\frac{1}{(2)_i2^i} \zeta_0^{2i}\left\{ \frac{1}{\ln b}\left[2\ln \zeta_0-\Psi (i+2)-\gamma\right]\right\}+\Oc(b^{\frac 32}) \\
\nonumber &=\frac{b}{\zeta_0^2}+\frac{b \ln \zeta_0}{\ln b}-\frac{b}{2\ln b}+\frac{b}{2\ln b}\sum_{i=1}^{\infty}\frac{1}{(2)_i2^i} \zeta_0^{2i}\left\{ 2\ln \zeta_0-\Psi (i+2)-\gamma \right\} +\Oc(b^{\frac 32}), 
\end{align*}
and similarly, we have
$$
 (r\partial_r F_0)(R_0)= \frac{-2b}{\zeta_0^2} +\frac{b}{\ln b} + \frac{b}{2\ln b}\sum_{i=1}^{\infty} \frac{1}{(2)_i 2^i}\zeta_0^{2i}2i\left[2\ln \zeta_0+\frac 1i-\Psi(i+2) -\gamma \right]+\Oc\left( b^{\frac 32}\right)
$$
$$
\pa_b (b^{-1}F_0(R_0))=O\left(\frac{1}{b|\ln b|^2}\right), \ \ \pa_b (b^{-1}r\partial_r F_0(R_0))=\Oc\left(\frac{1}{b|\ln b|^2}\right).
$$
From \fref{id:refinementRn=0}, we obtain
\begin{align*}
&\sum_{0\leq k \leq 2, \ 0\leq \ell+\ell'\leq 1} ((b\pa_b)^\ell \pa_{\bar \alpha}^{\ell'}(r \partial_r)^k E_0) (R_0)| \leq C(\zeta_0) \frac{b}{|\ln b|^2}.
\end{align*}
One also has
$$
G_0(R_0)=2\left(-\frac{1}{4}\ln b-\frac 14 \ln b \sum_{i=1}^\infty \frac{1}{(2)_i2^i}\zeta_0^{2i}\right)+\Oc( 1)=-\frac{\ln b}{2} \tilde G_0(\zeta_0)+\Oc(1), \quad \pa_bG_0(R_0)=\Oc\left(\frac 1b \right),
$$
where
\be \label{def:tildeG0matching}
\tilde G_0(\zeta_0)=\sum_{i=0}^\infty \frac{1}{(2)_i2^i}\zeta_0^{2i},
\ee
so that 
$$
r\partial_r G_0(R_0)=-\frac{\ln b}{2} \zeta \partial_\zeta \tilde G_0(\zeta_0)+\Oc(1), \quad \pa_b r\partial_r \tilde G_0(R_0)=\Oc\left(\frac 1b \right).
$$
We obtain from the above identities
\begin{align}
\nonumber \phi_0^\inn(R_0)&= b\Bigr[ \frac{1}{\zeta_0^2}+\frac{1}{2 \ln b}J_0(\zeta_0) -\frac{\ln b}{2}\bar \alpha \tilde{G}_0(\zeta_0)+\Oc(|\ln b|^{-2})  \Bigr],
\end{align}
\begin{align}
\nonumber r\partial_r \phi_0^\inn(R_0)&= b\Bigr[ \frac{-2}{\zeta_0^2}+\frac{1}{\ln b}\zeta \pa_\zeta J_0(\zeta_0) -\frac{\ln b}{2}\bar \alpha r\partial_r\tilde{G}_0(\zeta_0)+\Oc( |\ln b|^{-2})  \Bigr],
\end{align}
where
\be \label{def:J0matching}
J_0(\zeta_0)= 2 \ln \zeta_0-1+\sum_{i=1}^{\infty} \frac{1}{(2)_i2^i}\zeta_0^{2i} \left[2\ln \zeta_0-\Psi(i+2)-\gamma \right],
\ee
and for $\bar \alpha=\Oc( |\ln b|^{-2})$,
$$
\pa_b \left( b^{-1} \phi_0^\inn(R_0) \right)=O\left(\frac{1}{b |\ln b|^2}\right), \quad \pa_b \left(b^{-1} r\partial_r \phi_0^\inn(R_0)\right)=\Oc\left(\frac{1}{b |\ln b|^2}\right),
$$
$$
\pa_{\bar \alpha} \left( \phi_0^\inn(R_0) \right)=-\frac{b\ln b}{2}\tilde G_0(\zeta_0)+\Oc\left( b\right), \quad \pa_{\bar \alpha} \left( r\partial_r \phi_0^\inn(R_0)\right)=-\frac{b\ln b}{2}r\pa_r\tilde G_0(\zeta_0)+\Oc\left(b\right)
$$
We finally obtain
\begin{align}
\nonumber &\frac{r\partial_r\phi^{\inn}_0(R_0)}{\phi^{\inn}_0(R_0)}  =  \frac{\frac{-2}{\zeta_0^2}+\frac{1}{2\ln b}\zeta \pa_\zeta J_0(\zeta_0) +\bar \alpha r\partial_rG_0(R_0)+\Oc( |\ln b|^{-2}) }{ \frac{1}{\zeta_0^2}+\frac{1}{2 \ln b}J_0(\zeta_0) +\bar \alpha G_0(R_0)+\Oc( |\ln b|^{-2})}\\
\nonumber & \quad  =  \frac{-2+\frac{\zeta_0^2}{\ln b}\zeta \pa_\zeta J_0(\zeta_0) +\zeta_0^2\bar \alpha r\partial_rG_0(R_0)+\Oc( |\ln b|^{-2}) }{1+\frac{\zeta_0^2}{2 \ln b}J_0(\zeta_0) +\bar \alpha \zeta_0^2G_0(R_0)+\Oc( |\ln b|^{-2})}\\
\label{id:refinedinteriorn=0} & \quad =-2+\frac{1}{\ln b}\zeta_0^2(\frac 12 \zeta \pa_\zeta J_0+J_0)-\frac{\ln b}{2} \bar \alpha \zeta_0^2(\zeta \pa_\zeta \tilde G_1(\zeta_0)+2\tilde G_1(\zeta_0)+\Oc( |\ln b|^{-1}))+\Oc( |\ln b|^{-2}),
\end{align}
and 
\begin{align}
\nonumber &\pa_b \left( \frac{r\partial_r\phi^{\inn}_0(R_0)}{\phi^{\inn}_0(R_0)}\right)=\pa_b \left( \frac{b^{-1}r\partial_r\phi^{\inn}_0(R_0)}{b^{-1}\phi^{\inn}_0(R_0)}\right)\\
& \quad = \frac{ \pa_b(b^{-1}r\partial_r\phi^{\inn}_0(R_0))b^{-1}\phi^{\inn}_0(R_0)-\pa_b(b^{-1}\phi^{\inn}_0(R_0))b^{-1}r\partial_r\phi^{\inn}_0(R_0)}{b^{-1}\phi^{\inn}_0(R_0)}=\Oc( \frac{1}{b|\ln b|^2}), \label{id:refinedinteriorn=0pab} 
\end{align}
and 
\begin{align}
\label{id:refinedinteriorn=0paalpha}  \pa_{\bar \alpha} \left( \frac{r\partial_r\phi^{\inn}_0(R_0)}{\phi^{\inn}_0(R_0)}\right)&=-\frac{\ln b}{2} \zeta_0^2(\zeta \pa_\zeta \tilde G_1(\zeta_0)+2\tilde G_1(\zeta_0)+\Oc( |\ln b|^{-1})) ,
\end{align}
where the constant in the $\Oc( |\ln b|^{-2})$ is independent of $\bar \alpha$.\\

\noindent \textbf{Step 2:} \emph{The exterior term}. Recall the decomposition $q[b,\bar \alpha](z)=\Gamma (\theta) h_\theta(z)+\mathcal G[b,\bar \alpha](z)$ from \fref{eq:qn}. From the estimates \fref{est:Gc} the second term is of lower order and satisfies:
\be \label{bd:mathgexterieur}
\sum_{0\leq k+\ell\leq 1} |(b\pa_b)^k \pa_{\bar \alpha}^\ell (\mathcal G(z_0))|+|(b\pa_b)^k \pa_{\bar \alpha}^\ell (z\pa_z \mathcal G(z_0))| \lesssim b^{\frac 12}.
\ee
We now investigate the formula giving $h_{\theta}$. From the recurrence relation of the Gamma function and the identity $\partial_\theta (\theta)_i=(\theta)_i(\Psi(\theta+i)-\Psi (\theta))$:
\bea
\label{id:hthetamatching} \Gamma (\theta) h_\theta(z) &=& \frac{1}{z} + (\theta-1)\sum_{i = 0}^\infty\frac{(\theta)_i}{(2)_i i!} z^i \Big[\ln z + \Psi(\theta + i) - \Psi(1 + i) - \Psi(2+i) \Big],\\
\nonumber z\pa_z \Gamma (\theta)h_\theta(z)&=& - \frac{1}{z} +(\theta-1)\sum_{i = 0}^\infty\frac{(\theta)_i}{(2)_i i!} z^i \Big[ i\left(\ln z + \Psi(\theta + i) - \Psi(1 + i) - \Psi(2+i)\right) +1\Big],\\
\eea
\bee
\partial_{\theta}\Gamma (\theta)h_\theta(z) &=& \sum_{i = 0}^\infty\frac{(\theta)_i}{(2)_i i!} z^i \Big[ \big(\ln z + \Psi(\theta + i) - \Psi(1 + i) - \Psi(2+i)\big)\big(1+(\theta-1)(\Psi (\theta+i)-\Psi (\theta))\big) \\
&& \hspace*{10cm} +(\theta-1)\partial_\theta \Psi (\theta+i) \Big],
\eee
\bee
&&\partial_{\theta} z\pa_z \Gamma (\theta)h_\theta(z)\\
&& \quad =\sum_{i = 0}^\infty\frac{(\theta)_i}{(2)_i i!} z^i \Big[i \big\{\left(\ln z + \Psi(\theta + i) - \Psi(1 + i) - \Psi(2+i)\big)\big(1+(\theta-1)(\Psi (\theta+i)-\Psi (\theta))\big\}\right) \\
 && \hspace*{10cm}+i\left((\theta-1)\partial_\theta \Psi (\theta+i)\right) +1\Big].
\eee
We now decompose all above expressions into leading order and lower terms. We first collect some estimates on the coefficients. Note that for $i\geq n$ one has from the recurrence relation of the Gamma function:
\begin{align} \label{id:formulathetaexpansion}
(\theta)_i&=\frac{\Gamma (\theta +i)}{\Gamma (\theta)}=(\theta)(\theta+1)...(\theta+i-1)=(1-n+\tilde \alpha)(2-n+\tilde \alpha)...(i-n+\tilde \alpha)=\Oc( |\tilde \alpha|)
\end{align}
because there is some $0\leq j \leq i-1$ such that $ 1-n+j=0$. Moreover, for a large enough argument the digamma function
\be \label{id:nonsingularPsi}
\Psi (\theta+i)=\Psi (1-n+i+\tilde \alpha)=\Psi (1-n+i)+\Oc( \tilde \alpha)=\Oc( 1) \quad \mbox{ for } i\geq n
\ee
is non-singular since $1-n+i>1$. We recall the recurrence relation for the digamma function $\Psi (z+1)=\Psi(z)+1/z$, with $\Psi (1)=-\gamma$ the Euler constant. Then, if $k$ is an integer:
$$
\Psi (k+1)=\frac{1}{k}+\Psi (k)=\frac{1}{k}+\frac{1}{k-1}+...+\frac 12 +1-\gamma.
$$
Hence, refining \fref{id:formulathetaexpansion} for $i<n$, we obtain
\begin{align} 
\nonumber (\theta)_i &= (1-n)_i\left(1+\tilde \alpha (\Psi (n-i)-\Psi(n))\right)+\Oc( \tilde \alpha^2)
\end{align}
and
\begin{align}
\nonumber \Psi (\theta+i) & = -\frac{1}{\theta+i}+\Psi(\theta+i+1)=-\frac{1}{1-n+i+\tilde \alpha}-\frac{1}{2-n+i+\tilde \alpha}-...-\frac{1}{-1+\tilde \alpha}
-\frac{1}{\tilde \alpha}+\Psi (1+\tilde \alpha)\\
\label{id:psithetaidvpt} &= -\frac{1}{\tilde \alpha}+\Psi(n-i)+\Oc( \tilde \alpha),
\end{align}
\be \label{id:estimationpathetapsi}
\partial_\theta \Psi(\theta+i) =\partial_{\tilde \alpha}\Psi(\theta+i) =\frac{1}{\tilde \alpha^2}+\Oc( 1) \quad \mbox{ for } i<n.
\ee
The coefficients that will appearing in the expansion are related to the inner expansion the following way. Using the recurrence relations \fref{id:recurrencedi}-\fref{id:recurrencecnj} and the initial values for $c_{n,1}$ and $\hat d_1$, there holds
\be \label{id:linkcoefficients1}
-c_{n,i+1}\hat d_{i+1}=n\frac{(1-n)_i}{(2)_ii!2^i},
\ee
and similarly using the recurrence relations \fref{id:recurrencedi}, there holds
\be \label{id:linkcoefficients2}
-\frac{2d_{i+1}}{\hat d_{i+1}}=2+\frac{2}{2}+\frac{2}{3}+...+\frac{2}{i}+\frac{1}{i+1}=\Psi(i+2)+\Psi(i+1)+2\gamma.
\ee
Hence, the strategy is the following. We first truncate the series \fref{exp:varphi1at0} expressing $h_{\theta}$ for $0<z\lesssim 1$ using \fref{id:nonsingularPsi} and \fref{id:formulathetaexpansion}. Then, we expand it  with respect to $\tilde \alpha$. Finally, we express the coefficients in function of those of the inner expansion via \fref{id:linkcoefficients1}-\fref{id:linkcoefficients2}. The result of this strategy is given by 
\begin{align*}
 \Gamma (\Theta) h_\theta(z) &=  \frac{1}{z} +(\theta-1)\sum_{i = 0}^\infty\frac{(\theta)_i}{(2)_i i!} z^i \Big[\ln z + \Psi(\theta + i) - \Psi(1 + i) - \Psi(2+i) \Big]\\
\nonumber &= \frac{1}{z} +(\theta-1)\sum_{i = 0}^{n-1} [...]+(\theta-1)\sum_{i = n}^{\infty} [...] \\
\nonumber&= \frac{1}{z} +(\theta-1)\sum_{i = 0}^{n-1}\frac{(\theta)_i}{(2)_i i!} z^i \Big[\ln z + \Psi(\theta + i) - \Psi(1 + i) - \Psi(2+i) \Big]+\Oc( |\tilde \alpha |).
\end{align*}
\begin{align}
\nonumber  \Gamma (\Theta) h_\theta(z) &= \frac{1}{z} +(\tilde \alpha-n)\sum_{i = 0}^{n-1}\frac{ (1-n)_i}{(2)_i i!} \left(1+\tilde \alpha \left(\Psi(n-i)-\Psi(n)\right)+\Oc( \tilde \alpha^2)\right)z^i \\
\nonumber  &\times \Big[\ln z -\frac{1}{\tilde \alpha}+\Psi(n-i)+\Oc( |\tilde \alpha|) - \Psi(1 + i) - \Psi(2+i) \Big]+\Oc( |\tilde \alpha|)\\
\nonumber &=  \frac 1z+\sum_{i=0}^{n-1}n \frac{(1-n)_i}{(2)_ii!}z^i\left(-\ln z -\Psi(n+1)+\Psi(i+1)+\Psi(i+2)+\frac{1}{\tilde \alpha} \right)+\Oc( |\tilde \alpha|)\\
\nonumber &= \frac 1z+\sum_{i=0}^{n-1}n \frac{(1-n)_i}{(2)_ii!}z^i\left( -\ln z-\ln 2 +\Psi(i+1)+\Psi(i+2)+2\gamma+e_n+\frac{1}{\tilde \alpha} \right)+\Oc( |\tilde \alpha|),\\
\label{idhthetaext1}&=\frac 1z+\sum_{i=1}^{n} 2^{i-1} c_{n,i} z^{i-1}\left( \hat d_{i} \left( \ln z+\ln 2 -e_n-\frac{1}{\tilde \alpha}\right) +2 d_{i} \right)+\Oc( |\tilde \alpha|).
\end{align}
Similarly, skipping the computations which are verbatim the same as the one above yields
\begin{align}
\nonumber z\pa_z \Gamma (\theta)h_\theta(z)&= - \frac{1}{z} +(\theta-1)\sum_{i = 0}^\infty\frac{(\theta)_i}{(2)_i i!} z^i \Big[ i\left(\ln z + \Psi(\theta + i) - \Psi(1 + i) - \Psi(2+i)\right) +1\Big] \\
\label{idhthetaext2} &  = - \frac{1}{z} + \sum_{i = 1}^{n} 2^{i-1} c_{n,i} z^{i-1} \Big[ (i-1)\left(\hat d_i\left(\ln z +\ln 2 -e_n-\frac{1}{\tilde \alpha}\right)+2d_i\right) +\hat d_i\Big] +\Oc( |\tilde \alpha|)
\end{align}
Then, using \fref{id:formulathetaexpansion}, \fref{id:nonsingularPsi}, \fref{id:psithetaidvpt}, \fref{id:estimationpathetapsi} and $\partial_\theta \tilde \alpha=1$, we compute
\begin{align}
\nonumber &\partial_\theta ( \Gamma (\theta)h_\theta(z))\\
\nonumber &= \sum_{i = 0}^\infty\frac{(\theta)_i}{(2)_i i!} z^i \Big[ \left(\ln z + \Psi(\theta + i) - \Psi(1 + i) - \Psi(2+i)\right)\left(1+(\theta-1)(\Psi (\theta+i)-\Psi (\theta))\right)+(\theta-1)\partial_\theta \Psi (\theta+i) \Big]\\
\nonumber &= -\frac{n}{\tilde \alpha^2}\sum_{i = 0}^{n-1}\frac{ (1-n)_i}{(2)_i i!} z^i+\Oc( 1)= \frac{1}{\tilde \alpha^2}\sum_{i = 1}^{n}2^{i-1}c_{n,i}\hat d_i z^{i-1}+\Oc( 1)
\end{align}
so that from \fref{idhthetaext1}:
\begin{align}
\nonumber &\partial_{\theta}\left(\tilde \alpha \Gamma (\theta)h_\theta(z)\right)= \Gamma (\theta)h_\theta(z)+\tilde \alpha \partial_\theta ( \Gamma (\theta)h_\theta(z))\\
\nonumber =&\frac 1z+\sum_{i=1}^{n} 2^{i-1} c_{n,i} z^{i-1}\left( \hat d_{i} \left( \ln z+\ln 2 -e_n-\frac{1}{\tilde \alpha}\right) +2 d_{i} \right)+\Oc( |\tilde \alpha|) +\frac{1}{\tilde \alpha}\sum_{i = 1}^{n}2^{i-1}c_{n,i}\hat d_i z^{i-1}+\Oc( |\tilde \alpha|)\\
\label{idhthetaext3} =&\frac 1z+\sum_{i=1}^{n} 2^{i-1} c_{n,i} z^{i-1}\left( \hat d_{i} \left( \ln z+\ln 2 -e_n\right) +2 d_{i} \right)+\Oc( |\tilde \alpha|),
\end{align}
and similarly
\begin{align}
\nonumber &\partial_{\theta} \left(z\pa_z \Gamma (\theta)h_\theta(z)\right)\\
\nonumber =&\sum_{i = 0}^\infty\frac{(\theta)_i}{(2)_i i!} z^i \Big[i \left(\left(\ln z + \Psi(\theta + i) - \Psi(1 + i) - \Psi(2+i)\right)\left(1+(\theta-1)(\Psi (\theta+i)-\Psi (\theta))\right)+(\theta-1)\partial_\theta \Psi (\theta+i)\right) +1\Big]\\
\nonumber=& -\frac{n}{\tilde \alpha^2}\sum_{i=0}^{n-1} \frac{(1-n)_iiz^i}{(2)_ii!}+\Oc( 1)= \frac{1}{\tilde \alpha^2}\sum_{i=1}^{n} 2^{i-1}(i-1)z^{i-1}c_{n,i}\hat d_i+\Oc( 1)
\end{align}
so that from \fref{idhthetaext2}, we get
\begin{align}
\nonumber &\partial_{\theta}\left(\tilde \alpha z\partial_z\Gamma (\theta)h_\theta(z)\right)= z\partial_z\Gamma (\theta)h_\theta(z)+\tilde \alpha \partial_\theta ( z\partial_z \Gamma (\theta)h_\theta(z))\\
\label{idhthetaext4} & \qquad = - \frac{1}{z} + \sum_{i = 1}^{n} 2^{i-1} c_{n,i} z^{i-1} \Big[ (i-1)\left(\hat d_i\left(\ln z +\ln 2 -e_n\right)+2d_i\right) +\hat d_i\Big] +\Oc( |\tilde \alpha|).
\end{align}

Therefore we obtain from \fref{idhthetaext1}, \fref{bd:mathgexterieur}, as $z=\zeta^2/2$ and $\tilde \alpha=1/\ln b+\Oc( |\ln b|^{-2})$:
\begin{align}
\nonumber q (z_0)&=\frac{2}{\zeta_0^2}+\sum_{i=1}^{n} 2^{i-1} c_{n,i} \frac{\zeta_0^{2(i-1)}}{2^{i-1}}\left( \hat d_{i} \left( \ln \left(\frac{\zeta_0^2}{2}\right)+\ln 2 -e_n-\frac{1}{\tilde \alpha}\right) +2 d_{i} \right)+\Oc( |\tilde \alpha|)+\Oc( b^{\frac 12})\\
\label{id:exprqz0}&=-\frac{1}{\tilde \alpha}H_n(\zeta_0)+2K_n(\zeta_0)-e_nH_n(\zeta_0)+\Oc( |\tilde \alpha|),
\end{align}
where $H_n$ and $G_n$ are given by \fref{def:HnGn}. Similarly, we compute from \fref{idhthetaext2} and \fref{bd:mathgexterieur},
\begin{align*} 
 (z\pa_z)q (z_0) & =- \frac{2}{\zeta_0^2} + \sum_{i = 1}^{n} 2^{i-1} c_{n,i} \left(\frac{\zeta_0^2}{2}\right)^{i-1} \Big[ (i-1)\left(\hat d_i\left(\ln \left(\frac{\zeta_0^2}{2}\right) +\ln 2 -e_n-\frac{1}{\tilde \alpha}\right)+2d_i\right) +\hat d_i\Big]\\
 & \qquad \qquad  +\Oc( |\tilde \alpha|)+\Oc( b^{\frac 12})\\
&=-\frac{1}{2\tilde \alpha} \zeta \partial_{\zeta}H_n(\zeta_0)+ \zeta \partial_\zeta K_n(\zeta_0)-\frac{e_n}{2} \zeta \partial_{\zeta}H_n(\zeta_0)+\Oc( |\tilde \alpha|).
\end{align*}
From \fref{bd:mathgexterieur}, \fref{idhthetaext3}, \fref{idhthetaext4}, recalling that $b$ and $\bar \alpha$ are two independent parameters for the moment, using the relations $b\partial_b \theta= -1/|\ln b|^2=\Oc( 1/|\ln b|^2)$ and $\pa_{\bar \alpha}=\partial_{\theta}$:
\be \label{id:bpabqz0}
b\partial_b \left(\tilde \alpha q (z_0)\right)= \Oc( \frac{1}{|\ln b|^2}) \partial_\theta (\tilde \alpha \Gamma (\theta)h(\theta)(z_0))+\Oc( b^{\frac 32})=\Oc( \tilde \alpha^2),
\ee
\be \label{id:bpabzpazqz0}
b\partial_b \left(\tilde \alpha z\partial_z q(z_0)\right)= \Oc( \frac{1}{|\ln b|^2})\partial_\theta (\tilde \alpha z\partial_z\Gamma (\theta)h(\theta)(z_0))+\Oc( b^{\frac 32})=\Oc( \tilde \alpha^2),
\ee
\begin{align} 
\nonumber \partial_{\bar \alpha}\left(\tilde \alpha q(z_0)\right) &=\partial_\theta \left( \tilde \alpha q(z_0)\right)\\
\nonumber &  =\frac{2}{\zeta_0^2}+\sum_{i=1}^{n} 2^{i-1} c_{n,i} (\frac{\zeta_0^2}{2})^{i-1}\left( \hat d_{i} \left( \ln (\frac{\zeta_0^2}{2})+\ln 2 -e_n\right) +2 d_{i} \right)+\Oc( |\tilde \alpha|)+\Oc( b^{\frac 12})\\
\label{id:paalphaqz0} &=2K_n(\zeta_0)-e_nH_n(\zeta_0)+\Oc( |\tilde \alpha|),
\end{align}
and 
\begin{align} 
\nonumber \partial_{\bar \alpha}\left(\tilde \alpha z\partial_z q(z_0)\right) &=\partial_\theta \left( \tilde \alpha z\partial_z q(z_0)\right) \\
\nonumber & = - \frac{2}{\zeta_0^2} + \sum_{i = 1}^{n} 2^{i-1} c_{n,i} (\frac{\zeta_0^2}{2})^{i-1}\Big[ (i-1)\left(\hat d_i\left(\ln (\frac{\zeta_0^2}{2}) +\ln 2 -e_n\right)+2d_i\right) +\hat d_i\Big] +\Oc( |\tilde \alpha|)\\
\label{id:paalphazpazqz0} & = \zeta \partial_\zeta K_n(\zeta_0)-\frac{e_n}{2}\zeta \partial_{\zeta}H_n(\zeta_0)+\Oc( |\tilde \alpha|).
\end{align}
We deduce that for $n\geq 2$, 
\begin{align}
\nonumber  \frac{z\partial_z q(z_0)}{q(z_0)} &=\frac{-\frac{1}{2\tilde \alpha} \zeta \partial_{\zeta}H_n(\zeta_0)+ \zeta \partial_\zeta K_n(\zeta_0)-\frac{e_n}{2}\zeta \partial_{\zeta}H_n(\zeta_0)+\Oc( |\tilde \alpha|)}{-\frac{1}{\tilde \alpha}H_n(\zeta_0)+2K_n(\zeta_0)-e_nH_n(\zeta_0)+\Oc( |\tilde \alpha|)}\\
\nonumber &= \frac 12 \frac{ \zeta \partial_{\zeta}H_n(\zeta_0)-2\tilde \alpha \zeta \partial_\zeta K_n(\zeta_0)+\tilde \alpha e_n \zeta \partial_{\zeta}H_n(\zeta_0)+\Oc( |\tilde \alpha|^2)}{H_n(\zeta_0)-2\tilde \alpha K_n(\zeta_0)+\tilde \alpha e_nH_n(\zeta_0)+\Oc( |\tilde \alpha|^2)}\\
\nonumber &= \frac 12 \left( \frac{\zeta \partial_\zeta H_n(\zeta_0)}{H_n(\zeta_0)}+\tilde \alpha \frac{(e_n\zeta\partial_\zeta H_n-2\zeta \partial_\zeta K_n)H_n-(e_nH_n-2K_n)\zeta \partial_\zeta H_n}{H_n^2(\zeta_0)}\right)+\Oc( \tilde \alpha^2) \\
\label{id:extmatching1} &= \frac 12 \frac{\zeta \partial_\zeta H_n(\zeta_0)}{H_n(\zeta_0)}+\tilde \alpha \frac{K_n(\zeta_0)\zeta \partial_\zeta H_n(\zeta_0)-\zeta \partial_\zeta K_n(\zeta_0)H_n(\zeta_0)}{H_n^2(\zeta_0)}+\Oc( \tilde \alpha^2) 
\end{align}
and similarly from \fref{id:bpabqz0}, \fref{id:paalphaqz0}, \fref{id:paalphaqz0} and \fref{id:paalphazpazqz0},
\begin{align}
\nonumber b\partial_b \left(\frac{z\partial_z q(z_0)}{q(z_0)} \right)&=b\partial_b \left( \frac{\tilde \alpha z\partial_z q(z_0)}{\tilde \alpha q(z_0)} \right)\\
\label{id:extmatching2} &= \frac{b\partial_b(\tilde \alpha z\partial_z q(z_0))\tilde \alpha q(z_0)-\tilde \alpha z\partial_z q (z_0)b\partial_b (\tilde \alpha q(z_0))}{\tilde \alpha^2q(z_0)^2}=\frac{\Oc( \tilde \alpha^2)}{\tilde \alpha^2 q(z_0)^2}=\Oc( \tilde \alpha^2),
\end{align}
\begin{align}
\nonumber &\partial_{\tilde \alpha} \left(\frac{z\partial_z q(z_0)}{q(z_0)} \right)= \partial_{\tilde \alpha} \left(\frac{\tilde \alpha  z\partial_z q(z_0)}{\tilde \alpha q(z_0)} \right) =\frac{\partial_{\tilde \alpha}(\tilde \alpha z\partial_z q(z_0)) \tilde \alpha q(z_0)-\partial_{\tilde \alpha}(\tilde q(z_0))\tilde \alpha z\partial_zq(z_0)}{\tilde \alpha^2 q^2(z_0)}\\
\label{id:extmatching3} &=\frac{K_n(\zeta_0)\zeta\partial_\zeta H_n(\zeta_0) -\zeta \partial_\zeta K(\zeta_0)H(\zeta_0)}{ H_n^2(\zeta_0)}+\Oc( |\tilde \alpha|).
\end{align}
\underline{The case $n=1$:} For $n=1$, $\theta=\tilde \alpha$, so we refine further $\tilde \alpha$ and take
$$
\tilde \alpha=\frac{1}{\ln b}+\frac{e_1}{|\ln b|^2}+\hat \alpha, \quad e_1=\ln 2 -\gamma-1=\ln 2 -\Psi (2)-2\gamma, \quad \hat \alpha=\Oc( |\ln b|^{-3}).
$$
We then refine further $\Gamma(\theta)h_{\theta}$ by noticing that for $i\geq 1$, $(\tilde \alpha)_i=\tilde \alpha \Gamma (i)+\Oc( \tilde \alpha^2)$ and $\Psi (\tilde \alpha)=-\tilde \alpha^{-1}-\gamma+\pi^2\tilde \alpha /6+\Oc( \tilde \alpha^2)$,
\begin{align*}
 \Gamma (\Theta) h_\theta(z) &=  \frac{1}{z} +(\tilde \alpha-1)\sum_{i = 0}^\infty\frac{(\tilde \alpha)_i}{(2)_i i!} z^i \Big[\ln z + \Psi(\tilde \alpha + i) - \Psi(1 + i) - \Psi(2+i) \Big]\\
&= \frac{1}{z}+(\tilde \alpha-1) \Big[\ln z + \Psi(\tilde \alpha ) - \Psi(1 ) - \Psi(2) \Big]\\
& \qquad   - \tilde \alpha \sum_{i = 1}^{\infty}\frac{\Gamma (i)}{(2)_i i!} z^i \Big[\ln z + \Psi( i) - \Psi(1 + i) - \Psi(2+i) \Big] +\Oc( \tilde \alpha^2) \\
 &= \frac{1}{z}+(\tilde \alpha-1)\Big[\ln z -\frac{1}{\tilde \alpha}+\tilde \alpha \frac{\pi^2}{6}- \Psi(2) \Big] \\
 & \qquad   - \tilde \alpha \sum_{i = 1}^{\infty}\frac{\Gamma (i)}{(2)_i i!} z^i \Big[\ln z + \Psi( i) - \Psi(1 + i) - \Psi(2+i) \Big] +\Oc( \tilde \alpha^2) \\
&=\frac{1}{\tilde \alpha}+ \frac{1}{z}-[\ln z +\gamma] \\
& \qquad +\tilde \alpha \left(\ln z -\Psi (2)-\frac{\pi^2}{6} - \sum_{i = 1}^{\infty}\frac{\Gamma (i)}{(2)_i i!} z^i \Big[\ln z + \Psi( i) - \Psi(1 + i) - \Psi(2+i) \Big]\right) +\Oc( \tilde \alpha^2) .
\end{align*}
With this, a further refinement of \fref{id:hthetamatching} with the same computation as above yields in this case, using \fref{def:J1matching},
$$
q(z_0)=-\frac{1}{\tilde \alpha}H_1(\zeta_0)+2K_1(\zeta_0)-e_1H_1(\zeta_0)+\tilde \alpha (J_1-2-\frac{\pi^2}{6})(\zeta_0)+\Oc( |\tilde \alpha|^2),
$$
$$
z\partial_z q(z_0)= \zeta \partial_\zeta K_1(\zeta_0)+\frac{\tilde \alpha}{2}\zeta \partial_\zeta J_1(\zeta_0)+\Oc( \tilde \alpha^2),
$$
$$
\partial_{\bar \alpha} (z\partial_z q(z_0))=1-\sum_{i=1}^\infty \frac{\Gamma (i)i}{(2)_ii!2^i}\zeta_0^{2i}[2\ln \zeta_0-\ln 2 -\Psi (2+i)]+O\left(\frac{1}{|\ln b|}\right).
$$
Hence, combining these identities with the previous estimates, and using $H_1=\ln b/2+\Oc( 1)$, we obtain
\begin{align}
\nonumber \frac{z\partial_z q(z_0)}{q(z_0)}&=\frac{ \zeta \partial_\zeta K_n(\zeta_0)+\tilde \alpha \zeta \partial_\zeta J_1(\zeta_0)+\Oc( \tilde \alpha^2)}{-\frac{1}{\tilde \alpha}H_1(\zeta_0)+2K_1(\zeta_0)-e_1H_1(\zeta_0)+\tilde \alpha (J_1-2-\frac{\pi^2}{6})(\zeta_0)+\Oc( |\tilde \alpha|^2)}\\
\nonumber &=\tilde \alpha \frac{ \zeta \partial_\zeta K_n(\zeta_0)+\tilde \alpha \zeta \partial_\zeta J_1(\zeta_0)+\Oc( \tilde \alpha^2)}{-H_1(\zeta_0)+2\tilde \alpha K_1(\zeta_0)-\tilde \alpha e_1H_1(\zeta_0)+\tilde \alpha^2 (J_1-2-\frac{\pi^2}{6})(\zeta_0)+\Oc( |\tilde \alpha|^3)} \\
\nonumber &=\tilde \alpha \left[-\frac{\zeta \partial_\zeta K_1(\zeta_0)}{H_1(\zeta_0)}+\tilde \alpha \frac{\zeta \partial_\zeta K_1(\zeta_0)(e_1H_1(\zeta_0)-2K_1(\zeta_0))-\zeta \partial_\zeta J_1(\zeta_0)H_1(\zeta_0)}{H_1^2(\zeta_0)}\right]+\Oc( \tilde \alpha^3).
\end{align}
We now use the expansion $\tilde \alpha=1/\ln b+e_1/(\ln b)^2+\hat \alpha$ to derive
\begin{align}
\nonumber \frac{z\partial_z q(z_0)}{q(z_0)}&=-\frac{1}{\ln b}\frac{\zeta \partial_\zeta K_1}{H_1}\\
& \qquad -\frac{1}{|\ln b|^2}\frac{2\zeta \partial_\zeta K_1(\zeta_0) K_1(\zeta_0))+\zeta \partial_\zeta J_1(\zeta_0)H_1(\zeta_0)}{H_1^2(\zeta_0)}  -\hat \alpha \frac{\zeta \partial_\zeta K_1}{H_1}+\Oc( |\ln b|^{-3}),\label{id:refinedexteriorn=1}
\end{align}
and 
\bee
\partial_{\bar \alpha}\left(\frac{z\partial_z q(z_0)}{q(z_0)}\right) &=& -\frac{\zeta \partial_\zeta K_1(\zeta_0)}{H_1(\zeta_0)}+\Oc( \tilde \alpha^2), \quad \partial_{b}\left(\frac{z\partial_z q(z_0)}{q(z_0)}\right) =\Oc( |\tilde \alpha|^2).
\eee

\underline{The case $n=0$:}  For $n=0$, $\theta=1+\tilde \alpha$. We then refine further $\Gamma(\theta)h_{\theta}$ by noticing that for $i\geq 0$, $(1+\tilde \alpha)_i=(1)_i+\Oc( |\tilde \alpha|)$ and $\Psi (1+\tilde \alpha+i)=\Psi (1+i)+\Oc( |\tilde \alpha|)$,
\begin{align*}
\Gamma (\Theta) h_\theta(z) &=   \frac{1}{z} + \tilde \alpha \sum_{i = 0}^\infty\frac{(1+\tilde \alpha)_i}{(2)_i i!} z^i \Big[\ln z + \Psi(1+\tilde \alpha+ i) - \Psi(1 + i) - \Psi(2+i) \Big]\\
&=   \frac{1}{z} + \tilde \alpha \sum_{i = 0}^\infty\frac{(1)_i}{(2)_i i!} z^i \Big[\ln z - \Psi(2+i) \Big]+\Oc( |\tilde \alpha|^2)
\end{align*}
With this, performing the same computations as the previous ones and using $\tilde \alpha =1/\ln b+\Oc( |\ln b|^{-2})$, we obtain
$$
q(z_0)=\frac{2}{\zeta_0^2}+\tilde \alpha \left(J_0(\zeta_0)+(\gamma-\ln 2)\tilde G_0(\zeta_0)\right)+\Oc( \tilde \alpha^2),
$$
$$
z\partial_z q(z_0)= -\frac{2}{\zeta_0^2}+\frac{\tilde \alpha}{2}\left(\zeta \pa_\zeta J_0(\zeta_0)+(\gamma-\ln 2)\zeta \pa_\zeta \tilde G_0(\zeta_0)\right)+\Oc( \tilde \alpha^2),
$$
where $J_0$ and $\tilde G_0$ are defined in \fref{def:J0matching} and \fref{def:tildeG0matching}, and
$$
\partial_{\tilde \alpha} (q(z_0))= J_0(\zeta_0)-1+(\gamma-\ln 2)\tilde G_0(\zeta_0)+\Oc( |\tilde \alpha|),
$$
$$ \partial_{\tilde \alpha} (z\pa_z q(z_0))= \frac 12 \zeta \pa_\zeta J_0(\zeta_0)+\frac{\gamma-\ln 2}{2}\zeta \pa_\zeta \tilde G_0(\zeta_0) +\Oc( |\tilde \alpha|^2).
$$
Hence, using $\pa_b \tilde \alpha=-1/b|\ln b|^2$, we obtain
\begin{align}
\nonumber \frac{z\partial_z q(z_0)}{q(z_0)}&=\frac{-\frac{2}{\zeta_0^2}+\frac{\tilde \alpha}{2}\left(\zeta \pa_\zeta J_0(\zeta_0)+(\gamma-\ln 2)\zeta \pa_\zeta \tilde G_0(\zeta_0)\right)+\Oc( \tilde \alpha^2)}{\frac{2}{\zeta_0^2}+\tilde \alpha \left(J_0(\zeta_0)+(\gamma-\ln 2)\tilde G_0(\zeta_0)\right)+\Oc( \tilde \alpha^2)}\\
\label{id:refinedexteriorn=0} & =-1+\tilde \alpha \zeta_0^2\left(\frac 14 \zeta \pa_\zeta J_0+\frac{\gamma-\ln 2}{4}\zeta \pa_\zeta \tilde G_0(\zeta_0)+\frac 12 J_0 +\frac{\gamma-\ln 2}{2}\tilde G_1(\zeta_0)\right)+\Oc( \tilde \alpha^2),
\end{align}
\begin{align}
\label{id:refinedexteriorn=0pab} \pa_{\bar \alpha}\left(\frac{z\partial_z q(z_0)}{q(z_0)}\right)&=\zeta_0^2\left(\frac 14 \zeta \pa_\zeta J_0+\frac{\gamma-\ln 2}{4}\zeta \pa_\zeta \tilde G_0(\zeta_0)+\frac 12 J_0+\frac{\gamma-\ln 2}{2}\tilde G_1(\zeta_0)\right)+\Oc( |\tilde \alpha|),
\end{align}
\be \label{id:refinedexteriorn=0paalpha}
\pa_{b}\left(\frac{z\partial_z q(z_0)}{q(z_0)}\right)=\Oc\left(\frac{1}{|\ln b|^2}\right).
\ee

\noindent \textbf{Step 3} \emph{Existence of $\tilde \alpha_n$, proof of \fref{est:nu0ntil} and \fref{est:nurefinedn=01}}. We first prove the existence and the bound for $\tilde \alpha_n$, and then prove a bound for $\pa_b \tilde \alpha_n$. From \fref{def:Thetaeps0}, \fref{est:phin0D1}, \fref{id:extmatching1} we arrive at the following. \\

\noindent \underline{The case $n\geq 2$}. In this case, we have 
\begin{align*}
\Theta(b, \bar{\alpha}) &= \frac{\zeta \partial_\zeta H_n(\zeta_0)}{2H_n(\zeta_0)}+\frac{1}{\ln b}\frac{K_n(\zeta_0)\zeta \partial_\zeta H_n(\zeta_0)-H_n(\zeta_0)\zeta\partial_\zeta K_n(\zeta_0)}{H_n^2(\zeta_0)}+\bar \alpha \frac{\Oc( 1)}{H_n(\zeta_0)^2}+\Oc( |\ln b|^{-2})\\
&-\frac 12 \frac{\zeta \partial_\zeta H_n(\zeta_0)}{H_n(\zeta_0)}-\tilde \alpha \frac{K_n(\zeta_0)\zeta \partial_\zeta H_n(\zeta_0)-\zeta \partial_\zeta K_n(\zeta_0)H_n(\zeta_0)}{H_n^2(\zeta_0)}+\Oc( \tilde \alpha^2) \\
&=\bar \alpha \frac{-K_n(\zeta_0)\zeta \partial_\zeta H_n(\zeta_0)+\zeta \partial_\zeta K_n(\zeta_0)H_n(\zeta_0)+\Oc( 1)}{H_n^2(\zeta_0)}+\Oc( |\ln b|^{-2}) \\
\end{align*}
where the constant in the $\Oc( 1)$ is independent of $\zeta_0$, and the constant in the $\Oc( |\ln b|^{-2}) $ is independent of $\bar \alpha$. We compute for $n\geq 1$ from \fref{def:Hn} the nondegeneracy for $\zeta_0$ small enough, as $\hat d_1=-1/2$ and $c_{n,1}=2n$:
\begin{align}
\nonumber  -K_n \zeta \partial_\zeta H_n& +\zeta \partial_\zeta K_nH_n \\
\nonumber &=-\left(\frac{1}{\zeta_0^2}+\sum_{i=1}^{n} c_{n,i} \zeta_0^{2(i-1)}\left( \hat d_{i}  \ln \zeta_0 + d_{i} \right)\right)\left(\sum_{i=1}^{n}  2(i-1)c_{n,i} \hat d_i\zeta_0^{2(i-1)}\right)\\
\nonumber &+\left(\frac{-2}{\zeta_0^2}+\sum_{i=1}^{n} c_{n,i} \zeta_0^{2(i-1)}\left( 2(i-1)(\hat d_{i}  \ln \zeta_0 + d_{i})+\hat d_i\right)\right)\left(\sum_{i=1}^{n} c_{n,i} \hat d_i\zeta_0^{2(i-1)}\right)\\
 &=-\left(\frac{1}{\zeta_0^2}+\Oc( |\ln \zeta_0|)\right) \left(\Oc( \zeta_0^2)\right)+\left(\frac{-2}{\zeta_0^2}+\Oc( 1)\right)\left(-n+\Oc( \zeta_0^2)\right) =\frac{2n}{\zeta_0^2}+\Oc( 1). \label{id:matchingnondegeneracy}
\end{align}
So that, as $H_n(\zeta_0)=-n+\Oc( \zeta_0^2)$ we arrive at:
$$
\Theta(b, \bar{\alpha}) =\bar \alpha \left(\frac{2}{n \zeta_0^2}+\Oc( 1)\right)+\Oc( |\ln b|^{-2}).
$$
An application of the intermediate value theorem then yields that there exists at least one value $\bar \alpha=\bar \alpha_n=\Oc( |\ln b|^{-2})$ such that $\Theta(b,\bar \alpha)=0$. \\

\noindent \underline{The case $n=1$}. We obtain from the refined identities \fref{id:refinedinteriorn=1} and \fref{id:refinedexteriorn=1}:
\bee
\Theta & = & -\frac{1}{\ln b} \frac{\zeta\partial_\zeta K_1}{H_1}-\frac{1}{\ln b^2}\frac{\zeta \partial_\zeta J_1H_1+2K_1\zeta \partial_\zeta K_1}{H_1^2}+\tilde \alpha \frac{\Oc( 1)}{H_1^2}+\Oc( |\ln b|^{-3}) \\
&&-\left(-\frac{1}{\ln b}\frac{\zeta \partial_\zeta K_1}{H_1}-\frac{1}{|\ln b|^2}\frac{2\zeta \partial_\zeta K_1(\zeta_0) K_1(\zeta_0))+\zeta \partial_\zeta J_1(\zeta_0)H_1(\zeta_0)}{H_1^2(\zeta_0)}  -\hat \alpha \frac{\zeta \partial_\zeta K_1}{H_1}+\Oc( |\ln b|^{-3})\right)\\
&=& \hat \alpha  \frac{\zeta \partial_\zeta K_1+\Oc( 1)}{(H_1)^2}+O\left(\frac{1}{|\ln b|^3}\right).
\eee
From the nondegeneracy \fref{id:matchingnondegeneracy}, an application of the intermediate value Theorem yields that there exists at least one value $\hat \alpha=\hat \alpha_1=\Oc( |\ln b|^{-3})$ such that $\Theta=0$.\\

\noindent \underline{The case $n=0$}. We obtain from the identities \fref{id:refinedinteriorn=0} and \fref{id:refinedexteriorn=0}, injecting $\tilde \alpha=1/\ln b+e_0/(\ln b)^2+\hat \alpha$ with $e_0=\ln 2-\gamma$ and $\hat \alpha=\Oc( |\ln b|^{-3})$:
\bee
\Theta&=&-1+\frac{1}{2\ln b}\zeta_0^2(\frac 12 \zeta \pa_\zeta J_0+J_0)-\frac{\ln b}{4} \bar \alpha \zeta_0^2(\zeta \pa_\zeta \tilde G_0(\zeta_0)+2\tilde G_0(\zeta_0)+\Oc( |\ln b|^{-1}))+\Oc( |\ln b|^{-2})\\
&&-\left(-1+\tilde \alpha \zeta_0^2\left(\frac 14 \zeta \pa_\zeta J_0+\frac{\gamma-\ln 2}{4}\zeta \pa_\zeta \tilde G_0(\zeta_0)+\frac 12 J_0 +\frac{\gamma-\ln 2}{2}\tilde G_1(\zeta_0)\right)+\Oc( \tilde \alpha^2)\right)\\
&=&-\frac{\ln b}{4} \bar \alpha \zeta_0^2(\zeta \pa_\zeta \tilde G_0(\zeta_0)+2\tilde G_0(\zeta_0)+\Oc( |\ln b|^{-1}))+\Oc( |\ln b|^{-2})\\
&&+ \frac{\ln 2-\gamma}{4\ln b} \zeta_0^2\left( \zeta \pa_\zeta \tilde G_0(\zeta_0) +2\tilde G_0(\zeta_0)\right)\\
&&-\left(\bar \alpha \zeta_0^2\left(\frac 14 \zeta \pa_\zeta J_0+\frac{\gamma-\ln 2}{4}\zeta \pa_\zeta \tilde G_0(\zeta_0)+\frac 12 J_0 +\frac{\gamma-\ln 2}{2}\tilde G_0(\zeta_0)\right)\right)\\
&=&-\frac{\ln b}{4} \hat \alpha \zeta_0^2(\zeta \pa_\zeta \tilde G_0(\zeta_0)+2\tilde G_0(\zeta_0)+\Oc( |\ln b|^{-1}))+\Oc( |\ln b|^{-2}).
\eee
Therefore, as $\zeta \pa_\zeta \tilde G_0(\zeta_0)+2\tilde G_0(\zeta_0)\neq 0$ for $\zeta_0$ small enough, an application of the implicit function Theorem gives the existence of $\hat \alpha=\hat \alpha_0=\Oc( |\ln b|^{-3})$ such that $\Theta(b,\hat \alpha_0)=0$.

\noindent \underline{Estimate of $\partial_b \tilde{\alpha}_n$:} We estimate for $n\geq 1$ from \fref{est:phin0D1pab}, \fref{est:phin0D1pab}, \fref{id:extmatching2}, \fref{id:extmatching3} and \fref{id:matchingnondegeneracy},
$$
\pa_b \Theta =\pa_b\left(\frac{r\pa_r \phi_n^\inn (R_0))}{2\phi_n^\inn (R_0)}\right) -\pa_b\left(\frac{z\pa_z q(z_0)}{q(z_0)}\right)=\Oc( b^{-1}|\ln b|^{-2}),
$$
and 
\begin{align*}
\pa_{\bar \alpha} \Theta &=\pa_{\bar \alpha }\left(\frac{r\pa_r \phi_n^\inn (R_0))}{2\phi_n^\inn (R_0)}\right) -\pa_{\bar \alpha }\left(\frac{z\pa_z q(z_0)}{q(z_0)}\right)\\
& =\frac{-K_n(\zeta_0)\zeta \partial_\zeta H_n(\zeta_0)+\zeta \partial_\zeta K_n(\zeta_0)H_n(\zeta_0)+\Oc( 1)}{H_n^2(\zeta_0)}=\frac{2}{n\zeta_0^2}+\Oc( 1).
\end{align*}
Therefore, differentiating the fixed point relation $\Theta(b,\bar \alpha(b))=0$ gives $\pa_{b}\bar \alpha \pa_{\bar \alpha}\Theta=-\pa_b \Theta$, so $|\partial_b \bar{\alpha}_n| = \left|\frac{\partial_b \Theta}{\partial_{\tilde{\alpha}_n}\Theta} \right| = \Oc\left(\frac{1}{b |\ln b|}\right)$ which concludes the proof of \fref{est:nu0ntil} for $n\geq 1$. For  $n=0$ the very same computation yields the same estimate, using \fref{id:refinedinteriorn=0pab}, \fref{id:refinedinteriorn=0paalpha}, \fref{id:refinedexteriorn=0pab} and \fref{id:refinedexteriorn=0paalpha}.\\

\noindent \textbf{Step 4:} \emph{Proof of the refined pointwise estimate \fref{bd:refinedtildephin}}. Recall $\tilde \phi_n$ is defined by \fref{id:deftildephin}. By \eqref{def:Mn}, we estimate $\tilde \phi$ in two zones, $r\leq R_0$ and $r\geq R_0$.\\
- For $r \leq R_0$, We write from \eqref{eq:forMinn}:
$$
\tilde{\phi}_{n} = b\left(-\frac{2}{\ln b}T_1+\As_0^{-1}\Theta_0\right)(r)+ 2\bar{\alpha}\sum_{j = 0}^n b^{j + 1}\Big(-c_{n,j}T_{j + 1}(r) + S_{j}(r)\Big)  + b\Rc_{n}(r).
$$
Then, the estimates \fref{bd:improveddecayTheta0T0}, \fref{est:T0iatinf}, \eqref{est:S0j}, \fref{id:refinementSjn=0}, \fref{id:refinementSjn=1}, \eqref{est:R0n}, \fref{id:refinementRn=0} and \fref{id:refinementRn=1} imply that $\| \tilde \phi_n \|_{\mathcal I^0_{-1}}\lesssim b$ which means that for $r\leq R_0$:
\be \label{bd:tildephinintpointwise}
|\tilde{\phi}_{n}(r)|\lesssim b r^{2}\langle r \rangle^{-2}\left(1+2\frac{\ln (r+1)}{\ln b} \right) \lesssim \left| \begin{array}{l l} \frac{1}{|\ln b|}r^{2}\langle r\rangle^{-4} \qquad \mbox{on }[0,R_0],\\ br^2\langle r \rangle^{-2}\qquad \mbox{on }[0,R_0] \mbox{ as well.} \end{array} \right.
\ee
- For $r \geq R_0$, we switch to $\zeta= \sqrt{b} r$ variables and write from \eqref{def:Mn} and \eqref{eq:qn}:
\be \label{id:exprtildephinext}
\tilde{\phi}_{n}(\frac{\zeta}{\sqrt b}) = \beta_0 \big( \Gamma (\theta) h_{\theta_n} + \Gc \big)\big(\frac{\zeta^2}{4} \big) - \sum_{j = 0}^n c_{n,j}b^j T_j\big(\frac{\zeta}{\sqrt{b}}\big),
\ee
We first estimate the parameter $\beta_0$, which from \fref{def:Thetaeps0}, \fref{id:exprqz0} and \fref{id:exprphiinR0} is:
$$
\beta_0=\frac{\phi^\inn_n (R_0)}{\phi_n^\out (R_0)}=\frac{b\left(-\frac{\ln b}{2}H_n+K_n+O\left(|\ln b|^{-1}\right)\right)}{-\frac{1}{\tilde \alpha}H_n+2K_n-e_nH_n+O(|\ln b|^{-1})}
$$
We deduce from the above identity, using that $\tilde \alpha=(\ln b)^{-1}+O(|\ln b|^{-2})$:
$$
\beta_0=\frac b2+O\left(\frac{b}{|\ln b|} \right) \qquad \mbox{and} \qquad \frac{\beta_0}{\tilde \alpha}=\frac{b \ln b}{2}-\frac b2 e_n+O\left(\frac{b}{|\ln b|^2} \right).
$$
The $O(|\tilde \alpha |)$ remainder in \fref{idhthetaext1} can be bounded by an explicit weight on $[z_0,\infty)$, for example via the same perturbation argument as used in the proof of Lemma \ref{lemm:out0}. We do not repeat such an argument which shows that, since $\tilde \alpha=O(|\ln b|^{-1})$:
$$
\Gamma(\theta)h_\theta(z)=\frac 1z+\sum_{i=1}^{n} 2^{i-1} c_{n,i} z^{i-1}\left( \hat d_{i} \left( \ln z+\ln 2 -e_n-\frac{1}{\tilde \alpha}\right) +2 d_{i} \right)+\Oc\left( \frac{1}{|\ln b|}z^{n-1+\delta}\right)
$$
for any $\delta>0$. The two above identities then imply the identity for the first term in \fref{id:exprtildephinext}:
$$
\beta_0 \Gamma (\theta) h_{\theta_n}\left(\frac{\zeta^2}{2}\right)= \frac{b}{\zeta^2}+\frac{b}{2}\sum_{i=1}^n c_{n,i}\zeta^{2(i-1)}\left(\hat d_i(2\ln \zeta-\ln b)+2d_i \right)+\Oc\left( \frac{b}{|\ln b|}\zeta^{2n-2+\delta}\right).
$$
Next we turn to the third term in \fref{id:exprtildephinext}, which from \fref{est:T0iatinf} is for $\zeta \geq \zeta_0$:
$$
\sum_{j = 0}^n c_{n,j}b^j T_j\big(\frac{\zeta}{\sqrt{b}}\big)=\frac{b}{\zeta^2}+b\sum_{i=1}^n c_{n,i}\zeta^{2(i-1)}\left(\hat d_i (\ln \zeta-\frac{\ln b}{2})+d_i\right)+O(b^2|\ln b|^C\zeta^{2n-4}|\ln \zeta|^C)
$$
for some constant $C>0$. One thus has in \fref{id:exprtildephinext} a cancellation for the leading order terms, and combined with the estimate \fref{est:Gc} for $\mathcal G$ this yields:
\be \label{bd:tildephinextpointwise}
|\tilde{\phi}_{n}(\frac{\zeta}{\sqrt b}) |\lesssim \frac{b}{|\ln b|} \zeta^{2n-2+\delta}.
\ee
- Conclusion : Combining \fref{id:exprtildephinext} and \fref{bd:tildephinextpointwise}, recalling $r=\frac{\zeta}{\sqrt b}$ we see that:
$$
|\tilde{\phi}_{n}(r) |\lesssim \frac{1}{|\ln b|} r^2 \langle r \rangle^{-4}\langle \sqrt b r \rangle^{2n+\delta} \quad \mbox{and} \quad |\tilde{\phi}_{n}(r) |\lesssim b r^2 \langle r \rangle^{-2}\langle \sqrt b r \rangle^{2n+\delta},
$$
which is precisely the first bound in \fref{bd:refinedtildephin} with $k=0$. The first bound in \fref{bd:refinedtildephin} for $k=1,2$, and the second bound in \fref{bd:refinedtildephin} for $k=0,1,2$ are proved the exact same way, using that the bounds on the corrective terms \fref{est:S0j} and \fref{est:R0n}, and \fref{est:Gc} provide the desired control for $D_r$, $\pa_b$ and $\pa_\alpha$ derivatives, along with the estimate $b\pa_{b}\tilde \alpha =O(|\ln b|^{-2})$ that was proved in Step 3.
\end{proof}

\begin{proof}[Proof of Corollary \ref{coro:perturbedspectral}]
We claim that the same proof applies as for Lemma \ref{lemm:radialmode}. Indeed, notice that from Lemma \ref{lem:perturbationinner} and the bound \fref{lem:perturbationinner}, the inner solution for the perturbed problem is of the very same form as the original problem \fref{id:innerdecompositionmatching}:
$$
\phi^{\inn,V}_n[b,\bar \alpha](r) = F_n[b](r) + \bar{\alpha}b G_n[b,\bar \alpha](r) +E^V_n[b,\bar \alpha](r),
$$
where $E^V_n=E_n+\phi^{\inn,V}-\phi^{\inn}$ satisfies the analogue of \fref{bd:Enmatching}:
$$
\sum_{0\leq k \leq 2}^2| ((r \partial_r)^k E_n) (R_0)| \leq C(\zeta_0) \frac{b}{|\ln b|}.
$$
So all computations made for the inner solution of the original problem are also valid for the perturbed problem. Notice similarly from Lemma \ref{lem:perturbationouter} that the outer solution for the perturbed problem is of the very same form as that of the original problem:
$$
q^V_n[b,\bar \alpha](z)=\Gamma(\theta)h(\theta)+\Gc_n^V[b,\bar \alpha](z)
$$
where $\Gc$ satisfies the analogue of \fref{bd:mathgexterieur}:
$$
|\mathcal G_n^V(z_0))|+|(z\pa_z \mathcal G_n^V(z_0))| \lesssim b^{\frac 12}.
$$
So all computations made for the outer solution of the original problem are also valid for the perturbed problem. The matching procedure can thus be done verbatim the same way. The only informations that we do not get in comparison with the original problem are the estimates for the variation with respect to $\tilde \alpha$ and $b$, and the next order $|\ln b|^{-2}$ term in the expansion of $\tilde \alpha$ for $n=0,1$, but these informations are not required. This concludes the proof of the Corollary.
\end{proof}

\begin{proof}[Proof of Proposition \ref{pr:spectralbarAzeta}] The existence part and the estimates on the eigenvalues are direct consequences of Corollary \ref{coro:perturbedspectral}. The bound \fref{bd:stabilityeigenmodes} is a direct consequence of \fref{lem:perturbationinner} and \fref{lem:perturbationouter}.
\end{proof}

\section{Coercivity in the non-radial sector, Proof of Proposition \ref{pr:coercivitenonradial}}\label{sec:CoerNonR}
Our argument takes place on the stationary state variables:
$$
\Ls u =\Delta u-\nabla  \cdot (u \nabla \Phi_U)-\nabla \cdot (U\nabla \Phi_u)-b\nabla \cdot (yu), \quad  0<b = \nu^2 \beta \ll 1, \quad y=\frac{z}{\sqrt \beta \nu}.
$$
The operator $\Ls$ can be written in two different divergence forms
\begin{align}
\Ls u = \Ls_0 u -b\nabla \cdot (yu) \quad\mbox{or}\quad \Ls u = \Hs u - \nabla U \cdot \nabla \Phi_u, \label{def:LsformH}
\end{align}
where $\Ls_0$ is defined in \fref{def:Lsform0}, and $\Hs u = (\omega[b])^{-1}\nabla \cdot \Big(\omega [b]\nabla u\Big) + 2(U - b)u,$ with the weight functions (we will often forget about the $[b]$ dependance from now on in this section)
\begin{equation}
\omega = \omega[b] = \frac{\rho[b]}{U}, \quad \rho[b] (y)= e^{-\frac{b| y|^2} {2}}.
\end{equation}
In the first form in \fref{def:LsformH}, the $b\nabla \cdot (yu)$ term can be treated as a perturbation up to the zone $|y|\sim \sqrt{b}$. In the second, the term $\nabla U \cdot \nabla \Phi_u$ formally scales like "$|y|^{-4}u$" at infinity due to the rapid decay of $U$ and is expected to be of lower order there. The mixed scalar product \fref{def:quadform} is adapted to these two structures. We make a slight abuse of notations and keep the same notation for it in $y$ variables:
\begin{equation}\label{def:scalarMtil}
\la u,v\ra_\ast := \int_{\Rb^2} u \sqrt{\rho} \Ms\Big(v \sqrt{\rho} \Big) dy= \int_{\Rb^2} u \tilde{\Ms} v \rho dy,
\end{equation}
where $\tilde{\Ms}$ is the linear operator with a suitably truncated Poisson field:
\begin{equation}\label{def:Phiutil}
\tilde{\Ms}= \tilde{\Ms}[b]: u \mapsto \frac{u}{U} - \tilde{\Phi}_u, \qquad \tilde \Phi_u= \tilde \Phi[b]_u= -\frac{1}{\sqrt{\rho}}\left[\frac{1}{2\pi} \ln(|y|)*\left(u \sqrt{\rho}\right)\right].
\end{equation}
Note that $\tilde{\Ms} v = \sqrt{\rho}^{-1} \Ms\Big(v \sqrt{\rho} \Big)$ so there holds in particular the relations:
$$
-\Delta \left(\tilde \Phi_u \, \sqrt{\rho}\right)= u\sqrt{\rho} \quad \textup{and}\quad \Delta \tilde \Phi_u=-u+ b y \cdot \nabla \tilde \Phi_u+\left(b+\frac{b^2}{4}|y|^2\right)\tilde \Phi_u.
$$
We shall consider the operator $\tilde{\Ls}$ which is the operator $\tilde \Ls^z$ defined by \fref{def:tildeLz} expressed in $y$ variable:
$$
\tilde \Ls u :=\Delta u-\nabla  \cdot (u \nabla \Phi_{U})-\nabla \cdot (U \nabla \tilde \Phi_u)-b\nabla \cdot (yu),
$$
To prove Proposition \ref{pr:coercivitenonradial} is then equivalent to prove its analogue in $y$ variables:

\begin{proposition}\label{pr:coercivitenonradialbis}
There exists $c,C>0$ and $b^*>0$ such that for all $0<b\leq b^*$, if $u\in \dot H^1_{\omega[b]}$ satisfies $\int_{|y|=r}u=0$ for almost every $r>0$, then:
\begin{equation}\label{est:coerLnuybis}
\langle -\tilde \Ls u,u\rangle_\ast \geq c \| \nabla u \|_{L^2_\omega}^2 - C \left( \left(\int_{\Rb^2} u\partial_{y_1}U\sqrt{\rho}dy\right)^2+\left(\int_{\Rb^2} u\partial_{y_2}U\sqrt{\rho}dy\right)^2\right).
\end{equation}
\end{proposition}

The proof is done in two parts: In the first part, we deal with the linear operator $\Ls_0$ and derive its coercivity under some suitable orthogonality conditions. Then, we extend this coercive property to the full linearized operator $\tilde \Ls$ where the scaling term $\nabla\cdot(yu)$ is taken into account.

\subsection{Coercivity of $\Ls_0$ in $\dot H^1$}

We recall that $\Ls_0$, at the $L^2$ level, satisfies the continuity estimate \fref{eq:contM} and the coercivity \fref{bd:coercivite L2} from \cite{RSma14}. We prove here a coercivity at the $\dot H^1$ level. While \cite{RSma14} proves a similar estimate at the $\dot H^2$ level, we state and prove in an analogous way the following result for the sake of completeness.
\begin{lemma} \label{lemm:coerML0} Let $u$ be such that $\int_{\Rb^2} u dy= 0$ and $ \nabla u \in L^2(U^{-1})$. Then, we have for some constants $\delta_2>0$ and $C>0$:
\be \label{bd:coercivite H1}
\int_{\Rb^2} U |\nabla (\Ms u)|^2 dy \geq  \delta_2\int_{\Rb^2} \frac{|\nabla u|^2}{U} dy - C\Big[\la u, \pa_1 U\ra_{L^2}^2 + \la u, \pa_2 U\ra_{L^2}^2 \Big].
\ee
\end{lemma}
\begin{proof}
We first prove that the projections are well-defined. This is a consequence of the following Hardy-type inequality:
\be \label{id:hardy}
\int_{\mathbb R^2} u^2(1+|y|^2)dy\lesssim \int_{\mathbb R^2}|\nabla u|^2(1+|y|^4)dy,
\ee
and of the decay $|U|\lesssim (1+|y|)^{-4}$:
\bee
 \la u, \pa_i U\ra_{L^2}^2  & \lesssim & \left(\int_{\Rb^2} |u|^2(1+|y|)^2 \right)^{\frac 12} \lesssim \left(\int_{\Rb^2} |\nabla u|^2(1+|y|)^4 \right)^{\frac 12} \lesssim  \left(\int_{\Rb^2} \frac{|\nabla u|^2}{U}\right)^{\frac 12} .
\eee
\textbf{Step 1} \emph{Subcoercivity estimate}: We use Young's inequality $ab\leq a^2/4+b^2$ to obtain:
\bee
\int_{\Rb^2} U |\nabla (\Ms u)|^2&=& \int_{\Rb^2} U \left(\left|\nabla \left(\frac u U\right)\right|^2+2\nabla  \left(\frac u U\right) \cdot \nabla \Phi_u+\left| \nabla \Phi_u \right|^2 \right) \\
&\geq & \frac 12 \int_{\Rb^2} U\left|\nabla \left(\frac u U\right)\right|^2-\int_{\Rb^2} U \left| \nabla \Phi_u\right|^2.
\eee
From the algebraic identity $
\int_{\Rb^2} U\left|\nabla \left(\frac u U\right)\right|^2=\int_{\Rb^2} \frac{|\nabla u|^2}{U}-\int_{\Rb^2} Uu^2$, the control of the Poisson field \fref{bd:poisson1} $
\int_{\Rb^2} U|\nabla \Phi_u|^2 \lesssim \int u^2$, and the decay $U(y)\lesssim (1+|y|)^{-4}$, one gets the following subcoercive estimate for some $C>0$:
\be \label{id:subcoercnr1}
\int_{\Rb^2} U |\nabla (\Ms u)|^2 \geq  \frac 12 \int_{\Rb^2} \frac{|\nabla u|^2}{U}-C \int_{\Rb^2} |u|^2.
\ee

\noindent \textbf{Step 2} \emph{Coercivity estimate}: We apply a standard minimisation technique. Assume by contradiction \fref{bd:coercivite H1} is false. Then there exists a sequence of functions $(u_n)_{n\in \mathbb N}\in \dot H^1((1+|y|)^4dy)$ without radial component such that
\be \label{id:coercnr1}
\int_{\Rb^2} \frac{|\nabla u_n|^2}{U}=1, \quad \int_{\Rb^2} u_n\pa_{y_i}U=0 \quad \text{for} \ i=1,2, \quad \int_{\Rb^2} U |\nabla (\Ms u_n)|^2 \rightarrow 0.
\ee
Up to a subsequence there exists a limit $u_{\infty}$ of $u_{n}$ in $H^1_{\text{loc}}$. Moreover, from the lower semi-continuity and the weak continuity, we have
$$
\int_{\Rb^2} \frac{|\nabla u_{\infty}|^2}{U}\leq 1, \quad \int_{\Rb^2} u_{\infty}\pa_{y_i}U=0 \quad \text{for} \ i=1,2.
$$
We now write
$$
 \int_{\Rb^2} U |\nabla (\Ms u_n)|^2=\int_{\Rb^2} \frac{|\nabla u_n|^2}{U}-\int_{\Rb^2} Uu_n^2.
$$
Above, $\frac{\nabla u_n}{U}$ converges weakly in $L^2(U\,dy)$. We remark that
$$
\int_{\Rb^2} u_n^2(1+|y|^2)\lesssim \int_{\Rb^2} \frac{|\nabla u_n|^2}{U}.
$$
From this and from the compactness of the embedding of $H^1(\Omega)$ in $L^2(\Omega)$ for $\Omega$ compact, $u_n$ converges strongly in $L^2(dy)$. Hence, from \fref{id:coercnr1} and lower semi-continuity:
$$
 \int_{\Rb^2} U |\nabla (\Ms u_{\infty})|^2=\int_{\Rb^2} \frac{|\nabla u_{\infty}|^2}{U}-\int_{\Rb^2} Uu_{\infty}^2\lesssim 0.
$$
Therefore, $\nabla \Ms u_{\infty}=0$. Since $u_{\infty}$ is without radial component, one obtains $\Ms u_{\infty}=0$. Hence, $u_{\infty}$ belongs to the Kernel of $\Ms$ intersected with $L^2((1+|y|)^2dy)$, which is $\text{Span}(\pa_{y_1}U,\pa_{y_2}U)$. From the orthogonality condition \fref{id:coercnr1}, one gets that necessarily $u_{\infty}=0$. From the subcoercivity estimate \fref{id:subcoercnr1},

$$
\int_{\Rb^2} |u_n|^2\geq \frac 1C \left( \frac 12 \int_{\Rb^2} \frac{|\nabla u_n|^2}{U}-\int_{\Rb^2} U |\nabla (\Ms u_n)|^2 \right),
$$
and hence from \fref{id:coercnr1}: $\liminf \int_{\Rb^2} |u_n|^2 \geq \frac 1C>0$. As $u_n$ converges strongly in $L^2(dy)$, this implies $\displaystyle \int_{\Rb^2} |u_\infty|^2 \neq 0$ which contradicts $u_{\infty}=0$. This concludes the proof of Lemma \ref{lemm:coerML0}.
\end{proof}

\subsection{Coercivity of $\tilde \Ls$, Proof of Proposition \ref{pr:coercivitenonradialbis}}
We are now in the position to conclude the proof of Proposition \ref{pr:coercivitenonradialbis} thanks to Lemma \ref{lemm:coerML0}. By noting that $\Delta u - \nabla \Phi_U \cdot \nabla u + uU = \nabla \cdot \left[U \nabla \Big(\dfrac{u}{U} \Big) \right] $ and
$$
Uu - \nabla U \cdot \nabla \tilde \Phi_u = - \nabla \cdot \left( U \nabla \tilde \Phi_u \right)-  b U y \cdot \nabla \tilde \Phi_u - \left(b+\frac{b^2}4|y|^2\right) U\tilde \Phi_u,
$$
we rewrite the linear operator $\tilde \Ls$ in terms of $\tilde{\Ms}$ as follows:
$$
\tilde \Ls u= \nabla \cdot \left( U \nabla \tilde{\Ms}  u  - b yu\right) -  b  U y\cdot\nabla \tilde \Phi_u  -  \left( b+\frac{b^2}{4}|y|^2\right) U \tilde \Phi_u.
$$
One has the identity
\begin{align*}
& -\int \nabla \cdot \left(U\nabla \tilde{\Ms}  u - b y u\right)\tilde{\Ms}  v \rho dy\\
& \qquad = \int U\nabla \tilde{\Ms}  u \cdot \nabla \tilde{\Ms}  v \rho + b\int y.\nabla \Phi_U u \tilde{\Ms}  v \sqrt{\rho} + b \int U y \cdot \nabla \tilde \Phi_u \tilde{\Ms}  v \rho dy +  2b \int u \tilde{\Ms}  v \rho.
\end{align*}
This leads to the following almost self-adjointness of $\tilde{\Ls}$:
\begin{equation} \label{def:bilinearformLtil}
\langle -\tilde \Ls u,v\rangle_\ast = F(u,v)+G(u,v)+2b\langle  u,v\rangle_\ast,
\end{equation}
where $F$ is the leading order part given by
$$
F(u,v):=\int_{\mathbb R^2} U\nabla \tilde{\Ms}  u \cdot \nabla \tilde{\Ms}  v \rho dy + b\int_{\mathbb R^2} y \cdot \nabla \Phi_U u \tilde{\Ms}  v \sqrt{\rho},
$$
and $G$ contains lower order terms, 
$$
G(u,v):= \int_{\mathbb R^2} \left( 2b U y \cdot \nabla \tilde \Phi_u+ \left(b +\frac{b^2}4 |y|^2\right) U \tilde \Phi_u \right)\tilde{\Ms}  v \rho dy.
$$

\begin{proof}[Proof of Proposition \ref{pr:coercivitenonradialbis}] To prove \eqref{est:coerLnuybis} we proceed in two steps:\\

\noindent \textbf{Step 1} \emph{Subcoercivity estimate}: We claim that for $u\in \dot H^1_\omega$:
\be \label{bd:subcoercivity}
F(u,u)+G(u,u)=\| \nabla u \|_{L^2_\omega}^2+\Oc \left(\| \nabla u \|_{L^2_\omega} \Big\| \frac{u}{1+|y|^{\frac 32}} \Big\|_{L^2_\omega}+\Big\| \frac{u}{1+|y|^{\frac 32}} \Big\|_{L^2_\omega}^2+b^{\frac 14} \| \nabla u \|_{L^2_\omega}^2  \right),
\ee
where the constant in the $\Oc(\cdot)$ does not depend on $b$. Let us begin with the form $F$ by writing

\bee
F(u,u)&=&\int_{\mathbb R^2} U \left|\nabla \left( \frac{u}{U} \right) \right|^2 \rho + b\int_{\mathbb R^2} \frac{y \cdot \nabla \Phi_U}{U}u^2 \rho\\
&& \quad -2\int_{\mathbb R^2} U \nabla \left( \frac{u}{U} \right) \cdot \nabla \tilde \Phi_u \rho +\int_{\mathbb R^2} U | \nabla \tilde \Phi_u |^2\rho.
\eee

The first line gathers the leading order terms at infinity. We compute

\begin{align*}
&\int_{\mathbb R^2} U \left|\nabla \left( \frac{u}{U} \right) \right|^2\rho +b\int_{\mathbb R^2} \frac{y \cdot \nabla \Phi_U}{U}u^2 \rho\\
& \qquad =\int_{\mathbb R^2} \frac{|\nabla u|^2}{U}\rho  -2\int_{\mathbb R^2} \frac{u}{U}\nabla u \cdot \nabla \Phi_U \rho +\int_{\mathbb R^2} \frac{u^2|\nabla \Phi_U|^2}{U}\rho + b\int_{\mathbb R^2} \frac{y \cdot\nabla \Phi_U}{U}u^2 \rho  \\
& \qquad = \int_{\mathbb R^2} \frac{|\nabla u|^2}{U} \rho +\int_{\mathbb R^2} u^2\nabla \cdot \left(\frac{\nabla \Phi_U}{U}\right) \rho +\int_{\mathbb R^2} \frac{u^2|\nabla \Phi_U|^2}{U}\rho = \| \nabla u\|^2_{L^2_\omega}-\int_{\mathbb R^2} u^2\rho.
\end{align*}

Thus, we have

\begin{align} \label{eq:Fexpression}
F(u,u)&=\| \nabla u\|^2_{L^2_\omega} -\int_{\mathbb R^2} u^2 \rho -2\int_{\mathbb R^2} U \nabla \left( \frac{u}{U} \right) \cdot \nabla \tilde \Phi_u \rho +\int_{\mathbb R^2} U| \nabla \tilde \Phi_u |^2\rho.
\end{align}

From \fref{bd:poisson1} with $\alpha=7/4$, and \fref{bd:generalisedhardy} with $\alpha=1/2$ we get:
\be \label{bd:poissonfieldpolynomial1}
b^{\frac 34} |\tilde \Phi_u (y)|^2 \lesssim \rho^{-1}(1+ |y|)^{-\frac 32} b^{\frac 34}\int_{\mathbb R^2}  |u|^2(1+|y|)^{\frac{7}{2}}e^{-\frac{b|y|^2}{2}}dy\lesssim \rho^{-1}(1+|y|)^{-\frac 32} \| u\|_{\dot H^1_\omega}^2.
\ee
As $\nabla \tilde \Phi_u=\nabla (\rho^{-1/2}\Phi_{\rho^{1/2} u})$, using the above inequality, and \fref{bd:poisson1} with $\alpha=1/2$, we obtain:
\be \label{bd:poissonfieldpolynomial2}
|\nabla \tilde \Phi_u(y)|^2\lesssim \rho^{-1}(1+|y|)^{-1}(1+\mathbbm 1(|y|\leq 1)\ln |y|) \int_{\mathbb R^2} u^2(1+|y|)\rho+ b^{\frac 54}\rho^{-1}(1+|y|)^{\frac 12} \| u\|_{\dot H^1_\omega}^2.
\ee
From the above estimate and the decay $|U(y)|\lesssim (1+|y|)^{-4}$, we obtain
\bee
\int_{\mathbb R^2} U |\nabla \tilde \Phi_u|^2 \rho dy \lesssim   \int_{\mathbb R^2} u^2(1+|y|)\rho dy + b^{\frac 54} \| u\|_{\dot H^1_\omega}^2.
\eee
Using again the Hardy inequality \fref{bd:hardyL2rho}, one gets

$$
\int_{\mathbb R^2} U \left|\nabla \left( \frac{u}{U} \right)\right|^2 \rho \leq C \| u\|_{\dot H^1_\omega}^2.
$$

We finally arrive at the subcoercivity estimate for $F$:
$$
F(u,u)=\| \nabla u \|_{L^2_\omega}^2+\Oc\left(\| \nabla u \|_{L^2_\omega} \Big\| \frac{u}{1+|y|^{\frac 32}} \Big\|_{L^2_\omega}+ \Big\| \frac{u}{1+|y|^{\frac 32}} \Big\|_{L^2_\omega}^2+ b^{\frac 58}\| \nabla u \|_{L^2_\omega}^2  \right).
$$
We now turn to the terms in $G$. From \fref{bd:poissonfieldpolynomial1}, \fref{bd:poissonfieldpolynomial2}, \fref{bd:hardyL2rho} and $|U|\lesssim (1+|y|)^{4}$, we get
\begin{align}
\nonumber &\sqrt{\rho}\left|2b U y\cdot \nabla \tilde \Phi_u+ \left( b +\frac{b^2}4|y|^2\right)U\tilde \Phi_u \right|\\
& \qquad  \lesssim \| \nabla u\|_{L^2_\omega}\left(b(1+|y|)^{-\frac 72} +b^{\frac{13}{8}}(1+|y|)^{-\frac{11}{4}}+b^{\frac 58}(1+|y|)^{-\frac{19}{4}}\right). \label{bd:pointwisetildephinonradial}
\end{align}
Using $|U|^{-1}\lesssim (1+|y|)^{4}$ and Cauchy-Schwarz, we obtain for the two first terms below from \fref{bd:generalisedhardy} with $\alpha=3/2$, and for the third with \fref{bd:generalisedhardy} with $\alpha=3/4$:

$$
\int_{\mathbb R^2} b(1+|y|)^{-\frac 72} \frac{u}{U}\sqrt \rho \lesssim b^{\frac 14}\left(b^{\frac 32}\int_{\mathbb R^2} u^2(1+|y|^5)\rho \right)^{\frac 12}\left(\int_{\mathbb R^2}(1+|y|)^{-4}\right)\lesssim b^{\frac 14}\| \nabla u \|_{L^2_\omega},
$$
$$
\int_{\mathbb R^2} b^{\frac{13}{8}}(1+|y|)^{-\frac{11}{4}} \frac{u}{U}\sqrt \rho \lesssim b^{\frac 78}\left(b^{\frac 32}\int_{\mathbb R^2} u^2(1+|y|^5)\rho \right)^{\frac 12}\left(\int_{\mathbb R^2}(1+|y|)^{-\frac 52}\right)\lesssim b^{\frac 78}\| \nabla u \|_{L^2_\omega},
$$
$$
\int_{\mathbb R^2} b^{\frac 58} (1+|y|)^{-\frac{19}{4}}\frac{u}{U}\lesssim b^{\frac 14}\left(b^{\frac 34} \int_{\mathbb R^2} |u|^2(1+|y|)^{\frac 72}\rho dy\right)^{\frac 12} \left(\int_{\mathbb R^2} (1+|y|)^{-5}\right)\lesssim b^{\frac 14}\| \nabla u \|_{L^2_\omega},
$$

from which we obtain the bound

\bee
\int_{\mathbb R^2} \left|2b U y \cdot \nabla \tilde \Phi_u+ \left(b+\frac{b^24}{4}|y|^2\right)U\tilde \Phi_u \right| \frac{u}{U}\rho \lesssim  b^{\frac 14} \| \nabla u \|_{L^2_\omega}^2.
\eee

By using the  estimate \fref{bd:poisson1} with $\alpha=1$ and \fref{bd:hardyL2rho} we get:
$$
\sqrt{\rho}|\tilde \Phi_u|\lesssim (1+|y|)^{\frac 12}\left(1+\mathbbm 1_{|y|\leq 1}|\ln |y||\right) \Big \| \nabla u \Big\|_{L^2_\omega},
$$
and hence from \fref{bd:pointwisetildephinonradial} one gets

\bee
\int_{\mathbb R^2} \left| 2b U y \cdot \nabla \tilde \Phi_u+ \left(b+\frac{b^2}{4}|y|^2\right)U\tilde \Phi_u \right| \tilde \Phi_u \rho \lesssim b^{\frac 58}\| \nabla u \|_{L^2_\omega}^2.
\eee

We then arrive at the estimate for $G$:

\be \label{id:continuityGcoercivity}
|G(u,u)|=\left|  \int_{\mathbb R^2} \left(2b U y \cdot \nabla \tilde \Phi_u+ \left(b+\frac{b^2}{4}|y|^2\right)U\tilde \Phi_u \right)\tilde{\Ms}  u \rho dy\right|\lesssim b^{\frac 14}\| \nabla u \|_{L^2_\omega}^2.
\ee

The estimates for $F$ and $G$ above yield the desired subcoercivity estimate \eqref{bd:subcoercivity}.\\

\noindent \textbf{Step 2} \emph{Asymptotic problem and rigidity}: First note that the third term in \fref{def:bilinearformLtil} is signed, and already satisfies that, from \fref{def:scalarMtil} and \fref{bd:coercivite L2} applied to $u\sqrt \rho$, if $u$ is without radial component, then:
\begin{align*}
\langle u,u\rangle_* & =  \int_{\Rb^2} u \sqrt{\rho} \Ms\Big(u \sqrt{\rho} \Big) dy \\
&\geq  \delta_1\int_{\Rb^2} \frac{(\sqrt \rho u)^2}{U}dy - C\Big[ \la \sqrt \rho u, \Lambda U\ra_{L^2}^2  + \la \sqrt \rho u, \pa_1 U\ra_{L^2}^2 + \la \sqrt \rho u, \pa_2 U\ra_{L^2}^2 \Big]\\
&=\delta_1 \| u \|_{L^2_\omega}^2-C \left( \left(\int_{\Rb^2} u\partial_{y_1}U\sqrt{\rho}dy\right)^2+\left(\int_{\Rb^2} u\partial_{y_2}U\sqrt{\rho}dy\right)^2\right)
\end{align*}
as $\Lambda U$ is radial. Therefore, if one assumes by contradiction that \fref{est:coerLnuybis} does not hold, then
$$
m:= \liminf_{b \rightarrow 0} \inf_{u\in \dot H^1_{\omega[b]}, \ \ \langle u , \sqrt{\rho}\nabla U\rangle_{L^2}=0} \frac{F[b](u,u)+G[b](u,u)}{\| \nabla u \|_{L^2_{\omega[b]}}} \leq 0 \quad \textup{with}\quad  \omega[b] = \frac{\rho[b]}{U}, \quad \rho[b] = e^{-\frac{b|y|^2}{2}}.
$$
From the subcoercivity estimate \fref{bd:subcoercivity} and \fref{bd:hardyL2rho}, we infer that $-\infty<m\leq 0$. Let $b_n\rightarrow 0$ and $u_n$ be sequences such that, without loss of generality, $\| \nabla u_n \|_{L^2_{\omega[b_n]}}=1$, $\langle u_n , \sqrt{\rho}\nabla U\rangle=0$ and 
$$
F[b_n](u_n,u_n)+G[b_n](u_n,u_n)\rightarrow 0.
$$
The above limit, with \fref{bd:subcoercivity} and $\| \nabla u_n \|_{L^2_{\omega[b_n]}}=1$, imply that there exists $c>0$ such that for all $n$:

$$
\int_{\mathbb R^2} u_n^2 (1+|y|)\rho[b_n]dy\geq c.
$$
The sequence $f_n=u_n \sqrt{\rho[b_n]}$ is then uniformly bounded in $\dot H^1((1+|y|^4)dy)$ from \fref{bd:estimationH1H1weigted}, with $
\int_{\mathbb R^2} f_n^2(1+|y|)\geq c$. 
Since also $\int f_n^2(1+|y|^2)$ is uniformly bounded by \fref{bd:hardyL2rho}, there exist $R,c'>0$ such that, up to a subsequence,
$$
\int_{|y|\leq R} |f_n|^2dy\geq c'.
$$
We pass to the limit: there exists $f_{\infty}\in \dot H^1((1+|y|^4)dy)$ that is the weak limit in this space of $f_n$. Moreover, by compactness of $H^1$ in $L^2$ on bounded sets, the convergence is strong in $L^2((1+|y|)dy)$, so that $f_\infty\neq 0$ from the above inequality. Let us write
$$
\sqrt{\rho[b]} \nabla \tilde \Phi_u = \nabla \Phi_{u\sqrt{\rho[b]}}-\frac{ b y}{4}\Phi_{u\sqrt{\rho[b]}}.
$$
From \fref{bd:poisson1}, we infer that the first term, i.e. the mapping $\sqrt{\rho[b]} u\mapsto  \nabla \Phi_{u\sqrt{\rho[b]}}$, is continuous from $L^2(1+|y|)$ into $L^2((1+|y|)^{-4})$. Similarly, the second term is controlled by
$$
\Big\| \frac{by}2 \Phi_{u\sqrt{\rho[b]}} \Big\|_{L^2((1+|y|)^{-4})}\lesssim \sqrt{b} \| u \|_{\dot H^1_{\omega[b]}} \to 0 \quad \textup{as}\;\; b \to 0.
$$
Therefore, $\sqrt{\rho[b_n]} \nabla \tilde \Phi_{u_n}$ converges strongly to $\nabla \Phi_{f_\infty}$ in $L^2((1+|y|)^{-4})$. Consequently, one has the continuity at the limit,

\begin{align*}
& -\int_{\mathbb R^2} u_n^2 \rho[b_n]-2\int_{\mathbb R^2} U \nabla \left( \frac{u_n}{U} \right) \cdot \nabla \tilde \Phi_{u_n} \rho[b_n]+\int_{\mathbb R^2} U | \nabla \tilde \Phi_{u_n} |^2\rho[b_n] \\
& \qquad \quad \underset{n\rightarrow \infty}{\longrightarrow}   -\int_{\mathbb R^2} f_{\infty}^2 -2\int_{\mathbb R^2} U \nabla \left( \frac{f_{\infty}}{U} \right) \cdot \nabla \Phi_{f_{\infty}} +\int_{\mathbb R^2} U \Big| \nabla \Phi_{f_{\infty}} \Big|^2.
\end{align*}

Together with the continuity estimate for $G$ \fref{id:continuityGcoercivity}, which implies its asymptotic vanishing, and lower-semicontinuity, we deduce
\begin{align*}
0&=\lim_{n\rightarrow \infty} F[b_n](u_n,u_n)+G[b_n](u_n,u_n)\\
& \qquad \geq \int_{\mathbb R^2} \frac{|\nabla f_{\infty}|^2}{U} -\int_{\mathbb R^2} f_{\infty}^2 -2\int_{\mathbb R^2} U \nabla \left( \frac{f_{\infty}}{U} \right) \cdot \nabla \Phi_{f_{\infty}} +\int_{\mathbb R^2} U| \nabla \Phi_{f_{\infty}} |^2
\end{align*}
However,
$$
\int_{\mathbb R^2} \frac{|\nabla f_{\infty}|^2}{U} -\int_{\mathbb R^2} f_{\infty}^2 -2\int_{\mathbb R^2} U \nabla \left( \frac{f_{\infty}}{U} \right) \cdot \nabla \Phi_{f_{\infty}} +\int_{\mathbb R^2} U | \nabla \Phi_{f_{\infty}} |^2= \int_{\mathbb R^2} U |\nabla \Ms f_{\infty}|^2.
$$
Hence, as $f_\infty$ is without radial component we deduce that $\Ms f_{\infty}=0$ and hence that $f_{\infty}=c_1\partial_{y_1}U+c_2\partial_{y_1}U$, with one coefficient being non zero since $f_{\infty}\neq 0$. On the other hand, the orthogonality $\langle u_n , \sqrt{\rho}\nabla U\rangle$ passes to the limit, yielding $\langle f_{\infty} , \nabla U\rangle=0$ so that $c_1=c_2=0$ which is a contradiction. This concludes the proof of Proposition \ref{pr:coercivitenonradial}. 
\end{proof}

\appendix
\section{Estimates on the Poisson field}
We first recall estimates relative to the weight $e^{-|z|^2/2}$ with polynomial corrections. First, there holds the bound for any $k\geq 0$ for any function without radial component
\be \label{bd:genpoincare}
\int v^2 |z|^{2k}(1+|z|^{2})e^{-\frac{|z|^2}{2}}dz \lesssim \int |\nabla v|^2 |z|^{2k} e^{-\frac{|z|^2}{2}}dz.
\ee
By a scaling argument, this implies that for $0<b\leq 1$:
\be \label{bd:hardyrho}
\int b^2(|y|^2+|y|^6)|u|^2e^{-\frac{b|y|^2}{2}}\lesssim \int (1+|y|^4)|\nabla u|^2e^{-\frac{b|y|^2}{2}}
\ee
with constant independent on $b$. Therefore:
\be \label{bd:estimationH1H1weigted}
\int (1+|y|^4)|\nabla (u e^{-\frac{b|y|^2}{4}})|^2\leq C \int (1+|y|^4)|\nabla u|^2e^{-\frac{b|y|^2}{2}}.
\ee
Applying \fref{id:hardy} one obtains from the above inequality the Hardy-type inequality with weight $e^{-b|z|^2/2}$:
\be \label{bd:hardyL2rho}
\int (1+|y|^2) u^2 e^{-\frac{b|y|^2}{2}}  \lesssim \int (1+|y|^4)|\nabla u|^2e^{-\frac{b|y|^2}{2}},
\ee
with constant independent on $b$. Interpolating between the above inequality and \fref{bd:hardyrho} we obtain that for any $0\leq \alpha\leq 2$:
\be \label{bd:generalisedhardy}
b^{\alpha} \int_{\mathbb R^2} |u|^2(1+|y|^{2+2\alpha})e^{-\frac{b|y|^2}{2}}dy \lesssim \int |\nabla u|^2(1+|y|^{4})e^{-\frac{b|y|^2}{2}}dy.
\ee
For $u$ localised on a single spherical harmonics $Y^{(k,i)}$ with
$$
Y^{(k,i)}(y)=\left\{\begin{array}{l l} \cos^k\left(\frac{y}{|y|}\right) \quad \mbox{if }i=1, \\ \sin^k\left(\frac{y}{|y|}\right) \quad \mbox{if }i=2,\end{array}\right.
$$
where we identify $y/|y|$ with its angle on the unit circle, the Laplace operator is written as

$$
\Delta u (x)=\Delta^{(k)} (u^{(k,i)})(r) Y^{(k,i)}\left(\frac{y}{|y|}\right), \ \ \Delta^{(k)}= \pa_{rr}+\frac{1}{r}\pa_r-\frac{k^2}{r^2}.
$$

The fundamental solutions to $\Delta^{(k)}f=0$ are $\ln (r)$ and $1$ for $k=0$, and $r^k$ and $r^{-k}$ for $k\geq 1$, with Wronskian relations:
$$
W^{(0)}=\frac{d}{dr}\ln (r)=r^{-1} \ \ \text{and} \ \ W^{(k)}=\frac{d}{dr}(r^k)r^{-k}-r^k \frac{d}{dr}(r^{-k})=2kr^{-1} \ \text{for} \ k\geq1.
$$
The solution to $-\Delta \Phi_u=u$ given by $\Phi_u=-(2\pi)^{-1}\ln (|x|)*u$ is then given on spherical harmonics by:
$$
\Phi_u^{(0,0)}(r)=-\ln (r) \int_0^r u^{(0,0)}(\tilde r)\tilde r d \tilde r -\int_r^{\infty} u^{(0,0)}(\tilde r)\ln(\tilde r)\tilde r d\tilde r ,
$$
\be \label{id:nablaPhi0}
\nabla \Phi_u^{(0,0)}(x)=-\frac{x}{|x|^2} \int_0^{|x|} u^{(0,0)}(\tilde r)\tilde r d \tilde r,
\ee
\be \label{id:Phiki}
\Phi_u^{(k,i)}(r)= \frac{r^k}{2k} \int_r^{\infty} u^{(k,i)}(\tilde r)\tilde r^{1-k} d \tilde r +\frac{r^{-k}}{2k} \int_0^r u^{(k,i)}(\tilde r) \tilde r^{1+k} d\tilde r,
\ee
\be \label{id:nablaPhiki}
\pa_r \Phi_u^{(k,i)}(r)= \frac{r^{k-1}}{2} \int_r^{\infty} u^{(k,i)}(\tilde r)\tilde r^{1-k} d \tilde r -\frac{r^{-k-1}}{2} \int_0^r u^{(k,i)}(\tilde r) \tilde r^{1+k} d\tilde r.
\ee

\begin{lemma}
If $u$ is without radial component, for any $0<\alpha<2$:
\be \label{bd:poisson1}
|\Phi_{u}|^2+|y|^2|\nabla \Phi_{u}|^2\lesssim |y|^2(1+|y|)^{-2\alpha}\left(1+\mathbbm 1_{|y|\leq 1}|\ln |y||\right) \int_{\mathbb R^2}  |u|^2(1+|y|)^{2\alpha }dy.
\ee
\end{lemma}

\begin{proof}

We decompose $\Phi_u$ in spherical harmonics. Note that $\Phi_u^{(0,0)}=0$ as $u$ has no radial component. Applying Cauchy-Schwartz inequality in both terms in \fref{id:nablaPhiki} one gets for $k\geq 1$, as $0<\alpha<2$:
\bee
&& \left| \int_r^{\infty} u^{(k,i)}(\tilde r)\tilde r^{1-k} d \tilde r  \right | \lesssim \left( \int_r^{\infty} |u^{(k,i)}|^2(1+r)^{2\alpha }\tilde r d\tilde r \right)^{\frac 12}\left( \int_{r}^{\infty} (1+r)^{-2\alpha}\tilde r^{1-2k} d\tilde r \right)^{\frac 12}\\
&\lesssim & r^{1-k}(1+r)^{-\alpha} \left(1+\mathbbm 1_{r\leq 1}|\ln r|\right) \left( \int_0^{\infty} |u^{(k,i)}|^2(1+r)^{2\alpha }\tilde r d\tilde r \right)^{\frac 12},
\eee
\bee
\left|\int_0^r u^{(k,i)}(\tilde r) \tilde r^{1+k} d\tilde r \right| &\lesssim & \left( \int_0^{r} |u^{(k,i)}|^2(1+r)^{2\alpha }\tilde r d\tilde r \right)^{\frac 12}\left( \int_{0}^{r} (1+r)^{-2\alpha}\tilde r^{1+2k} d\tilde r \right)^{\frac 12}\\
&\lesssim &r^{1+k}(1+r)^{-\alpha} \left( \int_0^{\infty} |u^{(k,i)}|^2(1+r)^{2\alpha }\tilde r d\tilde r \right)^{\frac 12}.\\
\eee
The two above inequalities, injected in \fref{id:Phiki}, \fref{id:nablaPhiki} produce:
$$
|\Phi_u^{(k,i)}(r)|\lesssim \frac 1 k r^{1}\left(1+\mathbbm 1_{r\leq 1}|\ln r|\right)(1+r)^{-\alpha}\left(  \int_0^{\infty} |u^{(k,i)}|^2(1+r)^{2\alpha }\tilde r d\tilde r\right)^{\frac 12},
$$
$$
|\pa_r \Phi_u^{(k,i)}(r)|\lesssim \left(1+\mathbbm 1_{r\leq 1}|\ln r|\right)(1+r)^{-\alpha}\left(  \int_0^{\infty} |u^{(k,i)}|^2(1+r)^{2\alpha }\tilde r d\tilde r\right)^{\frac 12}.
$$
On each spherical harmonic we thus have:
$$
\left| \Phi_{u}^{(k,i)}Y^{(k,i)} \right|^2+r^2\left|\nabla \left(\Phi_{u}^{(k,i)}Y^{(k,i)} \right) \right|^2 \lesssim r^{2}(1+r)^{-2\alpha }\left(1+\mathbbm 1_{r\leq 1}|\ln r|\right)\int_0^{\infty} |u^{(k,i)}|^2(1+r)^{2\alpha }\tilde r d\tilde r .
$$
The constant in the inequality above is independent on $k,i$, so by summing we obtain \fref{bd:poisson1}.
\end{proof}

\def\cprime{$'$}

\end{document}